\documentclass[11pt]{amsart}
\usepackage[margin = 1in]{geometry}
\usepackage{mathptmx}
\usepackage{latexsym}
\usepackage{amsmath}
\usepackage{amssymb, mathtools}
\usepackage{amsfonts} 
\usepackage{eucal} 
\usepackage{amsbsy}
\usepackage{amsthm}
\usepackage[all]{xy}
\usepackage[pagebackref=true]{hyperref}
\usepackage{graphicx}
\usepackage{xcolor}
\usepackage{enumerate}

\newtheorem{thm}{Theorem}[section]
\newtheorem*{thm*}{Theorem}
\newtheorem{lemma}[thm]{Lemma}
\newtheorem{prop}[thm]{Proposition}
\newtheorem{cor}[thm]{Corollary}
\newtheorem*{cor*}{Corollary}

\theoremstyle{definition}
\newtheorem{defn}[thm]{Definition}

\theoremstyle{remark}
\newtheorem{remark}[thm]{Remark}

\numberwithin{equation}{section}

\newcommand {\real}  {\ensuremath{\mathbb{R}}}
\newcommand {\intg}  {\ensuremath{\mathbb{Z}}}

\newcommand {\cplx}  {\ensuremath{\mathbb{C}}}
\newcommand {\rat}   {\ensuremath{\mathbb{Q}}}

\newcommand {\Hom}   {\ensuremath{\operatorname{Hom}}}

\newcommand {\ch}    {\ensuremath{\operatorname{ch}}}
\newcommand {\ph}    {\ensuremath{\operatorname{ph}}}

\newcommand {\smlhf} {\ensuremath{\mbox{$\frac{1}{2}$}}}
\newcommand {\smlquart} {\ensuremath{\mbox{$\frac{1}{4}$}}}

\newcommand {\Smash} {\ensuremath{\wedge}}

\newcommand {\pro}   {\ensuremath{\operatorname{pr}}}

\newcommand {\syml}  {\ensuremath{\mathbb{L}}}

\newcommand {\loc}   {\ensuremath{\operatorname{loc}}}

\newcommand {\pt}    {\ensuremath{\operatorname{pt}}}
\newcommand {\id}    {\ensuremath{\operatorname{id}}}

\newcommand{\cL} {\ensuremath{\mathcal{L}}}

\newcommand {\MSO}   {\ensuremath{\operatorname{MSO}}}

\newcommand {\MSPL}   {\ensuremath{\operatorname{MSPL}}}

\newcommand {\MSTOP}   {\ensuremath{\operatorname{MSTOP}}}

\newcommand {\SO}   {{\ensuremath{\operatorname{SO}}}}
\newcommand {\SPL}   {{\ensuremath{\operatorname{SPL}}}}

\newcommand {\Witt}   {{\ensuremath{\operatorname{Witt}}}}

\newcommand {\MWITT}   {\ensuremath{\operatorname{MWITT}}}
\newcommand {\sign}   {{\ensuremath{\operatorname{sign}}}}

\newcommand {\const}   {\ensuremath{\operatorname{const}}}
\newcommand {\Th}   {\ensuremath{\operatorname{Th}}}
\newcommand {\topo}   {\ensuremath{{\operatorname{top}}}}
\newcommand {\an}   {\ensuremath{{\operatorname{an}}}}

\newcommand {\geo}    {{\ensuremath{\operatorname{geo}}}}

\newcommand {\KO}   {{\ensuremath{\operatorname{KO}}}}
\newcommand {\K}   {{\ensuremath{\operatorname{K}}}}
\newcommand {\wKO}   {{\ensuremath{\widetilde{\operatorname{KO}}}}}
\newcommand {\wK}   {{\ensuremath{\widetilde{\operatorname{K}}}}}
\newcommand {\ko}   {{\ensuremath{\operatorname{ko}}}}
\newcommand {\ku}   {{\ensuremath{\operatorname{k}}}}
\newcommand {\KK}  {{\ensuremath{\operatorname{KK}}}}

\newcommand {\Per}   {{\ensuremath{\operatorname{Per}}}}
\newcommand {\ism}   {\ensuremath{\intg [\smlhf]}}

\newcommand {\Spin}   {{\ensuremath{\operatorname{Spin}}}}

\newcommand {\ABS}   {{\ensuremath{\operatorname{ABS}}}}
\newcommand {\td}   {{\ensuremath{\operatorname{td}}}}
\newcommand {\La}   {\ensuremath{\mathcal{L}}}

\def\Di{\mathfrak{D}\kern-6.5pt/}
\def\Spi{\mathfrak{S}\kern-6.5pt/}

\newcommand {\pr}  {\ensuremath{\mathbb{P}}}

\usepackage{mathrsfs}
\newcommand{\bb}[1]{\mathbb{#1}}

\newcommand{\cal}[1]{\mathcal{#1}}

\newcommand{\scr}[1]{\mathscr{#1}}

\newcommand\lra{\longrightarrow}
\newcommand\xlra[1]{\xrightarrow{\phantom{x} #1 \phantom{x}}}
\newcommand{\wt}[1]{\widetilde{#1}}
\newcommand\eps\varepsilon
\newcommand\pa\partial
\newcommand{\pullbackcorner}[1][dr]{\save*!/#1-1.2pc/#1:(-1,1)@^{|-}\restore}

\DeclareMathAlphabet{\mathpzc}{OT1}{pzc}{m}{it}
\newcommand{\cl}{\mathpzc{cl}}
\newcommand\fib{\; \operatorname{---}\; } 

\usepackage{xcolor}
\definecolor{darkgreen}{cmyk}{1,0,1,.2}
\definecolor{m}{rgb}{1,0.1,1}

\begin{document}


\title[K-orientations and Gysin maps]
  {Smooth atlas stratified spaces, K-Homology Orientations and Gysin maps. Part 1.}

\author{Pierre Albin}

\address{Department of Mathematics, University of Illinois at Urbana-Champaign, USA}

\email{palbin@illinois.edu}

\author{Markus Banagl}

\address{Institut f\"ur Mathematik, Universit\"at Heidelberg,
  Im Neuenheimer Feld 205, 69120 Heidelberg, Germany}

\email{banagl@mathi.uni-heidelberg.de}

\author{Paolo Piazza}

\address{Dipartimento di Matematica, Sapienza Universit\`a di Roma, Italy}

\email{piazza@mat.uniroma1.it}

\date{May 19, 2025}





\begin{abstract}
We begin this article by introducing {\it smooth atlas stratified spaces}. We show that this class is closed under cartesian products; consequently, it is possible to define fiber bundles of 
smooth atlas stratified spaces. We describe the resolution of such a space to a manifold with fibered corners and use this result in order to prove that the class of smooth atlas stratified spaces coincides with that of Thom-Mather stratified spaces.
We then consider Witt pseudomanifolds (such as singular complex algebraic varieties) where it is well-known that a bordism invariant signature is available and equal to the Fredholm index of a realization of the signature operator.
For an oriented fiber bundle of stratified spaces, with Witt fibers, $p:X\rightarrow Y$ we define a bivariant class $\Sigma (p)\in \KK^\ell(X,Y)[\tfrac12]$, $\ell=\dim X-\dim Y\;\;{\rm mod}\;\;2$. 
Kasparov multiplication on the left by this element $\Sigma (p)$ defines the analytic Gysin map in analytic K-homology $p^!: \K^{{\rm an}}_j (Y) [\tfrac12]\to K^{{\rm an}}_{\ell+j} (X) [\tfrac12]$ and one of our main results is that this Gysin map preserves the analytic signature class of Witt spaces: $p^! (\mathrm{sign}_K (Y))=\mathrm{sign}_K (X)$. We prove in fact a more general result: for three fiber bundles $p_{12}: X_1\to X_2$, $p_{23}: X_2\to X_3$, $p_{13}: X_1\to X_3$ of Witt pseudomanifolds satisfying  $p_{13}=p_{23}\circ p_{12}$,  we establish that  $\Sigma (p_{12})\otimes \Sigma (p_{23})=\Sigma (p_{13})$ in $\KK^* (X_1,X_3)[\tfrac12]$. We also discuss this latter result for other Dirac-type operators satisfying an analytic Witt condition, for example
the spin-Dirac operator on a fibration of psc-Witt spin pseudomanifolds. We next define the analytic Gysin map associated to an oriented normally non-singular inclusion of Witt spaces and prove that it also preserves the signature class. Finally, we relate the analytic signature class of a Witt space $X$, $\mathrm{sign}_K (X)\in \K_*^{{\rm an}} (X)[\tfrac12]$, with the topological Siegel-Sullivan orientation $\Delta (X)\in \KO^{{\rm top}}_* (X)[\tfrac12]$: if $\Psi^2: \KO_* ^{{\rm top}}(X)[\tfrac12] \to \KO^{{\rm top}}_* (X)[\tfrac12]$ denotes the second Adams operation
and $c:  \KO^{{\rm top}}_* (X)\to \K^{{\rm top}}_* (X)$ denotes complexification, then we show that $\mathrm{sign}_K (X)$
corresponds to  $c\circ (\Psi^2)^{-1} \Delta (X)$ under the natural identification between $\K_*^{{\rm an}} (X)[\tfrac12]$
and $\K_*^{{\rm top}} (X)[\tfrac12]$.
\end{abstract}

\maketitle


\tableofcontents


\section{Introduction and main results}
\subsection{Motivation and main results}$\;$

Let $\KO_* (-)$ denote topological $\KO$-homology and
let $M$ be a smooth $n$-dimensional closed oriented manifold.
In \cite{Sul:GTP}, Sullivan introduced a class\footnote{For an abelian group $A$,
we use the notation $A[\tfrac12]$ to indicate $A \otimes_{\bb Z} \bb Z[\tfrac12].$}
$\Delta_\SO (M) \in \KO_n (M)$,
which is an orientation and
plays a fundamental r\^ole in studying the $\K$-theory
of manifolds.
For instance, Sullivan showed that topological block bundles
away from $2$ are characterized as
spherical fibrations together with a $\KO [\smlhf]$-orientation.
For certain classes of singular oriented pseudomanifolds,
Goresky and MacPherson's intersection homology allowed for
the construction of bordism invariant signature invariants that
satisfy Novikov additivity, a product formula, and agree with the
signature of a manifold.
Such a class of pseudomanifolds is given by the Witt spaces
considered by Siegel in \cite{Sie:WSGCTKOP}.
A stratified space is a Witt space if its even-dimensional links have 
vanishing middle degree rational intersection homology groups 
with lower middle perversity, see e.g., Section \ref{sec:WittSpaces} below. 
For example, every
pure-dimensional complex algebraic variety is a Witt space.
Siegel extended Sullivan's orientation $\Delta_\SO$ to an orientation
$\Delta (X) \in \KO_n (X)[\tfrac12]$ for oriented compact Witt spaces
and used it to show that away from $2$, Witt bordism and $\KO$-homology
can be identified.
Under the Pontrjagin character, $\Delta (X)$ is a lift of the Goresky-MacPherson
$L$-class $L_* (X) \in H_* (X;\rat)$.

The second named author proved in \cite{Ban:TSSKSS} that
the Siegel-Sullivan orientation $\Delta (X)$ is preserved under $\KO$-homological
bundle transfer and Gysin restriction homomorphisms\footnote{In this paper, in accordance with parts of the literature and in contrast with others, we will use the terms transfer map and Gysin map interchangeably.} associated to
oriented normally non-singular bundles and inclusions of Witt spaces:
Let $j: X \hookrightarrow Y$ be an oriented normally non-singular 
codimension $\ell$ inclusion of
closed piecewise linear (PL) Witt spaces. 
Since $\SO$-bundles are $\KO [\smlhf]$-oriented, $j$ has an
associated Gysin homomorphism
$j^!: \KO_* (Y)[\tfrac12] \to \KO_{*-\ell} (X)[\tfrac12].$
Under this homomorphism, 
\[ j^! \Delta (Y) = \Delta (X). \]
Similarly, an oriented PL bundle $p: X \to Y$ of  compact Witt spaces
with closed $d$-dimensional PL manifold fiber has an associated bundle transfer
homomorphism 
$p^!: \KO_n (Y)[\tfrac12] \to \KO_{n+d} (X)[\tfrac12],$
under which
\[ p^! \Delta (Y) = \Delta (X). \]
In fact, this holds even when $X$ is only a block bundle over $Y$.
These results are obtained in  \cite{Ban:TSSKSS} by lifting the construction
of Siegel and Sullivan to the ring spectrum level. The lift is a multiplicative morphism
$\Delta:  \MWITT \to \KO [\smlhf]$ 
where $\MWITT$ denotes the ring spectrum representing Witt space bordism.
One then uses the L-theoretic results of 
\cite{Ban:GRTHCCSS} in the immersive case and
\cite{Ban:BTLOCSS} in the submersive case.
As an application, these results lead under the Pontrjagin character
to a computation of
the $L$-classes of singular Schubert varieties \cite{BanWra:UTGCCCSS} and to a 
proof of the Brasselet-Schürmann-Yokura conjecture for such varieties
\cite{BanSchWra:TGCACCSS}.

\bigskip
Let us now move to the main results of this research. The work of Banagl 
takes place solely in topological K-homology and employs heavily methods of stable
homotopy theory. There is, however, a different description of K-homology, due to Gennadi Kasparov; this goes under the name of analytic K-homology and it is denoted
by $\K^\an_* (X)$
with $X$, say,  a compact Hausdorff topological space. By definition, 
$$\K^\an_* (X):= \KK^* (C(X),\mathbb{C})$$
with $\KK$ denoting Kasparov bivariant K-theory. More  generally, for each pair of Hausdorff compact spaces $X$, $Y$ one can define the abelian groups
$\KK^* (X,Y):= \KK^* (C(X),C(Y)).$ 
One important property of Kasparov bivariant K-theory is the existence of an associative
product, the Kasparov product:
$$\KK^j (X,Y)\otimes \KK^\ell (Y,Z)\longrightarrow \KK^{j+\ell}(X,Z).$$
This product will play a fundamental r\^ole in this article.

As in the work of Banagl, we shall be interested in Witt pseudomanifolds, but since we plan to use analytic methods we shall consider 
{\it smoothly stratified}  Witt pseudomanifolds. Here smoothly stratified means stratified in the sense of Thom-Mather but, as explained further below, this is a notion we revisit and expand on in this article. 
Let $X$ be such a space. It is well known, see, e.g., \cite{moscoviciwu}, \cite[\S 6.2]{AlbLeiMazPia:SPWS}, \cite[\S 5.1]{AlbLeiMazPia:NCCS} and \S\ref{sec:KKThy} below,  that there exists a well defined  signature class 
in analytic K-homology $[D^{{\rm sign}}_X]\in \mathrm{KK}^* (C(X),\mathbb{C})$, with $D^{{\rm sign}}_X$ the signature operator on the regular part of $X$ endowed with a wedge metric 
$g$ (the class turns out to be independent of the choice of $g$). After inverting $2$ (and normalizing to better match the Sullivan orientation), this class determines a K-homology orientation which we denote 
\begin{equation*}
	\mathrm{sign}_K(X) := 2^{-\lfloor n/2 \rfloor}[D^{{\rm sign}}_X] \in \mathrm K_n(X)[\tfrac12], \quad n=\dim X.
\end{equation*}
Here is a (partial) list of the results established in this article.

\begin{itemize}
\item If 
$	W - X \xlra{p} Y
$
is an {oriented} fiber bundle of {oriented} Witt pseudomanifolds, 
then there exists an `analytic transfer class'
 $\Sigma (p)\in \KK^j (X,Y)[\tfrac12]$, with $j$ equal to the dimension of $W$ mod 2.
This class is defined by normalizing the class of the vertical family of signature operators associated 
to the fibration $W - X \xlra{p} Y$.
{The bundle $p$ is understood to trivialize locally via
 stratified diffeomorphisms.} 
\item We can define an analytic Gysin map $p^!: \K^\an_* (Y)[\tfrac12] \to \K^\an_{*+j} (X)[\tfrac12]$
by taking the Kasparov product, on the left, with $\Sigma (p)\in \KK^j (X,Y)[\tfrac12]$.
\item We show that the analytic Gysin map preserves the analytic K-homology orientation, viz.:
 \begin{equation}\label{preserves}p^! \mathrm{sign}_K(Y)=\mathrm{sign}_K(X).\end{equation}
 This result extends the work of Banagl from bundles with nonsingular fiber
 to bundles with singular (Witt pseudomanifold) fiber.
 
\item Formula \eqref{preserves} is in fact a special case of a more general result that we also establish in this article: if 
\begin{equation*}
	\xymatrix{
	X_1 \ar[r]^-{p_{12}} \ar[rd]_-{p_{13}} & X_2 \ar[d]^-{p_{23}} \\
	& X_3}
\end{equation*}
is a commutative diagram of Witt fiber bundles, then 
\begin{equation}\label{intro:functoriality}
	[D_{X_1/X_3}^{\mathrm{sign}}] = \ell ([D_{X_1/X_2}^{\mathrm{sign}}] \otimes [D_{X_2/X_3}^{\mathrm{sign}}])
	\text{ in } \mathrm{KK}^*(C(X_1), C(X_3))
\end{equation}
with $$ \ell = \begin{cases} 2 & \text{ if } \dim X_1/X_2 
	\text{ and } \dim X_2/X_3 \text{ are odd, }\\ 1 & \text{ otherwise. } \end{cases} $$
The corresponding analytic transfer classes satisfy
\begin{equation*}
	\Sigma(p_{13})= \Sigma(p_{12}) \otimes \Sigma(p_{23})
	\text{ in } \mathrm{KK}^*(C(X_1), C(X_3))[\tfrac12].
\end{equation*}
The result \eqref{preserves} is obtained by taking $X_3={\rm point}$, $X_1=X$, $X_2=Y$, $p_{12}=p$ and $p_{13}, p_{23}$ equal to the unique maps to a point.
\item We discuss these results in the more general context of  Dirac-type operators,
giving results, for example, on psc-Witt spin stratified pseudomanifolds \cite{BotPiaRos:PSCSPFGSI, BotPiaRos:PSCSCSP}.
\item For a codimension $\ell$ normally non-singular inclusion $j: X\hookrightarrow Y$ of Witt pseudomanifolds we define an element $\Sigma(j)\in\KK^\ell (X,Y)[\smlhf]$ and use Kasparov product on the left by this element 
in order to define an analytic Gysin map $j^!: \K^\an_{*+\ell} (Y)[\smlhf]\to 
\K^\an_{*} (X)[\smlhf]$; also in this case we show that the Gysin map preserves the analytic K-orientations:
\begin{equation}\label{preserve2}j^! \mathrm{sign}_K(Y)=\mathrm{sign}_K(X).\end{equation}
\item We discuss the compatibility between the topological orientation class of an $n$-dimensional  Witt pseudomanifold, $\Delta (X) \in \KO^\topo_n (X)[\smlhf]$,
and the analytic K-orientation $\mathrm{sign}_K(X)\in \K^\an_n (X)[\tfrac12]$. It is a folklore result that these two
classes ``correspond to each other"; we show that the situation is in fact rather subtle, proving that if $c:\KO \to \K$ denotes complexification and 
$\Psi^2: \KO [\tfrac12] \to \KO [\tfrac12]$ the stable second Adams operation, 
then under the natural isomorphism between 
 $\K^\topo_n (X)$ and $\K^\an_n (X)$ the class
\[ c(\Psi^2)^{-1} \Delta (X) \in \K^\topo_n (X)[\smlhf] \]
corresponds to the class 
\[ \mathrm{sign}_K(X)\in \K^\an_n (X)[\smlhf]. \]
\end{itemize}

\medskip
\noindent
All these results are discussed in the second part of the paper. The first part is devoted instead to a foundational 
treatment of smoothly stratified spaces. It was already observed by Verona in his seminal monograph 
\cite{Ver:SMT} that the cartesian product of two Thom-Mather pseudomanifolds, endowed with the product stratification
and the product Thom-Mather control data, is {\em not} a Thom-Mather pseudomanifold
unless one of the two factors is a smooth manifold. We give a simple counterexample
below, in Remark \ref{rk:counterxample-product} . Inspired by the treatment given in \cite{AyaFraTan:LSSS}, but adopting, crucially, a different notion of chart,\footnote{See Remarks \ref{rem:Ayala} and \ref{rem:AyalaEtAl} for the consequence of this difference.}
we introduce the notion of {\it smooth atlas stratified space}; this notion is such that if $X$ and $Y$ 
are smooth atlas stratified spaces, then their product $X\times Y$ 
is again a smooth atlas stratified space. Similarly, with this notion, it is possible to 
give a consistent definition of {\em  fiber bundle
of 
smooth atlas stratified spaces}. We also provide a detailed account of the process of {\it resolution} of a 
smooth atlas stratified space, giving as a final result a smooth manifold with fibered corners, a notion due to Richard Melrose
 \footnote{A smooth manifold with fibered corners is also known as 
  {\it a manifold with corners and an iterated fibration structure.}}. In the context of
Thom-Mather spaces this process is due to Thom and Verona \cite{Ver:SMT}; recent treatments can be found 
in \cite{BraHecSar:TRPVS} and also in \cite{AlbLeiMazPia:SPWS}, where it was proved that the resolution 
has in fact the additional structure of manifold with fibered corners. For a  smooth atlas stratified space we give a self-contained treatment of the resolution process 
and use it in order to show that the class of  {\it smooth atlas stratified spaces} coincides with the class
of {\it Thom-Mather stratified spaces}. Together with known results due to Goresky \cite{Gor:GCHSO} (cf. \cite{Teu:APSB}) and Mather \cite{Mat:NTS} this will establish the remarkable fact that
\begin{enumerate}
\item Thom-Mather stratified spaces,
\item Whitney embedded stratified spaces,
\item manifolds with fibered corners,
\item smooth atlas stratified spaces,
\end{enumerate}
all describe the same singular spaces. With this result we end the first part of the paper and this contribution of ours to the 
general theory of smoothly stratified spaces.

\subsection{Organization of the paper}$\;$\\
{\bf Section \ref{sect:smoothly-stratified}} is entirely devoted to the treatment 
of smooth atlas  stratified spaces. 
In Subsection \ref{subsect:prelude} we give a prelude of the general theory
by treating the case of depth 1 spaces, introducing the relevant notions and techniques in this 
particular case; along the way we illustrate how to resolve a depth 1 stratified space to a manifold with fibered boundary. We also treat briefly the case of depth 2 stratified spaces, mainly in order
to clarify what kind of spaces we obtain along the resolution procedure for spaces of arbitrary depth. In Subsection \ref{subsect:atlas} we give the general definition of smooth atlas stratified space
whereas in Subsection \ref{subsect:resolution} we describe in detail and full generality the process of resolution of such a space to
a smooth manifold with fibered corners. Finally in Subsection \ref{subsec:SASSThomMather} we 
compare the class of smooth atlas stratified spaces with the class of Thom-Mather stratified spaces, showing that they coincide.
From now on we briefly refer to such a space as a {\it smoothly stratified space} or 
a {\it smoothly stratified pseudomanifold} if we do make the additional assumption 
ensuring the pseudomanifold property.
\\ 
In {\bf Section \ref{sec:orientation}} we recall the definition, through the signature operator,  of the analytic orientation class of an orientable smoothly stratified 
Witt pseudomanifold.  We begin in Subsection \ref{subsect:general-orientation}
with a general discussion about orientation classes for a generalized homology theory. In Subsection  
\ref{sec:WedgeMetsDiracOps} we recall general results, due to the first author and Gell-Redman
\cite{AlbGel:IFFDTOP}, about the analysis of (families of) Dirac operators associated to a wedge metric on the regular part of a smoothly stratified pseudomanifold.
In Subsection \ref{sec:KKThy} we recall the fundamentals of unbounded $\KK$-theory and we show/recall that 
the signature operator on an $n$-dimensional  Witt pseudomanifold $X$ does define a class in $\KK^n (C(X),\mathbb{C})=: \K^\an_n (X)$; using this class we define an `analytic orientation class' $\mathrm{sign}_K(X) \in \K^\an_* (X)[\tfrac12].$ \\
In {\bf Section \ref{sect:invariance}} we discuss invariance properties of the analytic orientation class. In Subsection
\ref{subsect:diffeo} we prove the stratified diffeomorphism invariance of $\mathrm{sign}_K(X),$ whereas
in Subsection \ref{subsect:witt-invariance} we extend results in \cite{AlbLeiMazPia:NCCS} and prove a general form of Witt bordism
invariance for $\mathrm{sign}_K(X)$. Finally, in Subsection \ref{subsect:smooth-versus-PL} we compare smooth
Witt bordism with PL-Witt bordism.\\
In {\bf Section \ref{sec:Fibrations}} we tackle the analysis and geometry of a fiber bundle 
$
	W^\ell - X \xlra{p} Y
$
of smoothly stratified
 pseudomanifolds in which the fibers are $\ell$-dimensional oriented Witt pseudomanifolds. 
 The final goal is  to describe how to assign to such a  fiber bundle a class $[D_{X/Y}^{\mathrm{sign}}]$ in the  KK-group $\mathrm{KK}^\ell (C(X), C(Y)).$ To this end 
 we  treat in Subsection \ref{sec:ResStratFibBdles} the {\it grid resolution} of such a fiber bundle; this
 is a fiber bundle  of manifolds with corners with iterated fibration structures fitting into the sequence of manifolds with corners
$$
	\mathrm{res}(W) - \mathrm{res}_{\mathrm{grid}}(X) \lra \mathrm{res}(Y),
$$
with $\mathrm{res}(W)$ and $ \mathrm{res}(Y)$ the resolutions of the base and of the fiber.
Using the grid resolution and the analysis developed in \cite{AlbGel:IFFDTOP} we are able to define the class $[D_{X/Y}^{\mathrm{sign}}]\in  \mathrm{KK}^\ell (C_0(X), C_0(Y))$  in Subsection
\ref{sec:AnTransferMap}; we use the class $\Sigma(p) = 2^{-\lfloor \dim W \rfloor}[D_{X/Y}^{\mathrm{sign}}] \in \mathrm{KK}^\ell (C_0(X), C_0(Y))[\tfrac12]$ and Kasparov product in order to define our analytic Gysin map $p^!: \K^\an_* (Y)[\tfrac12]\to \K^\an_* (X)[\tfrac12]$.
In Subsection \ref{subsect:functoriality-for-fibrations} we prove one of the main results of this article:
the functoriality formula for orientation classes (this is formula \eqref{intro:functoriality} in this Introduction, a result
that implies as a special case formula \eqref{preserves}).\\
In {\bf Section \ref{sec:Inclusions}} we define the analytic Gysin map associated to a normally non-singular inclusion of two
oriented Witt pseudomanifolds and we prove that this Gysin map preserves the analytic orientation class.\\
Finally, in {\bf Section \ref{sec.compatanalytictoporient}} we compare the Siegel-Sullivan class $\Delta (X)$ with
the class $\mathrm{sign}_K(X)$, proving that $c(\Psi^2)^{-1} \Delta (X) \in \K^\topo_n (X)[\smlhf]$
corresponds to the class $\mathrm{sign}_K(X)\in \K^\an_n (X)[\smlhf]$ under the natural isomorphism between
topological and analytic K-homology.

\subsection{Further remarks}

\begin{remark}\label{rk:smooth-versus-witt}
Having defined $\Sigma (i)$ for oriented normally non-singular inclusions and $\Sigma (p)$ for oriented fiber bundles with Witt fibers, it would be natural to define the Gysin map
$f^!$ for any oriented map between smoothly stratified spaces $f: X\lra Y$ which admits a factorization
\begin{equation*}
	X \xhookrightarrow{\phantom{xx}i\phantom{xx}}  Z \xlra p Y
\end{equation*}
into such an $i$ and $p$ as Kasparov multiplication by $\Sigma(f)= \Sigma(i) \otimes \Sigma(p).$ We do not take this step because we have not established that this produces a well-defined element, i.e., that it is independent of the factorization.

For K-oriented maps between smooth manifolds, treated by Connes and Skandalis \cite{ConSka:TLPF}, every map has such a factorization by taking $Z=X\times Y,$ $i$ the inclusion of $X$ into the graph of $f,$ and $p$ the obvious projection. In this case one can show that $\Sigma(f) = \Sigma(i) \otimes \Sigma(p)$ is well-defined by 
starting with two factorizations
\begin{equation*}
	\xymatrix{
	& Z \ar[dr]^-p & \\
	X \ar[ur]^-i \ar[dr]_-{i'} & & Y\\
	& Z' \ar[ur]_-{p'} &
	}
\end{equation*}
then filling in the diamond
\begin{equation*}
	\xymatrix{
	& Z \ar[dr]^-p & \\
	X \ar[ur]^-i \ar[dr]_-{i'} \ar[r]^-J & Z\times_Y Z' \ar[u] \ar[d] \ar[r]^-Q & Y\\
	& Z' \ar[ur]_-{p'} &
	}
\end{equation*}
and then showing that $\Sigma(i) \otimes \Sigma(p)$ (and by symmetry 
$\Sigma(i') \otimes \Sigma(p')$) is equal to $\Sigma(J) \otimes \Sigma(Q)$ (see \cite[Proposition 4.9]{ConSka:TLPF} for details).
The problem with carrying this out for maps between stratified spaces is that it is easy for $i$ and $i'$ to be normally non-singular but for $J$ not to be normally non-singular. (The difficulty boils down to the fact that if $X$ is a singular space then the inclusion of $X$ as the diagonal in $X^2$ is not normally non-singular.)
\end{remark}

\begin{remark}
Our approach, of defining Gysin maps in K-homology by assigning KK-classes to oriented maps between spaces, follows Connes \cite{Con:SFOA} and Connes-Skandalis \cite{ConSka:TLPF, ConSka:LITF} who worked out the case of smooth manifolds and K-oriented maps. Later Hilsum \cite{Hil:FKBPVL} worked in the setting of Lipschitz manifolds and oriented maps and made use of the signature operator, as we do. In \cite{HilSka:MKDFFTKDCDC} Hilsum and Skandalis extended the approach of Connes-Skandalis to K-oriented morphisms of spaces of leaves of foliated smooth manifolds (Remark 7.1 in {\em loc cit} has a nice discussion of the relation between using K-oriented maps versus oriented maps).

As the reader will see, our treatment is analytically quite intricate and 
one might very well wonder why we can not just adapt the arguments of \cite{ConSka:LITF}
to our singular case. In the latter, KK-classes are constructed by quantizing normalizations of symbols of Dirac operators into bounded  operators which are Fredholm because the symbols are elliptic.
When we pass to stratified pseudomanifolds we need to use the 
edge calculus of Mazzeo \cite{Maz:ETDEO} (for spaces of depth 1) and, more generally,
 the edge calculus of 
\cite{AlbGel:IFFDTOP} (for spaces of arbitrary depth). These calculi involve the symbol of
the operator we are studying but also the so-called normal operators associated to the
boundary hypersurfaces of the resolution of our stratified space.  The analysis of the normal operators adds complexity to all the arguments having to do with Kasparov theory (for example, elliptic operators need not be Fredholm and a given elliptic symbol can have more than one Fredholm quantization) and prevents us from easily carrying out the arguments of \cite{ConSka:LITF}.

There are other technical difficulties stemming from the fact that we wish to carry out analytic arguments suited to smooth spaces but apply these arguments to singular spaces. This leads, for example, to the discussion of various levels of resolution of a fiber bundle of smoothly stratified spaces in section \ref{sec:ResStratFibBdles} below.
\end{remark}

\begin{remark}
Why work with the signature operator instead of spin-c Dirac operators?
One reason is that if a stratified space is Witt then we know that the signature operator will be Fredholm (see, e.g., \cite{Che:SGSRS, AlbLeiMazPia:SPWS}). In order for the spin-c Dirac operator associated to a wedge metric to induce Fredholm maps between the natural Sobolev spaces it is necessary to (choose an appropriate domain and) guarantee that certain model operators, the vertical spin-c Dirac operators on the link bundles, be invertible (see, e.g., \cite{AlbGel:IDOIES, Cho:DOSWCSPSC}). This depends on more than the topology of the stratified space, or the choice of spin-c structure, but depends even on the particular wedge metric being used. 
This difference makes the signature operator a superior choice.
\end{remark}

\subsection{Acknowledgements}
The authors thank the University of Heidelberg, Sapienza Universit\'a di Roma, Stanford University, and Washington University in St. Louis for hosting research visits and are happy to acknowledge interesting conversations with Daniel Grieser, Jens Kaad, Markus Land, Eric Leichtnam, Rafe Mazzeo, Richard Melrose,  Jonathan Rosenberg, J\"org Schurmann, Walter van Suijlekom, Lukas Waas, and Jon Woolf.
M. B. is funded by a research grant of the
 Deutsche Forschungsgemeinschaft (DFG, German Research Foundation)
 -- Projektnummer 495696766. This research was also partially funded by INdAM,
 Istituto Nazionale di Alta Matematica.

\subsection{Notation}
Various K-theoretic groups will play a r\^ole in the present paper and
need to be distinguished notationally.
The symbols
$\K^\topo_*$, $\K_\topo^*$, $\KO^\topo_*$ and $\KO_\topo^*$
denote topological complex K-homology,
topological complex K-theory, 
topological real K-homology and topological real K-theory.
These are represented by ring spectra $\K$ and $\KO$.
The connective versions of these spectra
will be denoted by $\ku$ and $\ko$.
The geometric complex K-homology
of Baum and Douglas \cite{BauDou:HIT} will be written as
$\K^\geo_*$.
The Kasparov K-groups of a pair $(A, B)$ of $C^*$-algebras
are denoted by $\KK^* (A, B)$.
Analytic complex K-homology groups of a (locally compact Hausdorff) 
topological space $X$ are given by the Kasparov groups
$\K^\an_* (X) = \KK^* (C_0 (X), \cplx)$.

\section{Smoothly stratified spaces}\label{sect:smoothly-stratified}

We will describe the classes of spaces and operators that we will be working with.
One notational convention that we hope will make it easy to keep track of which statements
refer to singular spaces and which to smooth spaces is that we will use $X$ to denote a stratified space,
$M$ to denote a smooth manifold with corners, and 
\begin{equation*}
	W \fib X \lra Y \quad \text{ and } \quad L \fib M \lra N,
\end{equation*}
to denote fiber bundles of these spaces. 

\subsection{Prelude: Spaces of depth one and two}\label{subsect:prelude}

This section can be skipped. As the constructions below can be rather intricate, we thought it might be useful to the reader to have a discussion of the simpler case of depth one spaces. We do not reference the bibliography in this section and refer to the main text for attributions of ideas.

By a depth one smoothly stratified space we mean first of all a Hausdorff, locally compact topological space $X$ with a countable basis for its topology which can be decomposed as the union of two smooth manifolds. One of them is called the singular part and will be denoted $X_{\mathrm{sing}}$ or $X_0$ and the other is called the regular part and will be denoted $X_{\mathrm{reg}}$ or $X_1.$ The regular part is required to be open and dense. For simplicity we will assume that $X,$ $X_0,$ and $X_1$ are connected.

(More generally below we will study stratified spaces with many strata. In that case we label them with elements of a partially ordered set (poset) with the important requirement that if two strata have intersecting closures then their labels are comparable.)

The usual extra structure to require is known as Thom-Mather data, making $X$ a {\bf Thom-Mather stratified space}. In this case, this refers to a `tubular' neighborhood $T_0$ of $X_0$ in $X$ and a pair of continuous functions
\begin{equation*}
	\pi_0: T_0 \lra X_0, \quad
	\rho_0: T_0 \lra [0,\infty)
\end{equation*}
such that $\pi_0|_{X_0} = \mathrm{id}$ and $\rho_0^{-1}(0) = X_0$ and with the property that
\begin{equation*}
	(\pi_0, \rho_0)|_{T_0 \cap X_1}: T_0\cap X_1 \lra X_0 \times (0,\infty)
\end{equation*}
is a smooth submersion.

Suppose $X$ and $X'$ are both Thom-Mather stratified spaces of depth one, with Thom-Mather data $(T_0, \pi_0, \rho_0)$ for $X$ and $(T_0', \pi_0', \rho_0')$ for $X'.$
The natural notion of morphism between them is a {\bf controlled map} meaning a continuous map
\begin{equation*}
	f: X \lra X'
\end{equation*}
that restricts to smooth maps $X_0 \lra X_0'$ and $X_1 \lra X_1'$ and moreover (after possibly shrinking the size of the tubular neighborhoods) satisfies
\begin{equation*}
	f(T_0) \subseteq T_0', \quad
	f \circ \pi_0 = \pi_0' \circ f, \quad
	\rho_0 = \rho_0' \circ f.
\end{equation*}

Given a Thom-Mather stratified space, one can resolve it to a smooth manifold with boundary (or more generally corners) as follows.
One can identify the neighborhood $T_0$ with the total space of a fiber bundle 
\begin{equation*}
	C(Z_0) \fib T_0 \lra X_0
\end{equation*}
with fiber the cone over a smooth manifold $Z_0,$ 
\begin{equation*}
	C(Z_0) = [0,\infty) \times Z_0 \diagup \{0\} \times Z_0 = \{*\} \sqcup (0,\infty) \times Z_0.
\end{equation*}
The manifold $Z_0$ is known as the link of $X_0$ in $X.$
If we remove $T_0$ from $X$ we obtain a smooth manifold with boundary
\begin{equation*}
	M = \mathrm{res}(X) := X \setminus T_0
\end{equation*}
known as the resolution of $X.$

This manifold $M$ is an example of a {\bf manifold with fibered corners}, though in this simple case $M$ is a manifold with fibered boundary and no actual corners. The extra structure here, which is inherited from the fiber bundle $T_0 \lra X_0,$ is a fiber bundle on the boundary of $M,$
\begin{equation*}
	Z_0 \fib \pa M \xlra{\phi_0} X_0.
\end{equation*}
In order to be more systematic, and hopefully less confusing, with our notation we will denote the part of the boundary of $M$ that sits over $X_0$ by $\pa_0 M$ (in this case this is all of the boundary, $\pa M = \pa_0M,$ but when there are more strata that will not be the case), then we will indicate the base of the fiber bundle by $B_0M$ and the fiber by $F_0M,$ thus
\begin{equation*}
	F_0M \fib \pa_0 M \xlra{\phi_0} B_0 M, \quad \text{ with } \quad
	F_0M = Z_0, \quad 
	\pa_0M = \pa M, \quad 
	B_0M = X_0.
\end{equation*}
This notation is particularly useful when we study manifolds with fibered corners directly, as opposed to starting with a stratified space and resolving it.

Every manifold with fibered boundary is the resolution of a Thom-Mather stratified space of depth one. We have only to collapse the fibers of the fiber bundle $\phi_0$ in order to recover $X$ from $M.$ We obtain Thom-Mather data from collar neighborhoods of the boundary of $M.$

If $M$ and $M'$ are two manifolds with fibered boundary then the natural notion of morphism between them is a {\bf fibered boundary map} (more generally a fibered corners map) by which we mean a smooth map $f: M \lra M'$ with the property that its restriction to $\pa_0M$ participates in a commutative diagram
\begin{equation*}
	\xymatrix{
	\pa_0M \ar[r]^-f \ar[d]_- {\phi_0} & \pa_0M' \ar[d]^-{\phi_0'} \\
	B_0M \ar[r]^-{\bar f} & B_0M' }
\end{equation*}
for some smooth map $\bar f: B_0M \lra B_0M'.$
Every controlled map between Thom-Mather stratified spaces $X \lra X'$ induces a fibered boundary map between their resolutions, $\mathrm{res}(X) \lra \mathrm{res}(X').$ The converse statement is trickier to work out but it is easy to see that a fibered boundary {\em diffeomorphism} is equivalent to a controlled isomorphism for some choice of Thom-Mather data.

These two ways of studying stratified spaces, as Thom-Mather stratified spaces or manifolds with fibered corners, are well known. In this paper we introduce another way of studying the same class of spaces by making use of an appropriate atlas.\\

Let us go back to just having a Hausdorff, locally compact topological space $X$ with a countable basis for its topology and a decomposition into two smooth manifolds $X_0$ and $X_1.$ (Though now we would be happy to just assume that this is a decomposition into two subsets since it will follow from the atlas description that the subsets are smooth manifolds.) We continue to assume that $X_1$ is open and dense and that $X,$ $X_0,$ and $X_1$ are connected.

By a stratified chart around a point $\zeta \in X_0$ we will mean a homeomorphism
\begin{equation*}
	\varphi: \cal U \lra \bb R^h \times C(L)
\end{equation*}
sending $\zeta$ to $\{0\}\times \{*\},$
where the domain $\cal U$ is a neighborhood of $\zeta$ in $X,$ $h \in \bb N_0,$ $L$ is a smooth manifold which {\em a priori} could depend on $\zeta,$ and we require that
\begin{equation*}
	\varphi(\cal U \cap X_0) \subseteq \bb R^h \times \{*\} 
	\quad \text{ and } \quad
	\varphi(\cal U \cap X_1) \subseteq \bb R^h \times (C(L) \setminus \{*\}).
\end{equation*}
By a stratified chart around a point $\zeta \in X_1$ we will mean a homeomorphism
\begin{equation*}
	\varphi: \cal U \lra \bb R^n
\end{equation*}
sending $\zeta$ to $\{0\},$
where the domain $\cal U$ is a neighborhood of $\zeta$ contained in $X_1$ and $n \in \bb N_0.$

To discuss compatibility of charts, we need to settle on when a continuous map
\begin{equation*}
	\psi: \bb R^h \times C(L) \lra \bb R^{h'} \times C(L')
\end{equation*}
should be considered a smooth stratified map. For us this will mean that there is a smooth map (between smooth manifolds)
\begin{equation*}
	\wt\psi: \bb R^h \times [0,\infty) \times L \lra \bb R^{h'} \times [0,\infty) \times L'
\end{equation*}
participating in a commutative diagram
\begin{equation*}
	\xymatrix{
	\bb R^h \times [0,\infty) \times L \ar[r]^-{\wt\psi} \ar[d] & \bb R^{h'} \times [0,\infty) \times L'	\ar[d] \\
	\bb R^h \times C(L) \ar[r]^-{\psi} & \bb R^{h'} \times C(L') }
\end{equation*}
where the vertical arrows are the collapse maps defining the cones.
Having settled on this definition, it is clear what we mean by compatibility of charts.
Indeed the only interesting case is when we look at two charts $\varphi$ and $\varphi'$ centered around points in $X_0,$
\begin{equation*}
	\varphi: \cal U \lra \bb R^h \times C(L), \quad
	\varphi': \cal U' \lra \bb R^{h'} \times C(L')
\end{equation*}
and in this case we require that either $\cal U \cap \cal U'=\emptyset$ or that 
\begin{equation*}
	\varphi' \circ \varphi^{-1}: \varphi(\cal U\cap \cal U') \lra \varphi'(\cal U \cap \cal U')
\end{equation*}
is a smooth map (in the sense just described) whose inverse is also a smooth map.

By a stratified atlas for $X$ we mean a collection $\cal A$ of stratified charts whose domains form a basis for the topology of $X.$
A stratified space endowed with a stratified atlas is a {\bf smooth atlas stratified space}.

A {\bf smooth stratified map} between two smooth atlas stratified spaces $(X, \cal A)$ and $(X',\cal A')$ (of depth one) is a continuous map $f: X \lra X'$ satisfying $f(X_0) \subseteq X_0'$ and $f(X_1) \subseteq X_1'$ and such that for every stratified chart $(\cal U, \varphi)$ on $X$ and $(\cal U', \varphi')$ on $X',$ compatible with the respective atlases, 
satisfying $f(\cal U) \subseteq \cal U',$ the composition 
\begin{equation*}
	\varphi' f \varphi^{-1}: \bb R^{h} \times C(Z) \lra \bb R^{h'} \times C(Z')
\end{equation*}
is a smooth stratified map.

We explain below how to resolve a smooth atlas stratified space to a manifold with fibered boundary. The key player is obtained from the space of paths that start at $X_0$ and immediately move on to $X_1,$
\begin{equation*}
	\scr P_0 = \{\chi: \bb R^+ \lra X \text{ smooth stratified map}: \chi(t) \in X_0 \iff t=0\}
\end{equation*}
(where we are considering $\bb R^+$ as a stratified space of depth one with $(\bb R^+)_0 = \{0\}$ and $(\bb R^+)_1 = (0,\infty)$).
By definition of smooth stratified maps, whenever $(\cal U \xlra{\varphi} \bb R^h \times C(Z))$ is a stratified chart centered at at point in $X_0$ and $\chi \in \scr P_a,$
$\varphi\circ \chi$ has a lift
\begin{equation*}
	\xymatrix{
	& \bb R^h \times [0,\infty) \times Z \ar[d] \\
	\bb R^+ \ar[r]^-{\varphi\circ \chi} \ar[ru]^-{\wt{\varphi\circ\chi}} & \bb R^h \times C(Z). }
\end{equation*}
Let us define an equivalence relation on $\scr P_0$ by
\begin{equation*}
	\chi_1 \sim \chi_2 \iff
	\begin{cases}
	\chi_1(0) = \chi_2(0) & \text{ and }\\ $ $\\
	\wt{\varphi\circ\chi_1}(0) = \wt{\varphi\circ\chi_2}(0) & \text{ for all charts in $\cal A$ including $\chi_1(0)$}
	\end{cases} 
\end{equation*}
and denote the set of equivalence classes by $\cal S_0,$
\begin{equation*}
	\cal S_0 = \scr P_0 \diagup \sim.
\end{equation*}
(If $X$ is a smooth manifold that we happen to be considering as a stratified space then $\cal S_0$ is the spherical normal bundle of $X_0$ in $X.$
If $X$ is the cone over a space $Z,$ $X = C(Z),$ then $\cal S_0$ is $Z.$)
There is an obvious map 
\begin{equation*}
	\phi_0: \cal S_0 \lra X_0, \quad [\chi]\mapsto \chi(0)
\end{equation*}
and we will see that the stratified charts on $X$ can be used to structure $\phi_0$ as a smooth fiber bundle of smooth manifolds.

To see this, start with any stratified coordinate chart centered at a point in $X_0,$ $\varphi: \cal U \lra \bb R^h \times C(Z),$ and denote
\begin{equation*}
	\cal U_0 = \cal U \cap X_0, \quad
	\varphi_0 = \varphi|_{\cal U_0}: \cal U_0 \lra \bb R^h
\end{equation*}
and then  
\begin{equation*}
	\xymatrix @R=1pt {
	\phi_0^{-1}(\cal U_0) \ar[r]^-{\widetilde\varphi_0} & \bb R^h \times Z,\\
	[\chi] \ar@{|->}[r] & \wt{\varphi\circ \chi}(0) }
	\text{ and }
	\varphi_{\cal S_0} = (\varphi_0^{-1}\times \id)\circ \widetilde \varphi_0: \phi_0^{-1}(\cal U_0) \lra \cal U_0 \times Z.
\end{equation*}
The idea is to declare that $\varphi_{S_0}$ is a diffeomorphism and we just need to check that this is unambiguous. 
But if we started with another coordinate chart $\varphi': \cal U' \lra \bb R^h \times C(Z),$ with $\cal U'_0 = \cal U_0,$ then on $\cal U \cap \cal U'$ these charts differ by a stratified diffeomorphism $\varphi'=f\circ \varphi$ and, by definition of smooth stratified map, $f$ participates in a commutative diagram
\begin{equation*}
	\xymatrix{
	\bb R^h \times \bb R \times Z \ar@{}[r]|-*[@]{\supseteq}  & \cal V \times \bb R \times Z \ar[r]^-{\wt f} 
		& \cal V' \times \bb R \times Z \ar@{}[r]|-*[@]{\subseteq} & \bb R^h \times \bb R \times Z  \\
	\bb R^h \times \bb R^+ \times Z \ar@{}[r]|-*[@]{\supseteq}  & \cal V \times \bb R^+ \times Z \ar@{^(->}[u] \ar[d] \ar[r]^{\wt f_{\geq 0}}
		& \cal V' \times \bb R^+ \times Z \ar@{}[r]|-*[@]{\subseteq} \ar@{^(->}[u] \ar[d]  & \bb R^h \times \bb R^+ \times Z  \\
	\bb R^h \times C(Z) \ar@{}[r]|-*[@]{\supseteq}  & \cal V \times C(Z) \ar[r]^-{f} 
		& \cal V' \times C(Z) \ar@{}[r]|-*[@]{\subseteq} & \bb R^h \times C(Z) }
\end{equation*}
and $\wt f$ restricts to a stratified diffeomorphism  $\cal V \times \{0\} \times Z \lra \cal V' \times \{0\} \times Z$ of the form $(y,0,z) \mapsto (\wt f'(y),0, \wt f''(y,z)).$
It follows that $\varphi_{\cal S_0}': \phi_0^{-1}(\cal U_0) \lra \cal U_0 \times Z$ is given by
\begin{equation*}
	\phi_0^{-1}(\cal U_0) \xlra{\varphi_{\cal S_0}} \cal U_0 \times Z \xlra{\wt f'''} \cal U_0 \times Z, \text{ where }
	\wt f'''(u,z) = (u, \wt f''(\varphi_0(u),z))
\end{equation*}
and hence differs from $\varphi_{\cal S_0}$ by a stratified diffeomorphism. Thus $\cal S_0$ inherits a smooth manifold structure.
Moreover, over each $\cal U_0$ we have a commutative diagram
\begin{equation*}
	\xymatrix{ \phi_0^{-1}(\cal U_0) \ar[rr]^-{\varphi_{\cal S_0}} \ar[rd]_-{\phi_0} & & \cal U_0 \times Z \ar[ld] \\
	& \cal U_0 & }
\end{equation*}
with the arrow on the right given by the projection onto the first factor, so $\phi_0: \cal S_0 \lra X_0$ is a smooth fiber bundle with fiber $Z.$

Having defined $\cal S_0$ we define the resolution of $X$ to be
\begin{equation*}
	\mathrm{res}(X) = (X \setminus X_0) \sqcup \cal S_0.
\end{equation*}
(We also denote this by $[X;X_0]$ and view it as the `radial blow-up' of $X$ along $X_0.$)
It is easy to see that the stratified charts on $X$ induce charts on $\mathrm{res}(X)$ making it a smooth manifold with boundary.
Finally the boundary of $\mathrm{res}(X)$ is $\cal S_0$ and so participates in the fiber bundle $\phi_0.$

(If $X$ is a smooth manifold that we happen to be considering as a stratified space then, for example we could pick a Riemannian metric on $X$ and identify
\begin{equation*}
	\mathrm{res}(X) = \{x\in X: d(x, X_0) \geq \eps\}
\end{equation*}
for some $\eps,$ so that $\pa\mathrm{res}(X)$ can be identified with the spherical normal bundle of $X_0$ in $X$ and $\phi_0$ is the projection map of this bundle. If $X$ is the cone over a space $Z,$ $X = C(Z),$ then $\mathrm{res}(X) = \bb R^+ \times Z$ and $\phi_0: \{0\}\times Z \lra Z$ is the identity map.)

In this way we have gone from smooth atlas stratified spaces of depth one to manifolds with fibered boundary. This construction is reversible and we may obtain from any manifold with fibered boundary a smooth atlas stratified space structure on the stratified space obtained by collapsing the boundary fiber bundle. As we have noted above, Thom-Mather stratified spaces are also in one-to-one correspondence with manifolds with fibered boundary, so all three of these spaces represent the same geometric objects. We may add one more description to the list as it is well-known, thanks to the work of Mather \cite{Mat:NTS} and Goresky \cite{Gor:GCHSO} (cf. Teufel \cite{Teu:APSB}), that Whitney stratified subspaces of Euclidean space are equivalent to Thom-Mather stratified spaces.

Moreover, let us point out that the morphisms between smooth atlas stratified spaces are themselves in one-to-one correspondence with those between manifolds with fibered boundary. Indeed, if $X$ and $X'$ are smooth atlas stratified spaces with resolutions $\mathrm{res}(X)$ and $\mathrm{res}(X')$ then every smooth stratified map $f: X \lra X'$ which satisfies that $f(X_0) \subseteq X_0'$ and $f(X_1) \subseteq X_1'$ is associated to a unique fibered boundary map $\wt f: \mathrm{res}(X) \lra \mathrm{res}(X')$ participating in a commutative diagram
\begin{equation*}
	\xymatrix{
	\mathrm{res}(X) \ar[r]^-{\wt f} \ar[d] & \mathrm{res}(X') \ar[d]\\
	X \ar[r]^-f & X' }
\end{equation*}
where the vertical maps are obtained by collapsing the boundary fiber bundles. Moreover every fibered boundary map $\mathrm{res}(X) \lra \mathrm{res}(X')$ arises in this way.\\

There is one more type of space that will be used below: `hybrid' spaces that arise as intermediate spaces in the resolution process. Let us discuss the resolution of a very simple space of depth two to explain how these come about. 
Consider the `toy case'\footnote{This is not technically an example of a smoothly stratified space as we have defined them, since the link $W \times C_{[0,1)}(Z)$ is not compact, but we will use it as it is easy to discuss. If the reader prefers an `honest' example to work through, replacing the truncated cones with suspensions produces $Y \times S(W \times S(Z))$ which is almost as simple.}
 in which $X$ is given by
\begin{equation*}
	X = Y \times C_{[0,1)}(W \times C_{[0,1)}(Z))
\end{equation*}
where $Y,$ $W,$ and $Z$ are compact smooth manifolds and $C_{[0,1)}$ denotes the truncated cone. The strata of $X$ are
\begin{equation*}
	X_0 = Y \times \{*\}, \quad
	X_1 = Y \times (0,1)_x \times W \times \{*\}, \quad
	X_2 = Y \times (0,1)_x \times W \times (0,1)_r \times Z
\end{equation*}
and satisfy
\begin{equation*}
	\overline {X_1} = X_0 \cup X_1 
	\quad \text{ and } \quad 
	X = \overline {X_2} = \overline{X_1}\cup X_2.
\end{equation*}
A stratified chart around a point in $X_0$ is a homeomorphism
\begin{equation*}
	\varphi:\cal U \lra \bb R^{h_0} \times C(W \times C_{[0,1)}(Z))
\end{equation*}
where $h_0$ is the dimension of $Y.$ The restriction of $\varphi$ to $\cal U \cap X_0$ corresponds to a coordinate chart on $Y.$
For a point in $X_1$ we can can choose a neighborhood $\cal U'\subseteq X \setminus X_0$ and a homeomorphism
\begin{equation*}
	\varphi': \cal U' \lra \bb R^{h_1} \times C(Z)
\end{equation*}
where $h_1$ is the dimension of $X_1.$ The restriction of $\varphi'$ to $\cal U' \cap X_1$ corresponds to a coordinate chart on $Y \times (0,1)_x \times W.$
Finally points in $X_2$ have charts that are homeomorphisms of the form
\begin{equation*}
	\varphi'': \cal U'' \lra \bb R^{h_2}
\end{equation*}
where $h_2$ is the dimension of $X_2.$

The resolution of $X$ is obtained by first blowing-up $X_0$ and then blowing-up (the lift of) $X_1.$ The first step produces the hybrid space
\begin{equation*}
	\wt X = [X;X_0] = Y \times [0,1)_x \times W \times C_{[0,1)}(Z)
\end{equation*}
where hybrid refers to the fact that this space presents a boundary as well as singularities\footnote{A similar notion is used by Verona \cite[\S 5]{Ver:SMT} as an intermediate step in showing that Thom-Mather stratified spaces have resolutions to smooth manifolds with corners. Verona refers to the intermediate spaces as `abstract stratifications with faces'.}
This space has a natural blow-down map 
\begin{equation*}
	\wt\beta:\wt X \lra X
\end{equation*}
which is the identity map between $Y \times (0,1)_x \times W \times C_{[0,1)}(Z)$ and sends $Y \times \{0\} \times W \times C_{[0,1)}(Z)$ to $Y \times \{*\}$ in the natural way. The hybrid stratification of $\wt X$ has three strata, which correspond to the three strata of $X$ by $\wt \beta,$ namely
\begin{equation*}
	\wt X_0 = Y \times \{0\} \times W \times C_{[0,1)}(Z), \quad
	\wt X_1 = Y \times (0,1)_x \times W \times \{*\}, \quad
	\wt X_2 = Y \times (0,1)_x \times W \times (0,1)_r \times Z.
\end{equation*}
The restriction of $\wt\beta$ to $\wt X_0$ is the trivial fiber bundle
\begin{equation*}
	\wt \phi_0: Y \times \{0\} \times W \times C_{[0,1)}(Z) \lra X_0 = Y \times \{*\}
\end{equation*}
with fiber $W \times C_{[0,1)}(Z).$ 
Points in $\wt X_0$ have `res-charts' in $\wt X$ (where {\em res} refers to {\it resolved}) of the form
\begin{equation*}
	\varphi: \cal U := \cal V \times [0,\eps) \times W \times C_{[0,1)}(Z) \lra \cal V \times \bb R^+ \times  W \times C_{[0,1)}(Z)
\end{equation*}
where $\cal V$ is an open subset of $Y.$ 
(What makes this a res-chart is that the restriction of $\varphi$ to $\cal U \cap \wt X_0$ is a trivialization of $\wt\phi_0,$ its restriction to $\cal U \cap \wt X_1$ has image in $\cal V \times \bb R^+ \times W \times \{*\},$ and its restriction to $\cal U \cap \wt X_2$ has image in $\cal V \times \bb R^+ \times W \times (0,1)_r \times Z.$ The point in the latter two cases being that the image is in $\cal V \times \bb R^+$ times a stratum of the fiber of $\wt\phi_0.$) Points in $\wt X \setminus \wt X_0$ have `sing-charts' in $\wt X$ which are just lifts (along $\wt \beta$) of stratified charts on $X.$

The resolution of $X$ is obtained from $\wt X$ by blowing-up the closure of $\wt X_1$ in $\wt X,$ i.e., the smooth manifold with boundary, $\overline{\wt X_1} = Y \times [0,1)_x \times W \times \{*\}.$ As this set consists of the cone points in the cone factor of $\wt X,$ we have
\begin{equation*}
	\mathrm{res}(X) = [\wt X; \overline{\wt X_1}] = Y \times [0,1)_x \times W \times [0,1)_r \times Z.
\end{equation*}
Let us denote the blow-down map by
\begin{equation*}
	\wt\beta': \mathrm{res}(X) \lra \wt X
\end{equation*}
and the composition of the two blow-down maps by $\beta = \wt\beta\circ \wt\beta': \mathrm{res}(X) \lra X.$

The resolution of $X$ is a smooth manifold with corners. It is partitioned into three subsets by the inverse images along $\beta$ of the stratification of $X,$ thus
\begin{equation*}
\begin{gathered}
	\mathrm{res}(X)_0 = Y \times \{0\} \times W \times [0,1)_r \times Z, \quad
	\mathrm{res}(X)_1 = Y \times [0,1)_x \times W \times \{0\} \times Z, \\
	\mathrm{res}(X)_2 = Y \times (0,1)_x \times W \times (0,1)_r \times Z.
\end{gathered}\end{equation*}
The first of these (and also the second) consists of a collective boundary hypersuface (i.e., a disjoint union of boundary hypersurfaces) and the last of these is simply the interior of $\mathrm{res}(X).$ There are natural (in this case trivial) fiber bundle structures 
\begin{equation*}
	\phi_0: \mathrm{res}(X)_0 \lra Y, \quad \text{ with fiber } \quad W \times [0,1)_r \times Z,
\end{equation*}
given by $\beta|_{\mathrm{res}(X)_0},$ and 
\begin{equation*}
	\phi_1: \mathrm{res}(X)_1 \lra Y \times [0,1)_x \times W, \quad \text{ with fiber } \quad Z,
\end{equation*}
given by $\wt\beta'|_{\mathrm{res}(X)_1}.$
Note that in the former case the base of the fiber bundle is $X_0$ and the fiber is the resolution of the link of $X_0$ in $X$ and the latter case the base of the fiber bundle is the resolution of the closure of $X_1$ in $X$ and the fiber is the link of $X_1$ in $X.$
At the corner $\mathrm{res}(X)_0\cap \mathrm{res}(X)_1 = Y \times \{0\} \times W \times \{0\} \times Z$ these fiber bundles are compatible in that there is a
fiber bundle $\phi_{10}:Y \times \{0\}\times W \lra Y$ (in this case trivial) participating in a commutative diagram of fiber bundles
\begin{equation*}
	\xymatrix{
	\mathrm{res}(X)_0\cap \mathrm{res}(X)_1 \ar[rr]^-{\phi_1} \ar[rd]^-{\phi_0} & & Y \times \{0\} \times W \ar[ld]_-{\phi_{10}} \\
	& Y & }
\end{equation*}
It is convenient to also endow $\mathrm{res}(X)_2$ with a fiber bundle but in this case we just take the identity map to itself. This structure makes $\mathrm{res}(X)$ a manifold with fibered corners.

\subsection{Smooth atlas stratified spaces}\label{subsect:atlas}

We find it convenient to adopt the point of view of \cite{AyaFraTan:LSSS} in which smoothly stratified spaces are defined
as topologically stratified spaces with a choice of smooth stratified atlas. Our definition of smooth function between stratified spaces, and hence
of what constitutes a smooth stratified atlas, differs from that of \cite{AyaFraTan:LSSS}, see Remark \ref{rem:Ayala}. For this reason we develop
the theory of these spaces without making use of the results in {\em loc cit}.
We will then explain why this class of spaces coincides with that of
stratified spaces in the sense of Thom-Mather, and with the class of manifolds with fibered corners. 
(The latter also goes by many names, e.g., interior maximal manifolds with fibered corners in \cite{KotRoc:PMWFC} and manifolds
with corners and iterated fibration structures in \cite{AlbGel:IFFDTOP}.)\\

\begin{defn} [Conical poset stratifications] \cite{AyaFraTan:LSSS}, \cite[Appendix A.5]{Lur:HA}
\begin{itemize}
\item If $A$ is a partially ordered set (poset) then a subset $\cal U \subseteq A$ is open if it is `closed upwards' meaning that if $x\leq y$ 
and $x \in \cal U$ then $y \in \cal U.$
\item If $A$ is a finite poset and $X$ a second countable, Hausdorff topological space then an {\bf $A$-stratification of $X$} is a continuous surjective map 
$X \to A.$ Open subsets of poset-stratified spaces naturally inherit poset-stratifications (though generally with a different poset). If $A$ consists of a single point we refer to the poset-stratification 
as the trivial stratification.
\item Given an $A$-stratification of $X,$ we let $X_a,$ $X_{\leq a},$ $X_{<a},$ $X_{\geq a},$ $X_{>a},$ denote the inverse images of the corresponding 
subsets of $A.$
\item If $A$ and $A'$ are posets, we define $A\times A'$ to be the poset with 
\begin{equation*}
	(a,a')\leq (b,b') \iff (a\leq b) \quad \& \quad (a'\leq b').
\end{equation*}
We define the product of two poset-stratified spaces $X\to A$ and $X'\to A',$ to be $X \times X' \to A \times A'$ with the product map.
\item Let $[1]$ be the two-element poset  $\{ 0 <1 \}$ and let $\bb R^+ \to [1]$ denote the $[1]$-stratification of $\bb R^+$ in which $0 \in \bb R^+$ is the inverse image of $0 \in [1].$
\item If $A$ is a poset we define $C(A)$ to be the poset obtained from $A$ by adjoining a new smallest element denoted $*$ or equivalently as the quotient poset 
\begin{equation*}
	C(A) = ([1] \times A) \diagup (\{0\} \times A).
\end{equation*}
If $X$ is a topological space then $C(X)$ denotes the quotient topological space
\begin{equation*}
	C(X) = (\bb R^+ \times X) \diagup ( \{0\} \times X )
\end{equation*}
and we denote the cone point by $*.$ The poset-stratification $\bb R^+ \times X \to [1] \times A$ descends to a poset-stratification $C(X) \to C(A).$ 
(Our convention is that $C(\emptyset)$ consists of a single point $\{*\}.$ We sometimes use $C_{[0,1]}(X)$ to indicate the quotient $([0,1] \times X) \diagup ( \{0\} \times X ).$)
\item A {\bf continuous stratified map} $f$ between two poset-stratified spaces, $X\to A$ and $X'\to A',$ is a commutative diagram of continuous maps
\begin{equation*}
	\xymatrix{ X \ar[r]^-f \ar[d] & X' \ar[d] \\ A \ar [r]^-{S(f)} & A'. }
\end{equation*}
A {\bf stratified homeomorphism} is a continuous stratified map that has a continuous stratified inverse.
\item
Let $X$ be an $A$-stratified space and let $\zeta \in X_a \subseteq X.$ 
A {\bf stratified chart} around $\zeta$ consists of an open subset $\cal U \subseteq X$ containing $\zeta,$ 
a natural number $h,$ a compact $A_{>a}$-stratified space $Z,$ and a stratified homeomorphism
\begin{equation*}
	\varphi: \cal U \lra \bb R^h \times C(Z)
\end{equation*}
which sends $\zeta$ to the point $\{0\} \times \{*\},$ where we regard $\cal U$ as poset-stratified as an open subset of $X,$ and $\bb R^h \times C(Z)$ as 
stratified with the product of the trivial stratification on $\bb R^h$ and the $C(A_{>a})\cong A_{\geq a}$ stratification on $C(Z).$
(If $A_{>a}=\emptyset$ then $Z = \emptyset$ and $C(Z) = \{*\}.$)
\item A {\bf conical $A$-stratification} of an $A$-stratified space $X$ consists of a collection of stratified charts on $X,$ $\{ \cal U_{\alpha}, \varphi_{\alpha}\}_{\alpha \in \mathscr A},$
such that the domains of these stratified charts form a basis for the topology of $X.$
\end{itemize}
\end{defn}

If $X$ is a conically $A$-stratified space then the connected components of $X_a,$ which we refer to as the strata of $X,$ are topological manifolds and
satisfy the `frontier condition', i.e., if a connected component $Y$ of some $X_a$ intersects the closure of a connected component $Y'$ of some $X_{a'}$
then $Y$ is contained in the closure of $Y',$
\begin{equation*}
	Y\cap \overline{Y'} \neq \emptyset \implies Y \subseteq \overline{Y'}
\end{equation*}
(see \cite[Lemma 2.3.7]{Fri:SIH}), moreover in this case $a\leq a'$ holds in $A.$ Spaces with conical $A$-stratifications are referred to as $C^0$-stratified spaces in \cite{AyaFraTan:LSSS}. If we assume that strata of $X$ are comparable only when their closures intersect, then spaces with conical $A$-stratifications coincide with Siebenmann's `CS sets' \cite{Sie:DHSS}, \cite[Definition 2.3.1]{Fri:SIH} (cf. \cite[Definition 3.10.1]{Pfl:AGSSS}).  If $\zeta \in X_a \in X$ has a stratified chart $\cal U \lra \bb R^h \times C(Z)$ then the space $C(Z)$ is determined up to stratified homeomorphism but $Z$ itself is not, see \cite[Example 2.3.12]{Fri:SIH}. In this generality $Z$ is referred to as {\em a} link of $\zeta$ in $X.$ For smooth atlas stratified spaces, defined momentarily, $Z$ will be determined up to stratified diffeomorphism and referred to as {\em the} link of $\zeta$ in $X.$

For simplicity, we will {\bf assume that every $X_a$ is connected}. This does not represent a real loss of generality as one could always refine the stratification to achieve it without changing the geometry. We refer to each $X_a$ as a {\bf stratum}. 

If $X$ is an $A$-stratified space, define the depth of $X$ to be the length of the longest ascending sequence $a_1 < a_2 < \cdots$ of elements in the 
range of $X\to A.$ If $X$ is conically stratified and $\zeta \in X$ define the depth of $\zeta$ in $X$ to be the depth of a link of $\zeta$ in $X$ (by convention, $\zeta$ has depth zero 
if its link in $X$ is $\emptyset$).

Note that is $X$ is endowed with a conical $A$-stratification then the points of depth zero are dense in $X$ and the corresponding strata are maximal.
We denote the union of these strata by
\begin{equation*}
	X_{\mathrm{reg}} = \{ \zeta \in X: \text{ depth }(\zeta)=0\}
\end{equation*}
and refer to it as the {\bf regular part of $X.$} 

Note that different components of the regular part might have different dimensions. For example, if $Z = \{\mathrm{pt}\} \sqcup \bb S^1$ then 
we can identify $C(Z)$ with the subset of $\bb R^3$ consisting of the $x$-$y$ plane and the non-negative $z$-axis and so its regular part with the
positive $z$-axis and the $x$-$y$ plane minus the origin. A conical $A$-stratified {\bf pseudomanifold} is a conical $A$-stratified space 
whose regular part is dense and has a fixed dimension and whose other strata have dimension at least two less than the regular part.

\begin{defn} [Smooth atlas stratified spaces and functions] \label{def:SASS}
Our definition proceeds by induction on depth. 

A {\bf smoothly stratified space of depth $0$} is a smooth manifold and a smooth stratified map between smoothly stratified spaces of depth zero is a 
smooth map of smooth manifolds.

Assume inductively that we have defined smoothly stratified spaces of depth less than $k,$ as well as smooth stratified maps between stratified spaces 
of depth less than $k.$

We first define smooth stratified maps between spaces that play the r\^ole of coordinate charts as follows.
If $Z_1 \lra A_1$ and $Z_2 \lra A_2$ are smooth atlas stratified spaces of depth less than $k,$  $h_1, h_2 \in \bb N_0,$ and 
\begin{equation*}
	f: \bb R^{h_1} \times C(Z_1) \lra \bb R^{h_2} \times C(Z_2)
\end{equation*}
is a continuous stratified map then we will say that it is a {\bf smooth stratified map} if there a smooth stratified map
\begin{equation*}
	\wt f: \bb R^{h_1} \times \bb R \times Z_1 \lra \bb R^{h_2} \times \bb R \times Z_2,
\end{equation*}
(as defined by the inductive hypothesis) that makes the diagram
\begin{equation*}
	\xymatrix{
	 \bb R^{h_1} \times \bb R \times Z_1 \ar[r]^-{\wt f} & \bb R^{h_2} \times \bb R \times Z_2 \\
	 \bb R^{h_1} \times \bb R^+ \times Z_1 \ar@{^(->}[u] \ar[d] & \bb R^{h_2} \times \bb R^+ \times Z_2  \ar@{^(->}[u] \ar[d]  \\
	 \bb R^{h_1} \times C(Z_1) \ar[r]^-f & \bb R^{h_2} \times C(Z_2), }
\end{equation*}
where the descending vertical arrows are the quotient maps of the cones, commutative.
(If $Z_2 = \emptyset$ then the corresponding $\wt f$ maps $\bb R^{h_1}\times \bb R \times Z_1$ into $\bb R^{h_2}.$)

Next, let $X$ be a conically $A$-stratified space of depth $k.$ By a {\bf smoothly stratified chart} on $X$ we mean a stratified chart
\begin{equation*}
	X \supseteq \cal U \xlra{\varphi} \bb R^h \times C(Z)
\end{equation*}
in which $Z$ is a smoothly stratified space of depth less than $k.$ Two smoothly stratified charts $\{ \cal U_1, \varphi_1\}$ and $\{\cal U_2, \varphi_2\}$ are compatible
if either $\cal U_1 \cap \cal U_2 = \emptyset$ or the stratified homeomorphism
\begin{equation*}
	\psi= \varphi_2 \circ \varphi_1^{-1}:  \varphi_1( \cal U_1 \cap \cal U_2) \lra \varphi_2 ( \cal U_2 \cap \cal U_2 ),
\end{equation*}
where $\varphi_j(\cal U_1 \cap \cal U_2)$ is stratified as an open subset of $\bb R^{h_j} \times C(Z_j),$ is a smooth stratified map and so is its inverse.

A {\bf smoothly stratified atlas} on a conically poset-stratified space $X$ of depth $k$ consists of a collection of compatible smoothly stratified charts
$\{\cal U_\alpha, \varphi_{\alpha}\},$ whose domains form an open cover of $X.$ 
Two smoothy stratified atlases are equivalent if their charts are mutually compatible.

A {\bf smooth atlas stratified space of depth $k$} consists of a conically poset-stratified space $X$ of depth $k$ 
together with an equivalence class of smoothly stratified atlases, $(X, \mathscr A).$
A {\bf smooth stratified map $f: X \lra Y$} between smoothly stratified spaces of depth at most $k$ is a continuous stratified map such that 
for each smooth chart $(\cal U_1, \varphi_1)$ on $X$ and $(\cal U_2, \varphi_2)$ on $Y,$ compatible with the respective atlases, 
satisfying $f(\cal U_1) \subseteq \cal U_2,$ the composition 
\begin{equation*}
	\varphi_2 f \varphi_1^{-1}: \bb R^{h_1} \times C(Z_1) \lra \bb R^{h_2} \times C(Z_2)
\end{equation*}
is a smooth stratified map.
A {\bf stratified diffeomorphism} between two smoothly stratified spaces of depth at most $k$ is a stratified homeomorphism that is smooth with smooth inverse.
\end{defn}

We will usually refer to a smooth atlas stratified space simply as a {\bf smoothly stratified space.} This is justified by the fact, established in \S\ref{subsec:SASSThomMather}, that this is the same class of spaces as Thom-Mather stratified spaces and (hence) as Whitney stratified subsets of Euclidean space.

As anticipated, the advantage of this definition is that it helps to generalize the constructions of differential geometry to stratified spaces.
For example, let us consider the Cartesian product. An easy induction shows that if $X_1\lra A_1$ and $X_2\lra A_2$ are smooth atlas stratified spaces 
then so is $X_1 \times X_2 \lra A_1 \times A_2$ with coordinate charts given by the product of coordinate charts on each factor. Indeed, for the inductive 
step it suffices to recall that
\begin{equation*}
	\bb R^{h_1}\times C(Z_1) \times \bb R^{h_2} \times C(Z_2) = \bb R^{h_1+h_2} \times C(Z_1 * Z_2),
\end{equation*}
where $Z_1 * Z_2,$ the join of $Z_1$ and $Z_2,$ is given by
\begin{equation*}
	Z_1 * Z_2 = C_{[0,1]}(Z_1) \times Z_2 \bigsqcup_{Z_1 \times Z_2} Z_1 \times C_{[0,1]}(Z_2)
\end{equation*}
and so, inductively, a smooth atlas stratified space.

Similarly, {\bf a fiber bundle of smoothly stratified spaces}
\begin{equation*}
	W \fib X \xlra{\phi} Y
\end{equation*}
is a smooth map between smoothly stratified spaces $X \xlra{\phi} Y$ such that for every point in $Y$ there is an open neighborhood 
$\cal U$ of that point and a stratified diffeomorphism $W \times \cal U \lra \phi^{-1}(\cal U)$ such that the diagram
\begin{equation*}
	\xymatrix{
	W \times \cal U \ar[r]\ar[d] &  \phi^{-1}(\cal U) \ar[d]^-{\phi} \\
	\cal U \ar@{^(->}[r] & Y }
\end{equation*}
in which the left vertical arrow is the canonical projection, commutes. The stratum poset of $X$ is the product of those of $W$ and $Y.$

A smoothly stratified pseudomanifold is a smoothly stratified space whose regular part is dense and has a fixed dimension and whose other strata have dimension at least two less than the regular part. We say that a smoothly stratified space is orientable if its regular part is an orientable manifold and an orientation refers to an orientation of the regular part.

\begin{remark} [Relation to conically smooth stratified spaces] \label{rem:Ayala}
Our definition of smooth atlas stratified spaces is inspired by that of ``conically smooth stratified spaces" in \cite{AyaFraTan:LSSS}, however our
notion of smooth function is different from theirs. Suppose $Z$ is a smooth manifold and let us consider what constitutes a smooth stratified function 
from $C(Z)$ to itself. For us, this is the same as a stratified function $F:C(Z) \lra C(Z)$ that lifts to a smooth function 
$\bb R^+ \times Z \lra \bb R^+ \times Z;$ equivalently, a stratified function $F:C(Z) \lra C(Z)$ which restricted to the regular parts
$(0,\infty) \times Z \lra (0,\infty) \times Z$ extends by continuity to a smooth function $[0,\infty) \times Z \lra [0,\infty)\times Z.$

On the other hand, from \cite[Definition 3.3.1]{AyaFraTan:LSSS}, a function $F:C(Z) \lra C(Z)$ will be conically smooth if its restriction to the regular
parts $(0,\infty) \times Z \lra (0,\infty) \times Z$ is smooth and $f$ is ``conically smooth along $\bb R^0$". Let us consider what this means for a function
whose restriction to the regular parts is of the form 
\begin{equation*}
	F(s,z) = (s, f_s(z))
\end{equation*}
where $s \mapsto f_s$ is a smooth family of functions on $Z.$ From \cite[Definition 3.1.4]{AyaFraTan:LSSS}, in order for this to be 
``$C^1$ along $\bb R^0$," we make use of the dilation $\gamma_t(s,z) = (ts,z),$ consider
\begin{equation*}
\begin{gathered}
	Df: (0,\infty) \times Z \lra (0,\infty) \times Z, \\ 
	Df(s,z) := \lim_{t\to 0} (\gamma_t^{-1} \circ f \circ \gamma_t) (s,z) = (s, f_0(z)),
\end{gathered}\end{equation*}
and require for this function to extend to a continuous function $C(Z) \lra C(Z).$ This computation also shows that $D(Df),$ defined in the same way, is
again equal to $Df.$ Thus as long as $f_0:Z \lra Z$ is a continuous function, $F$ will be conically smooth along $\bb R^0$ and hence conically smooth 
as a map $C(Z) \lra C(Z).$ Note that, e.g., by mollification, it is possible for any continuous function $f_0:Z \lra Z$ to appear as the limit of a family 
$f_s(z)$ which is smooth in both $s$ and $z$ for $s>0.$ On the other hand, in order for $F$ to be smooth in our sense $f_0(z)$ must be smooth.
(The authors are grateful to Lukas Waas for explaining this example to them.)

\end{remark}

\subsection{The resolution of a smooth atlas stratified space}\label{subsect:resolution}

There is a resolution process for smoothly stratified spaces, somewhat analogous to a $\cal C^{\infty}$ resolution of singularities, that can be interpreted as saying that any smoothly stratified space of finite depth can be obtained from a manifold with corners by making certain identifications on the boundaries. In fact this construction establishes an equivalence of categories between smoothly stratified space of finite depth and a class of manifolds identified by Richard Melrose. We discuss the resolution of a stratified space with a single singular stratum into a smooth manifold with fibered boundary in the prelude above. For a stratified space with more than a single singular stratum, the resolution of $X$ is a manifold with corners and its structure is the following.

\begin{defn} [Manifolds with fibered corners] \cite[Definition 2]{AlbLeiMazPia:SPWS}, \cite[Definition 3.3]{AlbMel:RSGA}\\
Let $M$ be a second countable Hausdorff topological space.
\begin{itemize}
\item A {\bf chart with corners} on $M$ consists of an open subset $\cal U \subseteq M$ and a homeomorphism
\begin{equation*}
	\varphi: \cal U \lra \bb R^h \times [0,\infty)^k
\end{equation*}
for some $h, k \geq 0.$
Two charts, $(\cal U_1, \varphi_1)$ and $(\cal U_2, \varphi_2),$ are compatible if either $\cal U_1 \cap \cal U_2 = \emptyset$ or
\begin{equation*}
	\varphi_2\circ \varphi_1^{-1}: \varphi_1(\cal U_1\cap \cal U_2) \lra \varphi_2(\cal U_1\cap \cal U_2)
\end{equation*}
is a diffeomorphism of open subsets of $\bb R^{h_1} \times [0,\infty)^{k_1}$ and $\bb R^{h_2} \times [0,\infty)^{k_2}.$
(This means that there are open sets $\cal V_j \subseteq \bb R^{h_j+k_j}$ such that 
\begin{equation*}
	\varphi_j(\cal U_j\cap \cal U_j) = \cal V_j \cap (\bb R^{h_j} \times [0,\infty)^{k_j})
\end{equation*}
and a diffeomorphism $\cal V_1 \lra \cal V_2$ that restricts to $\varphi_2\circ \varphi_1^{-1}.$ In particular $h_1+k_1=h_2+k_2=:n,$ $h_1=h_2,$ and $k_1=k_2.$
\item A {\bf $\cal C^{\infty}$-structure with corners} on $M$ is a collection, $\{ \cal U_{\alpha}, \varphi_{\alpha}\}_{\alpha \in \mathscr A},$ of compatible charts with corners on $M$ such that the domains of these stratified charts form a basis for the topology of $M.$
\item If $M$ has a $\cal C^{\infty}$-structure with corners and $\zeta \in X$ is represented by $(x_1, \ldots, x_n)$ in some chart with corners then the number of zeros in this $n$-tuple,which we denote by $c(\zeta),$ is independent of the chart. By a {\bf boundary hypersurface} of $M$ we mean the closure of a component of $\{ \zeta \in M: c(\zeta) =1\}$ and by a collective boundary hypersurface we mean the union of a disjoint collection of boundary hypersurfaces. 
\item The space $M$ is a {\bf manifold with corners} if it is endowed with a $\cal C^{\infty}$-structure with corners and every boundary hypersurface is embedded. (Equivalently, for every boundary hypersurface $H$ there is a non-negative smooth function $\rho: M \lra \bb R$ such that $\rho^{-1}(0) = H$ and near each point of $H$ there are coordinates with $\rho$ as a first element. We call such a function a boundary defining function for $H.$)
\item By a {\bf boundary stratification} of a manifold with corners $M$ we mean a poset-stratification $S: M \lra A$ such that, for each $a\in A,$ 
\begin{equation*}
	\pa_a M := S^{-1}(a)
\end{equation*}
is either the interior of $M$ or a collective boundary hypersurface and such that if $\overline{S^{-1}(a)} \cap \overline{S^{-1}(b)} \neq \emptyset$ then $a$ and $b$ are comparable.
\item An {\bf iterated fibration structure} on a manifold with corners $M$ with a boundary stratification $S: M \lra A$ consists of, for each $a\in A,$ a smooth fiber bundle with compact fibers,
\begin{equation*}
	F_aM - \pa_aM \xlra{\phi_a} B_aM,
\end{equation*}
where the base $B_aM$ and the fiber $F_aM$ are smooth manifolds with corners, and, whenever $\overline{S^{-1}(a)} \cap \overline{S^{-1}(b)} \neq \emptyset$ and $a<b,$ a commutative diagram of smooth fiber bundles
\begin{equation*}
	\xymatrix{
	\pa_a M \cap \pa_b M \ar[rr]^-{\phi_b} \ar[rd]^-{\phi_a} & & \pa_aB_bM \ar[ld]_-{\phi_{ba}} \\ & B_aM & }
\end{equation*}
where $\pa_a B_bM$ is a collective boundary hypersurface of $B_bM.$
For the interior of $M$ the fiber bundle is the identity map $M^{\circ} \lra M^{\circ}.$

\item We refer to a manifold with corners with an iterated fibration structure as a {\bf manifold with fibered corners}.
\item If $M$ and $N$ are manifolds with corners and iterated fibration structures then by a {\bf smooth fibered corners map} between them we mean a smooth stratified map between them that sends fibers of each boundary fiber bundle of $M$ to fibers of the corresponding boundary fiber bundle of $N.$ By a {\bf fibered corners diffeomorphism} we mean an invertible smooth fibered map whose inverse is also a smooth fibered map.
\end{itemize}

\end{defn}

\begin{remark}
Note that if $M$ is a connected manifold with fibered corners then the interior of $M$ corresponds to a maximum in the poset of strata.
In the terminology of \cite{KotRoc:PMWFC} what we have described is an ``interior maximal" manifold with fibered corners. As we shall have no need of other iterated fibration structures we omit the `interior maximal' qualifier.
\end{remark}

\begin{remark}
If $M$ is a manifold with fibered corners and $F_aM - \pa_aM \xlra{\phi_a} B_aM$ is one of the corresponding boundary fiber bundles, then both the fiber $F_aM$ and the base $B_aM$ themselves inherit the structure of manifold with fibered corners (see, e.g., \cite[Proposition 4.2]{KotRoc:PMWFC}).
\end{remark}

\begin{remark}\label{rmk:StratMfdCorners}
Note that a manifold with corners can be considered as a smoothly stratified space with strata given by the connected components of the level sets of the codimension function $c$ above. Endowing $(\bb R^+)^h$ with this stratification, which we refer to as the {\bf natural stratification of a manifold with corners},
we can obtain a smooth stratified atlas from an atlas of charts with corners. 
 This is generally not a boundary stratification as just defined but will be useful in section \ref{sec:ResStratFibBdles}. 
 (For example if $M=[0,1]^2$ then an example of a boundary stratification would be to take $A = \{0,1\},$ $\pa_0M = \{0,1\} \times [0,1]$ and $\pa_1M = [0,1] \times \{0,1\},$ while the natural stratification would have strata $(0,0),$ $(0,1),$ $(1,0),$ $(1,1),$ $\{0\} \times (0,1),$ $\{1\} \times (0,1),$ $(0,1) \times \{0\},$ and $(0,1) \times \{1\}.$)
 \end{remark}

To resolve a stratified space we will perform a radial blow-up of a minimal singular stratum thereby replacing it with a boundary hypersurface. To carry this out
systematically we will need to introduce another definition for the intermediate spaces we will encounter.

\begin{defn} [Hybrid spaces and functions] \label{def:Hybrid}
We will define hybrid spaces as particular types of poset stratified spaces, inductively by depth. 
A {\bf hybrid space of depth zero} is a smooth manifold and a smooth hybrid function is an ordinary smooth function.
A {\bf hybrid space of depth one} is either a smooth manifold with boundary or a smooth atlas stratified space of depth one, with the corresponding smooth functions as defined above.

Assume inductively that we have defined hybrid spaces of depth less than $k$ and the corresponding smooth functions and let $A$ be poset of depth $k.$
Let $X$ be a connected second countable Hausdorff topological space with a poset-stratification $X \lra A.$
Assume that $A$ is partitioned into
\begin{equation*}
	A = A_{\mathrm{res}} \sqcup A_{\mathrm{sing}} \sqcup \{ \mathrm{reg} \}.
\end{equation*}
The stratum $X_{\mathrm{reg}},$ known as the regular part, is open and dense, and $\mathrm{reg}$ is the maximum element of $A.$
The set $A_{\mathrm{res}}$ (which will correspond to the strata that have already been resolved) and the set $A_{\mathrm{sing}}$ (which will correspond
to the strata that have not yet been resolved) are possibly empty, and no element of $A_{\mathrm{res}}$ is greater than an element in $A_{\mathrm{sing}}.$

For $a \in A_{\mathrm{res}},$ $X_a$ is a hybrid space of depth less than $k$ and participates in a fiber bundle
\begin{equation*}
	Z \fib X_a \xlra{\phi_a} B_aX
\end{equation*}
where $B_aX$ is a smooth manifold with corners with a boundary stratification and $Z$ is a hybrid space of depth less than $k.$
A {\bf res-chart} centered at a point $\zeta \in X_a$ is a map $\varphi: \cal U \lra \bb R^+ \times \cal V \times Z,$ where $\cal V$ is a coordinate chart in $B_aX,$ with the property that $\cal U \cap X_a = \phi_a^{-1}(\cal V)$ and $\varphi|_{\cal U\cap X_a}$ participates in a trivialization of $\phi_a,$ 
\begin{equation}\label{trivial-hybrid}
	\xymatrix{
	\phi_a^{-1}(\cal V)  \ar[rr]^-{\varphi|_{\cal U\cap X_a}} \ar[rd]^-{\phi_a} & & \cal V \times Z \ar[ld] \\ & \cal V}
\end{equation}
and such that $\varphi$ is a stratified homeomorphism which sends strata of $\cal U$ corresponding to $A_{\mathrm{res}}$ to (strata of $\bb R^+ \times \cal V)$ times $Z$ and strata of $\cal U$ corresponding to $A_{\mathrm{sing}}$ to $\bb R^+ \times \cal V$ times (strata of $Z).$

For $a\in A_{\mathrm{sing}}\cup \{\mathrm{reg}\},$ $X_a$ is a smooth manifold. A {\bf sing-chart} centered at a point $\zeta \in X_a$ consists of a smoothly stratified chart which does not intersect any $X_b$ with $b \in A_{\mathrm{res}}.$ 

A {\bf smooth hybrid map} between sets of the form 
\begin{equation}\label{eq:GenHybridMap}
	(\bb R^+)^h \times Z \supseteq \cal V \times Z \lra  \cal V' \times Z' \subseteq (\bb R^+)^{h'} \times Z',
\end{equation}
in which $Z$ and $Z'$ are hybrid spaces of depth less than $k,$ is a poset-stratified map when the first factors are endowed with a boundary stratification if appropriate (i.e., if $\cal V$ has boundary) which is also a smooth map when 
the first factors are considered as smooth (i.e., trivially stratified).

Two overlapping charts are compatible if the corresponding transition map of the form \eqref{eq:GenHybridMap} is a smooth hybrid difeomorphism.

A {\bf hybrid atlas} on $X$ consists of a collection of compatible charts whose domains form a basis for the topology of $X.$ A space $X$ as above endowed with a hybrid atlas is a {\bf hybrid space of depth $k.$} A smooth map between hybrid spaces of depth at most $k$ is a continuous stratified map whose localization to charts in the atlases is smooth as described above.

\end{defn}

To show that every smooth atlas stratified space resolves to a manifold with fibered corners, the key result is the following proposition.

\begin{prop}\label{prop:resolution}
Let $X\lra A$ be a hybrid space with partition $A = A_{\mathrm{res}} \sqcup A_{\mathrm{sing}} \sqcup \{\mathrm{reg}\}$ such that $A_{\mathrm{sing}}\neq \emptyset,$ and let $a$ be a minimal element of $A_{\mathrm{sing}}.$ 
\begin{enumerate} [i)]
\item the closure of $X_a$ in $X,$ $\overline{X_a},$ is a smooth manifold with corners,
\item there is a hybrid space $\scr S_a$ participating in a fiber bundle 
\begin{equation*}
	Z \fib \scr S_a \xlra{\phi_a} \overline{X_a}
\end{equation*}
where $Z$ is the link of $X_a$ in $X,$
\item a hybrid atlas on $X$ induces a hybrid atlas on
\begin{equation*}
	[X;\overline{X_a}] := (X\setminus \overline{X_a}) \sqcup \scr S_a
\end{equation*}
endowed with the poset stratification $[X;\overline{X_a}] \lra A'$ where 
\begin{equation*}
	A' = A \quad \text{ with } \quad
	A'_{\mathrm{res}} = 
	A_{\mathrm{res}} \cup \{a\}, \quad
	A'_{\mathrm{sing}} =
	A_{\mathrm{sing}} \setminus\{a\}
\end{equation*}
and where the strata are
\begin{equation*}
	[X;\overline{X_a}]_b 
	= \begin{cases}
	\overline{ X_b\setminus \overline{X_a}} & \text{ if } b<a \\
	\scr S_a & \text{ if } b=a \\
	X_b & \text{ otherwise }
	\end{cases}
\end{equation*}
\item if there is a stratum $X_b$ such that $X_b \cap \overline {X_a} \neq \emptyset$ (necessarily $b<a$) and 
\begin{equation*}
	X_b \xlra{\phi_b} B_bX, \quad [X;\overline{X_a}]_b \xlra{\wt\phi_b} B_bX
\end{equation*}
are the corresponding fiber bundles, then 
\begin{equation*}
	[X;\overline{X_a}]_b \cap \scr S_a = \phi_a^{-1}(\overline X_a \cap X_b)
\end{equation*}
participates in the commutative diagram
\begin{equation*}
	\xymatrix{
	[X;\overline{X_a}]_b \cap \scr S_a \ar[rr]^-{\phi_a} \ar[rd]_-{\wt \phi_b} & & \overline X_a \cap X_b \ar[ld]^-{\phi_b} \\ & B_bX & }
\end{equation*}
\end{enumerate}
\end{prop}

\begin{proof}
We proceed by induction on the depth of $A.$ If the depth of $A$ is zero then there is nothing to prove.

i) Let us denote the link of $X_a$ in $X$ by $Z,$ so that each point in $X_a$ has a stratified chart that identifies a neighborhood with $\bb R^{h}\times C(Z).$ Our immediate goal is to show that points in the closure of $X_a$ have compatible charts with image in $(\bb R^+)^h \times C(Z).$
{\em Note that here to ease the notational burden we will use charts into $(\bb R^+)^h \times C(Z)$ with the understanding that they are homeomorphisms onto their image but while we assume that this image projects onto the second factor we do not assume that it is onto the first factor.
Thus when we use $(\bb R^+)^h$ we are allowing corners up to codimension $h$ but not asserting that this codimension is realized.} 
If we recall that a smooth map defined on $(\bb R^+)^h$ means the restriction of a smooth map defined on some open set in $\bb R^h,$ then our notion
of smooth map between the images of charts of this form is as one would expect.

We know that strata in $A_{\mathrm{sing}}$ whose closures have non-trivial intersection satisfy the frontier condition, so if $\zeta \in \overline X_a\setminus X_a$ is contained in $X_b$ then $b \in A_{\mathrm{res}}.$ If 
\begin{equation*}
	\varphi:\cal U \lra \bb R^+ \times \cal V \times W
\end{equation*}
is a res-chart of $X_b$ centered at $\zeta,$ where we are
using the notation from Definition \ref{def:Hybrid}, then $\varphi(\cal U \cap \overline{X_a})$ is of the form $\bb R^+ \times \cal V \times W'$ for some stratum $W'$ of $W.$ Now $W$ is itself a hybrid space and since approaching $W'$ from $W\setminus W'$ corresponds via $\varphi$ to approaching $\overline{X_a}$ from $X\setminus \overline{X_a},$ $W'$ must be a stratum in the `$\mathrm{sing}$' part of the partition of the strata of $W.$ Thus points in $W'$ have sing-charts in $W$ of the form
$\varphi':\cal V' \lra \bb R^{h'} \times C(L)$ for some $h'$ and $L$ (the link of $W'$ in $W$). Let $\cal U' \subseteq X$ be $\varphi^{-1}(\bb R^+ \times \cal V \times \cal V').$
By combining $\varphi$ with an open embedding of $\cal V$ into an appropriate $(\bb R^+)^{h''}$ and with $\varphi'$ we obtain
\begin{equation*}
	\psi:\cal U' \xlra{\varphi} \bb R^+ \times \cal V \times \cal V' \lra \bb R^+ \times (\bb R^+)^{h''} \times \bb R^{h'} \times C(L).
\end{equation*}
Since the restriction of $\psi$ to $\cal U'\cap X_a$ is a coordinate chart (a sing-chart for $X_a$) we see that $1+h'+h''=h$ and $L=Z.$ 
To simplify the notation we assume that the image of $\psi$ lies in $\bb R^+ \times (\bb R^+)^{h''} \times (\bb R^+)^{h'} \times C(Z) = (\bb R^+)^h \times C(Z)$ (by composing with a stratified diffeomorphism if necessary) and we refer to 
\begin{equation*}
	\psi: \cal U' \lra (\bb R^+)^h \times C(Z),
\end{equation*}
as an extended coordinate chart.

Note that 
\begin{equation}\label{image-of-psi}
	\psi(\cal U' \cap \overline{X_a})\subseteq (\bb R^+)^h \times \{*\}
\end{equation}
(since $\varphi'(W') \subseteq \bb R^{h'}\times \{*\}$) and that every point in $\overline X_a\setminus X_a$ is in the domain of such an extended coordinate chart. These charts and the compatibility they inherit
from that of the charts of $X$ give $\overline{X_a}$ the structure of a smooth manifold with corners.\\

ii) Let $\scr A$ be the union of a hybrid atlas for $X$ and the extended
coordinate charts for points in $\overline X_a\setminus X_a$ from (i),
and let $\scr A_a$ be the subset of those charts centered at points in $\overline{X_a},$

\begin{equation*}
	\scr A_a = \{ (\cal U \xlra{\varphi} (\bb R^+)^h \times C(Z)) \in \scr A: \varphi^{-1}((\bb R^+)^h \times \{*\}) \subseteq \overline{X_a} \}.
\end{equation*}
Let us consider $\bb R^+\lra [1]$ as a smooth atlas stratified space with poset $\{ 0<1\}$ and $\bb R^+_0 = \{0\},$ and let
$\scr P_a$ be the paths in $X$ that start at $\overline{X_a}$ and immediately move on to `higher' strata,
\begin{equation*}
	\scr P_a = \{\chi: \bb R^+ \lra X \text{ smooth stratified map}: \chi(t) \in \overline{X_a} \iff t=0\}.
\end{equation*}
By definition of smooth stratified maps, whenever $(\cal U \xlra{\varphi} (\bb R^+)^h \times C(Z))\in \scr A_a$ and $\chi \in \scr P_a,$
$\varphi\circ \chi$ has a lift
\begin{equation*}
	\xymatrix{
	& (\bb R^+)^h \times \bb R^+ \times Z \ar[d] \\
	\bb R^+ \ar[r]^-{\varphi\circ \chi} \ar[ru]^-{\wt{\varphi\circ\chi}} & (\bb R^+)^h \times C(Z) }
\end{equation*}
Let us define an equivalence relation on $\scr P_a$ by
\begin{equation*}
	\chi_1 \sim \chi_2 \iff
	\begin{cases}
	\chi_1(0) = \chi_2(0) & \text{ and }\\ $ $\\
	\wt{\varphi\circ\chi_1}(0) = \wt{\varphi\circ\chi_2}(0) & \text{ for all charts in $\scr A_a$ including $\chi_1(0)$}
	\end{cases} 
\end{equation*}
and denote the set of equivalence classes by $\scr S_a,$
\begin{equation*}
	\scr S_a = \scr P_a \diagup \sim.
\end{equation*}
To give the set $\scr S_a$ some more structure, first note that there is an obvious map
\begin{equation*}
	\phi_a: \scr S_a \lra \overline{X_a}, \quad [\chi]\mapsto \chi(0).
\end{equation*}
Next, given $(\cal U \xlra{\varphi} (\bb R^+)^h \times C(Z))\in \scr A_a,$ let us denote 
\begin{equation*}
	\cal U_a = \cal U \cap \overline{X_a}, \quad \varphi_a = \varphi|_{\cal U_a}: \cal U_a \lra (\bb R^+)^h
\end{equation*}
and then 
\begin{equation}\label{eq:InducedCoord}
	\xymatrix @R=1pt {
	\widetilde\varphi_a: \phi_a^{-1}(\cal U_a) \ar[r] & (\bb R^+)^h \times Z,\\
	[\chi] \ar@{|->}[r] & \wt{\varphi\circ \chi}(0) }
	\text{ and }
	\varphi_{\scr S_a} = (\varphi_a^{-1}\times \id)\circ \widetilde \varphi_a: \phi_a^{-1}(\cal U_a) \lra \cal U_a \times Z.
\end{equation}
If $(\cal U' \xlra{\varphi'} (\bb R^+)^h \times C(Z))\in \scr A_a$ is another chart with $\cal U'\cap \overline{X_a} = \cal U_a$ then, on $\cal U\cap \cal U',$ 
$\varphi' = f \circ \varphi$ for some stratified diffeomorphism whose lift 
\begin{equation}\label{eq:InducedDiff}
	\xymatrix{
	(\bb R^+)^h \times \bb R \times Z \ar@{}[r]|-*[@]{\supseteq}  & \cal V \times \bb R \times Z \ar[r]^-{\wt f} 
		& \cal V' \times \bb R \times Z \ar@{}[r]|-*[@]{\subseteq} & (\bb R^+)^h \times \bb R \times Z  \\
	(\bb R^+)^h \times \bb R^+ \times Z \ar@{}[r]|-*[@]{\supseteq}  & \cal V \times \bb R^+ \times Z \ar@{^(->}[u] \ar[d] \ar[r]^{\wt f_{\geq 0}}
		& \cal V' \times \bb R^+ \times Z \ar@{}[r]|-*[@]{\subseteq} \ar@{^(->}[u] \ar[d]  & (\bb R^+)^h \times \bb R^+ \times Z  \\
	(\bb R^+)^h \times C(Z) \ar@{}[r]|-*[@]{\supseteq}  & \cal V \times C(Z) \ar[r]^-{f} 
		& \cal V' \times C(Z) \ar@{}[r]|-*[@]{\subseteq} & (\bb R^+)^h \times C(Z) }
\end{equation}
restricts to a stratified diffeomorphism  $\cal V \times \{0\} \times Z \lra \cal V' \times \{0\} \times Z$ 
of the form $(y,0,z) \mapsto (\wt f'(y),0, \wt f''(y,z)).$
It follows that $\varphi_{\scr S_a}': \phi_a^{-1}(\cal U_a) \lra \cal U_a \times Z$ is given by
\begin{equation*}
	\phi_a^{-1}(\cal U_a) \xlra{\varphi_{\scr S_a}} \cal U_a \times Z \xlra{\wt f'''} \cal U_a \times Z, \text{ where }
	\wt f'''(u,z) = (u, \wt f''(\varphi_a(u),z))
\end{equation*}
and hence differs from $\varphi_{\scr S_a}$ by a stratified diffeomorphism. 
Thus we may, without ambiguity, give $\scr S_a$ the structure of a hybrid stratified space by requiring that all of the maps
$\varphi_{\scr S_a}$ are hybrid diffeomorphisms.
Finally, over each $\cal U_a$ we have a commutative diagram
\begin{equation*}
	\xymatrix{ \phi_a^{-1}(\cal U_a) \ar[rr]^-{\varphi_{\cal S_a}} \ar[rd]_-{\phi_a} & & \cal U_a \times Z \ar[ld] \\
	& \cal U_a & }
\end{equation*}
with the arrow on the right given by the projection onto the first factor, so $\phi_a: \scr S_a \lra \overline{X_a}$ is a stratified fiber bundle with fiber $Z.$\\

iii)
Next we define the blow-up of $X$ along $\overline{X_a}$ to be
\begin{equation*}
	[X;\overline{X_a}] = (X \setminus \overline{X_a}) \sqcup \scr S_a
\end{equation*}
together with a blow-down map 
\begin{equation*}
	\beta: [X;\overline{X_a}] \lra X
\end{equation*}
given by collapsing the fibers of $\phi_a.$ We define a poset stratification $[X;\overline{X_a}]\lra A$ by
\begin{equation*}
	[X;\overline{X_a}]_b 
	= \begin{cases}
	\text{ closure of }\beta^{-1}( X_b\setminus \overline{X_a} ) & \text{ if } b<a \\
	\scr S_a & \text{ if } b=a \\
	\beta^{-1}(X_b) & \text{ otherwise }
	\end{cases}
\end{equation*}
(which is just a more careful version of the statement in the proposition). We will show that a hybrid atlas for $X$ induces a hybrid atlas for $[X;\overline{X_a}]$
with $a$ moved from the `sing' part of the partition of $A$ to the `res' part. That is, we will show that we can produce a hybrid atlas for $[X;\overline{X_a}]$ for which points in a stratum indexed by $A_{\mathrm{res}}\cup \{a\}$ have res-charts and point in a stratum indexed by $A_{\mathrm{sing}}\setminus \{a\}$ have sing-charts.

First let us produce res-charts for the new res stratum $[X;\overline{X_a}]_a = \scr S_a.$
The interesting charts we need to consider are the ones obtained from $\scr A$ as follows.
If $(\cal U \xlra{\varphi} (\bb R^+)^h \times C(Z))\in \scr A_a$ then, recalling the definition of $\wt\varphi_a$ from \eqref{eq:InducedCoord}, we let $\cal U_a = \cal U\cap \overline{X_a}$ and 
\begin{equation*}
	\widehat {\cal U} = (\cal U \setminus \cal U_a) \sqcup \phi_a^{-1}(\cal U_a), \quad
	\widehat {\varphi} = \varphi|_{\cal U \setminus \cal U_a} \sqcup \wt \varphi_a: \widehat {\cal U} \lra (\bb R^+)^h \times \bb R^+ \times Z
\end{equation*}
where we are identifying $\mathrm{Image}(\wt \varphi_a) \subseteq (\bb R^+)^h \times Z$ as a subset of $(\bb R^+)^h \times \{0\} \times Z.$
Let us consider the behavior of this construction when changing coordinate chart. Suppose $(\cal U' \xlra{\varphi'} (\bb R^+)^h \times C(Z))\in \scr A_a$ is another chart with $\cal U \cap \cal U' \neq \emptyset$ 
then, on $\cal U\cap \cal U',$ $\varphi' = f \circ \varphi$ for some stratified diffeomorphism participating in a diagram like \eqref{eq:InducedDiff}
and, since $f|_{\cal V \times (0,\infty) \times Z} = \wt f|_{\cal V \times (0,\infty) \times Z}$ and $\wt\varphi' = \wt f|_{\cal V \times \{0\} \times Z} \circ \wt \varphi,$ we have
\begin{equation*}
	\widehat\varphi' = \wt f_{\geq0} \circ \widehat\varphi.
\end{equation*}
Thus $\hat \varphi$ and $\hat\varphi'$ are compatible res-charts for $[X;\overline{X_a}].$

(Although we started our induction with the case of depth zero, we note that the proof thus far could be used to establish the case of depth one as the base of the induction.)

Now let us consider a stratum $[X;\overline{X_a}]_b,$ $b\neq a.$
For $b\in A_{\mathrm{sing}},$ the pull-back of a sing-chart from $X$ to $[X;\overline{X_a}]$ is again a sing-chart. Indeed, a sing-chart centered at a point $\zeta \in X_b$ is a stratified diffeomorphism of the form $\varphi: \cal U \lra \bb R^\ell \times C(W)$ with $\cal U \subseteq X\setminus\overline{X_a}$ and, since $\beta$ restricts to the identity map between $[X;\overline{X_a}] \setminus [X;\overline{X_a}]_a$ and $X\setminus \overline X_a,$ $\varphi\circ \beta: \beta^{-1}(\cal U) \lra \bb R^\ell \times C(W)$ is a sing-chart. Similarly compatibility between these sing-charts is clear because their overlap takes place in $X\setminus \overline{X_a}$ and so is untouched by the blow-up of $\overline{X_a}.$

It remains to consider a stratum $[X;\overline{X_a}]_b,$ with $b \in A_{\mathrm{res}}.$ If $[X;\overline{X_a}]_b,$ does not intersect the new `front face' $[X;\overline{X_a}]_a = \scr S_a$ then res-charts of $X_b$ lift to res-charts of $[X;\overline{X_a}]_b$ essentially unchanged, so let us assume that $[X;\overline{X_a}]_b \cap \scr S_a \neq \emptyset,$ i.e., that $X_b \cap \overline{X_a}\neq \emptyset.$ Let us denote the fiber bundle of $X_b$ by
\begin{equation*}
	W \fib X_b \xlra{\phi_b} B_bX.
\end{equation*}
The discussion in (i) shows that the restriction of $\phi_b$ to $X_b \cap \overline{X_a}$ still fibers over $B_bX$ and (as we did above) we denote the fiber by $W',$ thus
\begin{equation*}
	W' \fib X_b\cap \overline{X_a} \xlra{\phi_b|_{X_b \cap \overline{X_a}}} B_bX.
\end{equation*}
Since $X_a$ is a minimal sing-stratum in the stratification of $X,$ $W'$ is a minimal sing-stratum in the stratification of $W$ and so our inductive hypothesis applies and we know that $[W;W']$ is a hybrid space. 
It follows that blowing-up $\overline{X_a}$ in $X$ has the effect of blowing-up $W'$ in each fiber of $\phi_b$ and so produces the fiber bundle
\begin{equation*}
	[W;W'] \fib [X;\overline{X_a}]_b \xlra{ \widetilde{\phi_b}:=\phi_b\circ \beta} B_bX.
\end{equation*}
A trivialization of $\phi_b$ lifts to a trivialization of $\widetilde{\phi_b}$ and correspondingly a lift of a res-chart $\varphi: \cal U \lra \bb R^+ \times \cal V \times W$ for $X_b$ lifts to a res-chart $\widetilde{\varphi}: \widetilde{\cal U} \lra \bb R^+ \times \cal V \times [W;W']$ for $[X;\overline{X_a}]_b.$

iv) Finally let us consider the relation between the fiber bundles $\phi_a$ and $\wt\phi_b.$
We have seen how blowing-up $\overline{X_a}$ in $X$ has the effect of blowing-up $W'$ in each fiber $W$ of $\phi_b.$
The `front face' of $[X;\overline{X_a}],$ i.e., the boundary hypersurface produced by the blow=up, $[X;\overline{X_a}]_a,$ thus corresponds to the front face of $[W;W']$ in each fiber of $\wt\phi_b.$ That is to say that $\wt\phi_b$ restricts to a fiber bundle
\begin{equation*}
	\mathrm{ff}[W;W'] \fib [X;\overline{X_a}]_b \cap [X;\overline{X_a}]_a \xlra{ \widetilde{\phi_b}} B_bX.
\end{equation*}
On the other hand, since $\phi_a$ is equal to the restriction of $\beta$ to $[X;\overline{X_a}]_a$ we have
\begin{equation*}
	\phi_a \big( [X;\overline{X_a}]_b \cap [X;\overline{X_a}]_a \big) = X_b \cap \overline{X_a}.
\end{equation*}
Thus we have the commutative diagram
\begin{equation*}
	\xymatrix{
	[X;\overline{X_a}]_a \cap [X;\overline{X_a}]_b \ar[rr]^{\phi_a} \ar[rd]_-{\wt\phi_b} & & \overline{X_a}\cap X_b \ar[ld]^-{\phi_b|_{\overline{X_a}\cap X_b}} \\
	& B_bX & }
\end{equation*}
as required.
\end{proof}

If $X\lra A$ is a smooth atlas stratified space and then repeated applications of this proposition produce {\bf the resolution of $X,$} $\mathrm{res}(X),$ 
a smooth manifold with corners with a boundary stratification parametrized by $A.$
That is to say,
\begin{equation*}
	\mathrm{res}(X) := [\cdots [X;\overline{X_{a_0}}]; \overline{X_{a_1}}]; \cdots; \overline{X_{a_n}}],
\end{equation*}
where $a_0, \ldots, a_n$ is any non-decreasing enumeration of $A.$ (It is easy to see that the result is independent of the order, up to a fibered corner
diffeomorphism.)
The natural map
\begin{equation*}
	\beta_X: \mathrm{res}(X) \lra X,
\end{equation*}
known as the blow-down map, is the composition of the blow-down maps of the individual blow-ups.

\begin{thm} \label{thm:ResSASS}
Whenever $X\lra A$ is a smooth atlas stratified space, its resolution $\mathrm{res}(X)$ inherits from $X$ the structure of a manifold with fibered corners with boundary stratification parametrized by $A.$
The restriction of $\beta_X: \mathrm{res}(X) \lra X,$  the blow-down map, to $\mathrm{res}(X)^{\circ}$ is a diffeomorphism onto $X_{\mathrm{reg}}.$
Moreover, for each $a \in A$ we have
\begin{equation}\label{eq:paares}
	\pa_a\mathrm{res}(X) = \beta_X^{-1}(\overline{X_a}),
\end{equation}
and, if $Z$ is the link of $X_a$ in $X,$ then the corresponding fiber bundle satisfies
\begin{equation*}
	F_a\mathrm{res}(X) = \mathrm{res}(Z) \fib \pa_a\mathrm{res}(X) \xlra{\phi_a} B_a\mathrm{res}(X) = \mathrm{res}(\overline{X_a}), \quad
	\phi_a|_{(\pa_a\mathrm{res}(X))^{\circ}} = \beta_X|_{(\pa_a\mathrm{res}(X))^{\circ}}.
\end{equation*}
Whenever $M \lra A$ is a manifold with fibered corners whose boundary stratification is parametrized by $A$ there is a smooth atlas stratified space $X\lra A$
such that $M = \mathrm{res}(X).$

Next,
suppose $X' \lra A'$ is another smooth atlas stratified space and $f:X \lra X'$ is a stratified map such that
\begin{equation}\label{eq:PreservingReg}
	f(X_{\mathrm{reg}}) \subseteq X'_{\mathrm{reg}}.
\end{equation}
Then $f$ is a smooth stratified map if and only if there exists a smooth fibered corners map $\wt f: \mathrm{res}(X) \lra \mathrm{res}(X'),$ called the `lift' of $f,$ participating in a commutative diagram
\begin{equation}\label{eq:thmLiftingDiag}
	\xymatrix{
	\mathrm{res}(X) \ar[r]^-{\wt f} \ar[d]_-{\beta_{X}} & \mathrm{res}(X')  \ar[d]^-{\beta_{X'}} \\
	X \ar[r]^-f & X' }
\end{equation}
Every smooth fibered corners map satisfying 
\begin{equation}\label{eq:PreservingInt}
	\wt f(\mathrm{res}(X)^{\circ}) \subseteq \mathrm{res}(X')^{\circ}
\end{equation}
is the lift of one and only one smooth stratified map.

It follows that resolution induces an equivalence of categories between smooth atlas stratified spaces and smooth stratified maps satisfying \eqref{eq:PreservingReg} and manifolds with fibered corners and smooth fibered corners maps satisfying \eqref{eq:PreservingInt}.
\end{thm}

\begin{proof}
The statements about the structure of $\mathrm{res}(X)$ follow immediately from the proposition.
So too does that fact that every manifold with fibered corners $M$ is the resolution of a smooth atlas stratified space. 
Indeed note that blowing-up a minimal singular stratum of a hybrid space is reversible. So starting with $M$ we can choose a maximal boundary stratum
and blow it down to obtain a hybrid space with one singular stratum. Iterating until we have blown-down all of the res-strata produces a smooth atlas stratified space whose resolution is, manifestly, equal to $M.$

Next let us consider the statements about lifting maps. First note that two maps $\wt f$ and $f$ participating in \eqref{eq:thmLiftingDiag} mutually determine
each other since the blow-down map identifies the interior of the manifold with corners with the regular part of the stratified space and these are each
dense in their corresponding space. Thus if a map has a lift, it is unique.

Given a smooth stratified map $f:X \lra X'$ and a minimal stratum $X_a$ of $X$ there is a stratum $X'_{a'}$ of $X'$ such that $f(X_a)\subseteq X'_{a'}.$
It is easy to see that $f$ will lift to a smooth map $\wt f: [X;\overline{X_a}] \lra [X';X'_{a'}].$ (We assume that $X_a$ and $X'_{a'}$ are singular strata for the sake of concreteness.)
Indeed, denote the blow-down maps by $\beta: [X;\overline{X_a}]\lra X$ and $\beta':[X';X'_{a'}]\lra X'$ and suppose $\cal U \xlra{\varphi} \bb R^h \times C(Z)$ and $\cal U' \xlra{\varphi'} \bb R^{h'} \times C(Z')$ are coordinate charts of
points in $X_a,$ $X'_{a'},$ respectively, such that $f(\cal U) \subseteq \cal U'.$ Each of these charts induce a chart of the corresponding blown-up space,
$\wt{\cal U} \xlra{\wt\varphi} \bb R^h \times \bb R^+ \times Z$ and  $\wt{\cal U'} \xlra{\wt\varphi'} \bb R^h \times \bb R^+ \times Z',$ and by definition of
smoothness there is a lift $\wt f_{\cal U}$ of $f|_{\cal U}$ fitting into a commutative diagram
\begin{equation*}
	\xymatrix{
	\wt{\cal U} \ar[ddd]_-\beta \ar[rd]_-{\wt{\varphi}} \ar[rrr]^-{\wt f_{\cal U}} & & & \wt{\cal U'} \ar[ddd]^-{\beta'}\\
	&  \bb R^h \times \bb R^+ \times Z \ar[r] \ar[d] & \bb R^h \times \bb R^+ \times Z' \ar[ru]_-{\wt{\varphi'}^{-1}} \ar[d] & \\
	&  \bb R^h \times C(Z) \ar[r] & \bb R^h \times C(Z') \ar[rd]^-{{\varphi'}^{-1}} & \\
	\cal U \ar[ru]^-{\varphi} \ar[rrr]^-{f|_{\cal U}} & & & \cal U' }
\end{equation*}
From the definition of smooth atlas we see that this lift of $f|_{\cal U}$ is independent of the particular coordinate chart used in its definition and hence lifting $f$ one coordinate chart at a time produces a global lift $\wt f.$

If $X$ and $X'$ are hybrid spaces then the analogous argument, with greater notational complexity, shows that $f$ lifts to $\wt f: [X;\overline{X_a}] \lra [X';X'_{a'}].$ 
In this way, lifting one blow-up at a time, we see that smooth stratified maps satisfying \eqref{eq:PreservingReg} lift to maps of the resolutions.
(Note that if \eqref{eq:PreservingReg} does not hold then the analogous result is usually false. Indeed if $f(X_{\mathrm{reg}}) \subseteq X'_{a'}$ then lifting
$f$ would require a smooth $\wt f$ satisfying
\begin{equation*}
	\xymatrix{
	& \pa_{a'}\mathrm{res}(X') \ar[d]^{\phi_{a'}} \\
	\mathrm{res}(X) \ar[ru]^-{\wt f} \ar[r]^-f & \mathrm{res}(\overline{X'_{a'}}) }
\end{equation*}
and this is clearly topologically obstructed.)

On the other hand given a smooth fibered corners map $\wt f: \mathrm{res}(X)\lra \mathrm{res}(X'),$ the definition of smoothness requires that $\wt f$
sends the fibers of a boundary fiber bundle of $X$ to the fibers of the corresponding boundary fiber bundle of $X'$ and so it descends to a map between
the spaces obtained by blowing-down these boundary faces. Successively blowing-down all of the boundary fiber bundles we end up with a map between
$X$ and $X'$ such that \eqref{eq:thmLiftingDiag} commutes.

Finally note that uniqueness implies that the lift of a composition is the composition of the lifts. Thus resolution and blow-down are mutually inverse functors
of categories between smooth atlas stratified spaces and smooth stratified maps satisfying \eqref{eq:PreservingReg} and manifolds with fibered corners and smooth fibered corners maps satisfying \eqref{eq:PreservingInt}.
\end{proof}

The idea of resolving a stratified space to a manifold with corners goes back to Thom \cite{Tho:EMS}. 
For Thom-Mather stratified spaces (which we will discuss next) resolution was worked out by Verona \cite[\S 6.3.1]{Ver:SMT}.
That one obtains a manifold with fibered corners was pointed out by Melrose, see \cite[\S 2]{AlbLeiMazPia:SPWS} where this is discussed following \cite{BraHecSar:TRPVS}.

Resolution of `conically smooth' spaces  is studied in \cite[\S 7]{AyaFraTan:LSSS} where it is referred to as "Unzip".
In \cite[Theorem 7.3.8]{AyaFraTan:LSSS} the authors establish that Unzip is a functor, however see Remark \ref{rem:Ayala}.\\

Below, we will find it useful to identify the $C^*$-algebra of continuous functions on a smooth atlas stratified space $X$ with a $C^*$-subalgebra
of the continuous functions on $\mathrm{res}(X).$
Specifically, we define
\begin{equation}\label{eq:DefCPhi}
	\cal C_{\Phi}(\mathrm{res}(X)) = \{ f \in \cal C(\mathrm{res}(X)): 
	\text{ for all } a \in S(X) \text{ we have } f|_{\pa_{a}X} \in \phi_a^*\cal C(B_{a}X) \}
\end{equation}
and it is easy to see that pull-back along the blow-down map $\beta_X:\mathrm{res}(X) \lra X$ establishes an isomorphism of $*$-algebras between 
$\cal C(X)$ and $\cal C_{\Phi}(\mathrm{res}(X)).$ We see that the smooth functions $X\lra \bb C$ are precisely the functions in $\cal C(X)$ that pull-back
along $\beta_X$ to be in $\cal C_{\Phi}^{\infty}(\mathrm{res}(X))  := \cal C_{\Phi}(\mathrm{res}(X)) \cap \cal C^{\infty}(\mathrm{res}(X)).$

\subsection{Smooth atlas stratified spaces and Thom-Mather stratified spaces} \label{subsec:SASSThomMather}

This section will not be used in the rest of the paper.
In this subsection we will show that the class of spaces we defined in Definition \ref{def:SASS} coincide with Thom-Mather stratified spaces.
We will first show that Thom-Mather stratified spaces are examples of smooth atlas stratified spaces. The converse is easy thanks to the resolution
process.

\begin{defn}
A {\bf Thom-Mather stratified space} consists of a triple $(X, \cal S, \cal T)$ where:
\begin{enumerate}[i)]
\item $X$ is a Hausdorff, locally compact topological space with a countable basis for its topology.
\item $\cal S$ is a locally finite family of locally closed subsets of $X,$ known as the strata of $X,$ each of which is a smooth manifold. 
The strata cover $X$ and satisfy the frontier condition
\begin{equation*}
	L,M \in \cal S \text{ and } L \cap \overline M\neq \emptyset
	\implies L\subseteq \overline M.
\end{equation*}
\item $\cal T$ contains, for each $M \in \cal S,$ a triple $(T_M, \pi_M, \rho_M)$ in which $T_M$ is an open neighborhood of $M$ in $X,$ and
\begin{equation*}
	\pi_M: T_M \lra M, \quad
	\rho_M: T_M \lra [0,\infty)
\end{equation*}
are continuous functions such that $\pi_M|_M = \mathrm{id}_M$ and $\rho_M^{-1}(0) = M.$ The neighborhood $T_M$ is known as a tubular neighborhood
of $M$ in $X$ and the triple $(T_M, \pi_M, \rho_M)$ as the control data of $M$ in $X,$
\end{enumerate}
The control data are compatible in that, if $L$ and $M$ are distinct strata such that $L \subseteq \overline M$ and $T_{L,M} = T_L\cap M,$ $\pi_{L,M} = \pi_L|_{T_{L,M}},$ $\rho_{L,M} = \rho_L|_{T_{L,M}},$ then
\begin{equation*}
	(\pi_{L,M}, \rho_{L,M}): T_{L,M} \lra L \times (0,\infty)
\end{equation*}
is a smooth submersion and, if $N$ is a third distinct stratum such that $M \subseteq \overline N,$ then for any 
$\zeta \in T_{L,M}\cap T_{L,N} \cap \pi_{M,N}^{-1}(T_{L,M})$ we have
\begin{equation*}
	\pi_{L,M}\pi_{M,N}(\zeta) = \pi_{L,N}(\zeta), \quad
	\rho_{L,M}\pi_{M,N}(\zeta) = \rho_{L,N}(\zeta).
\end{equation*}
\end{defn}

\begin{remark}\label{rk:counterxample-product}
As mentioned above, the main reason we work with smooth atlas stratified spaces is that they make it easy to work with fiber bundles of stratified spaces.
An example of the complications that arise when working with Thom-Mather stratified spaces is that, while the product of Thom-Mather stratified spaces is
again a Thom-Mather stratified space, {\em the product of the Thom-Mather data is not Thom-Mather data for the product} (pointed out in, e.g., 
\cite[\S 1.2.9]{Ver:SMT}). For a simple example, consider $Y = [0,\infty)_x$ as a Thom-Mather stratified space with strata $\{0\}$ and $(0,\infty),$ and with 
Thom-Mather data
\begin{equation*}
\begin{gathered}
	\cal T_{\{0\}} = [0,1), \quad \pi_{\{0\}}(x) = 0, \quad \rho_{\{0\}}(x) = x, \\
	\cal T_{(0,\infty)} = (0,\infty), \quad \pi_{(0,\infty)}(x) = x, \quad \rho_{(0,\infty)} = 0.
\end{gathered}
\end{equation*}
Now on $X = Y^2$ with the product stratification, the product Thom-Mather data would be
\begin{equation*}
\begin{aligned}
	\cal T_{(0,0)} = [0,1)_x \times [0,1)_{x'}, 
		\quad \pi_{(0,0)}(x,x') &= (0,0), 
		\quad \rho_{(0,0)}(x,x') = x+x', \\
	\cal T_{\{0\} \times (0,\infty)} = [0,1)_x \times (0,\infty)_{x'}, 
		\quad \pi_{\{0\}\times (0,\infty)}(x,x') &= (0,x'), 
		\quad \rho_{\{0\}\times (0,\infty)}(x,x') = x, \\
	\cal T_{(0,\infty)\times \{0\}} = (0,\infty)_{x}\times [0,1)_{x'}, 
		\quad \pi_{(0,\infty)\times\{0\}}(x,x') &= (x,0), 
		\quad \rho_{(0,\infty)\times \{0\}}(x,x') = x', \\
	\cal T_{(0,\infty)^2} = (0,\infty)^2, 
		\quad \pi_{(0,\infty)}(x,x') &= (x,x'), 
		\quad \rho_{(0,\infty)^2} = 0
\end{aligned}
\end{equation*}
and the problem is that this data does not satisfy the compatibility condition
\begin{equation*}
	\rho_{(0,0)} \circ \pi_{\{0\}\times (0,\infty)} = \rho_{(0,0)} \text{ on } \cal T_{(0,0)}\cap \cal T_{\{0\} \times (0,\infty)},
\end{equation*}
as this would imply, e.g., that $\rho_{(0,0)}(0,x') = \rho_{(0,0)}(x,x')$ for all $(x,x')$ in this set. This makes the definition of a fiber bundle of Thom-Mather stratified
spaces inconvenient.
\end{remark}

Given a Thom-Mather stratified space $(X, \cal S, \cal T),$ we can naturally consider $\cal S$ as a poset with $L \leq M \iff L \subseteq \overline M$ and 
$X$ as poset stratified by the obvious map $X \lra \cal S.$ 

A {\bf controlled map} between Thom-Mather stratified spaces $(X, \cal S_X, \cal T_X)$ and $(Y, \cal S_Y, \cal T_Y)$ is a continuous stratified map
$f$ between $X \lra \cal S_X$ and $Y\lra \cal S_Y$ such that, whenever $f(X_a) \subseteq Y_b,$ the restriction
\begin{equation*}
	f|_{X_a}: X_a \lra Y_b 
\end{equation*}
is a smooth map and (possibly after shrinking the size of the tubular neighborhoods) we have
\begin{equation*}
	f(T_{X_a})\subseteq T_{Y_b}, \quad
	f\circ \pi_{X_a} = \pi_{Y_b} \circ f, \quad
	\rho_{X_a} = \rho_{Y_b}\circ f.
\end{equation*}
A controlled isomorphism is a bijective controlled map whose inverse is also controlled, hence it is in particular a homeomorphism that restricts
to a diffeomorphism on each stratum.

An important property of controlled isomorphisms, identified by Brasselet-Hector-Saralegi \cite{BraHecSar:TRPVS}  is that they lift under `d\'eplissage'.
In particular this means that for every controlled isomorphism 
\begin{equation*}
	f: \bb R^{h_1} \times C(Z_1) \lra \bb R^{h_2} \times C(Z_2)
\end{equation*}
there a  controlled isomorphism 
\begin{equation*}
	\wt f: \bb R^{h_1} \times \bb R \times Z_1 \lra \bb R^{h_2} \times \bb R \times Z_2,
\end{equation*}
that makes the diagram
\begin{equation*}
	\xymatrix{
	 \bb R^{h_1} \times \bb R \times Z_1 \ar[r]^-{\wt f} & \bb R^{h_2} \times \bb R \times Z_2 \\
	 \bb R^{h_1} \times \bb R^+ \times Z_1 \ar@{^(->}[u] \ar[d] & \bb R^{h_2} \times \bb R^+ \times Z_2  \ar@{^(->}[u] \ar[d]  \\
	 \bb R^{h_1} \times C(Z_1) \ar[r]^-f & \bb R^{h_2} \times C(Z_2), }
\end{equation*}
where the descending vertical arrows are the quotient maps of the cones, commutative.
(We have taken this property and, for not necessarily bijective maps, made it the definition of smooth stratified maps in the setting of smooth atlas stratified spaces above.)

\begin{thm}
Thom-Mather stratified spaces admit smooth stratified atlases such that controlled isomorphisms are smooth stratified diffeomorphisms.

Smooth atlas stratified spaces admit control data giving them the structure of a Thom-Mather stratified space.
\end{thm}

\begin{proof}
Since the statement is obvious for Thom-Mather stratified spaces of depth zero, let us assume inductively that we have established the statement of the proposition for stratified spaces of depth less than some $\ell \in \bb N.$\\

Let us recall some of the well-known consequences of Thom's first isotopy lemma \cite[Theorem 2.6]{Ver:SMT} when applied to the Thom-Mather 
control data. This lemma says that every proper continuous map $F:X \lra M$ between a Thom-Mather stratified space $X$ and a smooth manifold $M,$
whose restriction to each stratum of $X$ is a smooth submersion, is a locally trivial fibration. Thus such a map participates in a fiber bundle
whose fiber is a Thom-Mather stratified space and with structure group given by the controlled isomorphisms of the fiber.

If $M$ is a stratum of a Thom-Mather stratified space with control data $(T_M, \pi_M, \rho_M)$ then 
(possibly after shrinking $T_M$ and scaling $\rho_M$) we can apply the isotopy lemma to recognize
\begin{equation*}
	F= (\pi_M, \rho_M)|_{T_M \setminus M}: T_M\setminus M \lra M \times (0,1),
\end{equation*}
which the definition of Thom-Mather spaces assures us is a smooth submersion when restricted to each stratum of $T_M\setminus M,$
as a locally trivial fibration. If we denote the restriction of $\pi_M$ to $H = \rho_M^{-1}(\tfrac12)$ by $\phi_M:H \lra M$ then 
we can identify the typical fiber of $F$ with $H \times (0,1),$ recognize that $\phi_H$ is a locally trivial fiber bundle, and, if we denote the typical fiber
of $\phi_H$ by $Z,$ finally conclude that $\pi_M: T_M \lra M$ is a locally trivial fiber bundle with fiber $C_{[0,1)}(Z).$

In more detail, for every point $p \in M$ there is a neighborhood $\cal V \subseteq M$ and a stratified homeomorphism
$\psi: \pi_M^{-1}(\cal V) \lra \cal V \times C_{[0,1)}(Z)$ participating in the commutative diagram
\begin{equation*}
	\xymatrix{
	\pi_M^{-1}(\cal V) \ar[rr]^-{\psi} \ar[rd]_-{\pi_M} & & \cal V \times C_{[0,1)}(Z) \ar[ld] \\ & \cal V & }
\end{equation*}
where the unlabeled arrow is the obvious map. If $\psi': \pi_M^{-1}(\cal V) \lra \cal V \times C_{[0,1)}(Z)$ is another such homeomorphism then
the transition map
\begin{equation*}
	\psi'\circ \psi^{-1}: \cal V \times C_{[0,1)}(Z) \lra \cal V \times C_{[0,1)}(Z) 
\end{equation*}
is the identity on the first factor and is a controlled isomorphism on the second factor. As mentioned above, this controlled isomorphism lifts to
a controlled isomorphism on the `d\'eplissage' and so, if the depth of $Z$ is less than $\ell,$ a stratified diffeomorphism thanks to the inductive
hypothesis.

Thus if $(X, \cal S, \cal T)$ is a Thom-Mather stratified space of depth at most $\ell$ then we obtain stratified charts from the trivializing neighborhoods of
these fiber bundle structures on the tubular neighborhoods of the strata of $X,$ together with the coordinate charts of the regular part of $X.$ The analysis
above shows that this is a smooth stratified atlas. Finally this same analysis shows that a controlled isomorphism is a smooth stratified diffeomorphism.

Next if we start with a smooth atlas stratified space $X\lra A$ we can obtain control data by using its resolution $\beta:\mathrm{res}(X) \lra X.$
This goes back at least to \cite[\S 5-6]{Ver:SMT} but a nice exposition can be found in \cite[Proposition A.4]{KotRoc:PMWFC}, to which we refer the reader for details.
\end{proof}

\begin{remark}\label{rem:AyalaEtAl}
Note that the unzip result in \cite[\S 7]{AyaFraTan:LSSS} should imply a categorical equivalence between the category of conically smooth spaces and stratified diffeomorphisms and the category of manifolds with fibered corners and fibered diffeomorphisms; however if this were true then the conically smooth maps would coincide with our smooth maps between smooth atlas stratified spaces and Remark \ref{rem:Ayala} shows that this is not the case. (In particular \cite[Conjecture 1.5.3]{AyaFraTan:LSSS} is true and well-known for Thom-Mather stratified spaces and hence for smooth atlas stratified spaces.)

In \cite{NocVol:WSCS}, it is shown that Thom-Mather stratified spaces are conically smooth spaces.
The authors do not make use of the resolution studied in \cite[\S 7]{AyaFraTan:LSSS} to relate the two types of spaces, perhaps because they
were unaware of the parallel resolution result in \cite{Ver:SMT} for Thom-Mather spaces.
\end{remark}

In summary, we now know that
\begin{enumerate}
\item Thom-Mather stratified spaces,
\item Whitney embedded stratified spaces,
\item manifolds with fibered corners,
\item smooth atlas stratified spaces,
\end{enumerate}
all describe the same singular spaces, which we refer to as smoothly stratified spaces.

\subsection{Witt spaces} \label{sec:WittSpaces}

In this short subsection we recall the main objects we will be studying throughout the rest of the paper.

\begin{defn} [{Witt spaces \cite[Definition 2.1]{Sie:WSGCTKOP}}] $ $

A smoothly stratified space  $X$ is a {\bf Witt space} if whenever $Z$ is the link in $X$ of a singular stratum we either have that $\ell = \dim Z$ is odd or that $\mathrm{IH}_{\ell/2}^{\overline m}(Z; \bb Q) = 0,$ where $\mathrm{IH}^{\overline m}_*(Z;\bb Q)$ denotes the lower middle perversity intersection homology groups of $Z$ with rational coefficients \cite[\S 5.1]{GorMac:IHT}.
\end{defn}

Siegel introduced Witt spaces, though he worked with PL stratified spaces. He introduced a bordism theory of {oriented} PL Witt 
pseudomanifolds\footnote{Recall from \S\ref{subsect:atlas} that a pseudomanifold is a stratified space whose regular part is dense and has a fixed dimension and whose other strata have dimension at least two less than the regular part.}, $\Omega^\Witt_* (-),$ which is now well known to be a homology theory. The name `Witt spaces' comes from the fact that the bordism groups of a point can be computed to essentially equal the Witt groups of $\bb Q;$ specifically, \cite[pg. 1068]{Sie:WSGCTKOP}
\begin{equation*}
	\Omega^\Witt_q (\mathrm{pt}) = \begin{cases}
	\bb Z & \text{ if } q = 0 \\
	0 & \text{ if } q \not\equiv 0 \mod 4 \\
	\mathrm{Witt}(\bb Q) & \text{ if } q>0, \; q \equiv 0 \mod 4
	\end{cases}
\end{equation*}
One can also consider the bordism theory based on cycles given by continuous maps on smoothly stratified {oriented} Witt pseudomanifolds, $\Omega^{\Witt,\infty}_* (-).$ 
We show in Proposition \ref{prop.bordofsmoothlystratandPLwitt} that these are naturally equivalent.)

Natural examples of {oriented} Witt pseudomanifolds include 
{pure-dimensional} complex algebraic varieties, or 
indeed any stratified space all of whose strata are even-codimensional.

\section{The analytic orientation class of a Witt pseudomanifold} \label{sec:orientation}

\subsection{K-homology orientations}\label{subsect:general-orientation}

Let $E$ be a ring spectrum.
We review $E$-homological orientation classes of manifolds, see e.g.
Rudyak \cite{Rud:TSOC}.
First, spheres $S^n$ are $E$-oriented as follows.
Let $1\in \pi_0 (E)$ be the unit of the homotopy ring $\pi_* (E)$.
The homology theory $E_*$ determined by $E$ has a 
suspension isomorphism
\[ \sigma_*: \pi_0 (E) = \widetilde{E}_0 (S^0) 
  \stackrel{\simeq}{\longrightarrow}
  \widetilde{E}_n (S^n) = E_n (S^n,*). \]
One defines the orientation of $S^n$ to be
$[S^n]_E := \sigma_* (1).$
Now let $M$ be an $n$-dimensional topological manifold, say without boundary.
For every point $p \in M$ and every disc neighborhood
$U \subset M$ of $p$, let
$\operatorname{coll}^{p,U}: M \to S^n$ be the map which collapses the complement
of $U$. If $U'$ is another disc neighborhood of $p$, then
$\operatorname{coll}^{p,U'}$ and $\operatorname{coll}^{p,U}$ are homotopic.
Thus we obtain a well-defined homotopy class
\[ \operatorname{coll}^p: M \to S^n. \]

\begin{defn}
An element $[M]_E \in E_n (M)$ is called an {\bf $E$-orientation of $M$},
if
\begin{equation} \label{equ.ehomolorientation} 
\operatorname{coll}^{p}_* [M]_E = \pm [S^n]_E 
\end{equation}
for every $p\in M$.
\end{defn}

Suppose that $M$ is connected.
Then the homotopy classes $\operatorname{coll}^p$
and $\operatorname{coll}^{p'}$ coincide for any two points $p,p' \in M$
by choosing a path from $p$ to $p'$.
So it suffices for connected manifolds to verify
(\ref{equ.ehomolorientation}) at a single point $p$.
If $M$ is $E$-oriented, then the Poincar\'e duality map
\[  -\cap [M]_E: E^k (M) \longrightarrow E_{n-k} (M) \]
is an isomorphism.
\begin{prop} \label{prop.eorientgoestoforient}
Let $\phi:E\to F$ be a ring morphism of ring spectra and 
$M$ an $n$-dimensional manifold. If $[M]_E \in E_n (M)$
is an $E$-orientation of $M$, then its image $\phi_* [M]_E \in F_n (M)$
is an $F$-orientation of $M$.
\end{prop}
\begin{proof}
The morphism $\phi$ induces a natural transformation $\phi_*: E_* (-)\to F_* (-)$ 
of homology	theories. This transformation commutes with the induced homomorphisms
$\operatorname{coll}^p_*$, and with the suspension isomorphisms.
Furthermore, it maps the unit $1\in \widetilde{E}_0 (S^0)$ to the unit
$1\in \widetilde{F}_0 (S^0)$. Therefore, for every $p\in M,$
\[ \operatorname{coll}^p_* \phi_* [M]_E
  =  \phi_* \operatorname{coll}^p_* [M]_E 
  = \pm \phi_* [S^n]_E = \pm \phi_* \sigma_{E,*} (1)
  = \pm \sigma_{F,*} \phi_* (1) = \pm \sigma_{F,*} (1) 
  = \pm [S^n]_F.   \]
\end{proof}	
Since complexification and the stable second Adams operation are multiplicative,
Proposition \ref{prop.eorientgoestoforient} together with
Theorem \ref{thm.sullivanandksignoponmanifolds} implies that the 
(normalized) class $\sign_K (M)$ of the signature operator
 is a $\K [\smlhf]$-orientation for a closed smooth
$H\intg$-oriented manifold $M$.
(It is not an orientation at $2$.)

\subsection{Wedge metrics and Dirac-type operators} \label{sec:WedgeMetsDiracOps}

We recall the definition of wedge Dirac-type operators from \cite{AlbGel:IFFDTOP}.
The example most important to us will be the signature operator of a wedge metric acting on wedge differential forms.

Let $M$ be a manifold with fibered corners and let $M \lra S(M)$ denote the corresponding boundary stratification.
Wedge one forms are defined by
\begin{equation*}
	\cal V_w^* 
	= \{ \theta \in \cal C^{\infty}(M; T^*M): \text{ for each $a\in S(M)$ and $y \in B_aM,$} \quad
		\theta|_{\phi_a^{-1}(y)}=0\}
\end{equation*}
and by the Serre-Swan theorem or by direct construction as in \cite[\S 8.1-8.2]{Mel:AIT} there is a vector bundle
\begin{equation*}
	{}^wT^*M \lra M
\end{equation*}
endowed with a bundle map $i: {}^wT^*M \lra T^*M$ with the property that
\begin{equation*}
	i_*\cal C^{\infty}(M; {}^wT^*M) = \cal V_w^* \subseteq \cal C^{\infty}(M; T^*M).
\end{equation*}
In particular, $i$ is a bundle isomorphism over $M^{\circ}$ and hence 
\begin{equation*}
	{}^wT^*M|_{M^{\circ}} \text{ is canonically isomorphic to } T^*(M^{\circ}).
\end{equation*}
Since $M$ is homotopic to $M^{\circ},$ it follows that ${}^wT^*M$ is isomorphic to $T^*M$ over $M$ 
(although not naturally and $i$ is not an isomorphism over the boundary of $M$).
In local coordinates near the total space of $F_aM \fib \pa_aM \xlra{\phi_a} B_aM,$ 
$(x,y_1, \ldots, y_h,z_1, \ldots, z_v),$ in which $x$ is a boundary defining function for $\pa_aM,$ 
$y_j$ are coordinates pulled-back from $B_aM,$ and $z_k$ are coordinates on $F_aM,$ 
a differential form
\begin{equation*}
	\theta = a(x,y,z) dx + b^j(x,y,z) dy_j + c^k(x,y,z) dz_k
\end{equation*}
is in $\cal V_w^*$ if and only if the coefficients $c^k$ vanish when $x=0$ and so the differentials $dx,$ $dy_1, \ldots, dy_h,$ $x\; dz_1, \ldots, x\; dz_v$ make up a local
frame for $\cal V_w^*.$ 
Abusing notation we think of elements of $\cal V_W^*$ as sections of ${}^wT^*M.$ Importantly though, a form like $x\; dz_1$ does not vanish at $\{x=0\}$ when thought
of as a section of ${}^wT^*M$ (e.g., because $dz_1$ is not a section of ${}^wT^*M$ so $x\;dz_1$ is not $x$ times another section).

The wedge tangent bundle ${}^w TM$ is the dual bundle to the wedge cotangent bundle. A general wedge metric can be defined as a bundle metric on ${}^wTM$ 
but it is convenient to require more structure. Namely, a product-type wedge metric is one that near each $\pa_aM$ has the form
\begin{equation*}
	dx^2 + x^2 g_{\pa_aM/B_aM} + \phi_a^*g_{B_aM}
\end{equation*}
where $g_{\pa_aM/B_aM} + \phi_a^*g_{B_aM}$ is a submersion metric on $\pa_aM.$
A totally geodesic wedge metric is one that coincides with a product-type wedge metric at each $\pa_aM$ up to $\cal O(x^2).$ (See \cite[\S1.2]{AlbGel:IFFDTOP}
for more details, and also \cite[\S6]{KotRoc:PMWFC}.)

For totally geodesic wedge metrics (and more generally) one can use the Koszul formula to see that the Levi-Civita connection defines a connection
on the wedge tangent bundle,
\begin{equation*}
	\nabla^M: \cal C^{\infty}(M; {}^w TM) \lra \cal C^{\infty}(M; T^*M \otimes {}^w TM),
\end{equation*}
and hence also on the wedge cotangent bundle.

\begin{defn}
Let $M$ be a manifold with fibered corners and a wedge Riemannian metric $g_M.$
A {\bf wedge Clifford module} on $M$ consists of a complex vector bundle $E \lra M$ endowed with a Hermitian metric $g_E,$ a metric connection $\nabla^E,$
and a bundle homomorphism (referred to as Clifford multiplication)
\begin{equation*}
	\cl:\bb C \otimes \mathrm{Cl}({}^wT^*M,g_M) \lra \mathrm{End}(E)
\end{equation*}
satisfying that, for any $\theta \in \cal C^{\infty}(M;{}^wT^*M)$ and $V \in \cal C^{\infty}(M;TM),$
\begin{equation*}
	g_{E}(\cl(\theta)\cdot, \cdot) = - g_E(\cdot, \cl(\theta)\cdot)
	\quad \text{ and } \quad
	\nabla^E_V\cl(\theta) = \cl(\theta)\nabla^E_V + \cl(\nabla^M_V\theta).
\end{equation*}
The corresponding Dirac-type operator is the differential operator given by the composition
\begin{equation*}
	\eth^E:\cal C^{\infty}_c(M^{\circ};E) \xlra{\nabla^E}
	\cal C^{\infty}_c(M^{\circ}; T^*M \otimes E) \xlra{\cl}
	\cal C^{\infty}_c(M^{\circ};E),
\end{equation*}
where we have used that $T^*M$ and ${}^wT^*M$ are canonically isomorphic over the interior of $M.$

A $\bb Z_2$-graded wedge Clifford module is a wedge Clifford module in which $E$ has a splitting $E = E^+\oplus E^-$
which is orthogonal with respect to $g_E,$ parallel with respect to $\nabla^E,$ and odd with respect to $\cl.$
\end{defn}

Using the metrics on $M$ and on $E$ we can define $L^2(M;E)$ and consider $\eth^E$ as a symmetric 
unbounded operator with initial domain $\cal C^{\infty}_c(M^{\circ};E).$ Elliptic regularity in this setting establishes that any element
of $L^2(M;E)$ that is mapped into $L^2(M;E)$ by $\eth^E$ (acting distributionally) must have Sobolev regularity but measured
using ``edge vector fields". An edge vector field on $M$ is one that is tangent to the fibers of all of the boundary fiber bundles of the
iterated fibration structure on $M$ and we denote the corresponding Sobolev space by $H^1_e(X;E).$
Thus the elliptic regularity statement is that
\begin{equation*}
	\cal D_{\max}(\eth^E) := \{ u \in L^2(M;E): \eth^Eu \in L^2(M;E)\} \subseteq H^1_e(X;E).
\end{equation*}
For our purposes, the initial domain $\cal C^{\infty}_c(M^{\circ};E)$ of $\eth^E$ is too small and the maximal domain, $\cal D_{\max}(\eth^E),$ is too large
so we use an intermediate domain known as the `vertical APS domain' (so called because of its close relation to the domain introduced by Atiyah-Patodi-Singer in \cite{AtiPatSin:SARG}).
This domain is defined as 
\begin{equation*}
	\cal D_{VAPS}(\eth^E) := \rho^{1/2}H^1_e(X;E) \cap \cal D_{\max}(\eth^E)
\end{equation*}
where $\rho$ is a `total boundary defining function', i.e., a smooth non-negative function that vanishes simply at the boundary of $M.$
In \cite{AlbGel:IFFDTOP} this was shown to be a natural choice of domain for $\eth^E$ and to have several nice properties
as long as an `analytic Witt assumption' is satisfied. 

To state this assumption, note that for any $\alpha \in S(M),$ a wedge Dirac-type operator $\eth^E$ takes the form
\begin{equation*}
	\frac1x( a(x,y,z) x\pa_x + b^i(x,y,z) x\pa_{y_j} + c^k(x,y,z) \pa_{z_k} + f(x,y,z))
\end{equation*}
near $\pa_{\alpha}M = \{x=0\}.$ We refer to
\begin{equation*}
	\eth^E_{\alpha} = x\eth^E|_{\pa_{\alpha}M} = c^k(0,y,z)\pa_{z_k} + f(0,y,z)
\end{equation*}
as the {\bf boundary family} of $\eth^E$ at $\pa_{\alpha}M.$ It is a family of operators on the fibers of $\phi_{\alpha}:\pa_{\alpha}M \lra B_{\alpha}M.$
In fact it is practically equal to a family of Dirac-type operators with respect to an induced wedge Clifford module, see \cite[Lemma 2.2]{AlbGel:IFFDTOP}.
We say that $\eth^E$ satisfies the {\bf analytic Witt condition} if every boundary family, endowed with its vertical APS domain, is invertible. 

In \cite[Theorem 1]{AlbGel:IFFDTOP} it is shown that if $M$ is compact and $\eth^E$ satisfies the analytic Witt condition then 
$(\eth^E, \cal D_{VAPS}(\eth^E))$ is closed\footnote{In the statement of \cite[Theorem 1]{AlbGel:IFFDTOP} the vertical APS domain is defined as the graph closure of the domain defined above but the proof of \cite[Theorem 4.3]{AlbGel:IFFDTOP} shows that the domain is already closed in the graph norm.}, Fredholm, and self-adjoint with compact resolvent. In case $M$ is non-compact, but the links of the associated
stratified space are compact, the method described in \cite[\S3]{AlbLeiMazPia:HTCS} and implemented in, e.g.,  the proof of 
\cite[Theorem 5.2]{AlbLeiMazPia:HTCS} shows that $(\eth^E, \cal D_{VAPS}(\eth^E))$ is closed, locally Fredholm, and self-adjoint with locally compact 
resolvent\footnote{By locally compact we mean an operator whose product with any smooth function of compact support is a compact operator and by locally Fredholm we mean an operator that is invertible up to a locally compact operator.}.

The most important special case of this construction for us is the signature operator.
First if we replace the ordinary differential forms on $M$ with the wedge differential forms, defined to be the sections of the exterior powers of the wedge cotangent bundle,
\begin{equation*}
	\cal C^{\infty}_c(M; \Lambda^* ({}^wT^*M)),
\end{equation*}
then the de Rham operator $d+\delta$ is a wedge Dirac-type operator. The analytic Witt assumption is satisfied by $d+\delta$ if and only if $M$ is (the resolution of) a Witt pseudomanifold in the sense of Siegel \cite{Sie:WSGCTKOP} (this is why it is called the analytic Witt assumption); see e.g., \cite[Corollary 4.2]{AlbLeiMazPia:HTCS}.

To define the signature operator we first replace wedge cotangent bundle with its complexification ${}^wT^*_{\bb C}M = {}^wT^*M\otimes \bb C$ and then
let $\star$ be the involution obtained from the Hodge star $*$ by multiplying it with an appropriate (form degree dependent) power of $i$ so that $\star^2=1.$
If $\dim M$ is even then the signature operator $D_M^{\mathrm{sign}} = \eth_{\mathrm{sign}}$ is the operator induced by the de Rham operator together with the $\bb Z_2$-grading induced by $\star.$
If $\dim M$ is odd then the de Rham operator commutes with $\star$ and the signature operator (known as the `odd signature operator') is its restriction to the $+1$ eigenspace of $\star.$ (In either case, the analytic Witt assumption for the signature operator holds precisely when the analytic Witt assumption for the de Rham operator does, i.e., when $M$ is a Witt pseudomanifold.)

Summarizing, we have the following fundamental result:
\begin{thm}Let $X$ be a smoothly stratified Witt pseudomanifold
and let $M=\mathrm{res}(X)$ be its resolution, a manifold with fibered corners. We endow the regular part of 
$X$ and thus the interior of $M$ with a wedge metric $g$. Then
the signature operator $D_M^{\mathrm{sign}}$ associated to $g$
satisfies the analytic Witt condition and therefore admits
a closed self-adjoint extension $\cal D_{VAPS}(\eth_{\mathrm{sign}}))\subset L^2(M,\Lambda^* ({}^wT^*M))$
which is Fredholm and has compact resolvent.
\end{thm}

For some wedge metrics on a Witt space, called adapted iterated edge metrics in \cite[\S 5.4]{AlbLeiMazPia:SPWS}, the de Rham operator is essentially self-adjoint as an unbounded operator on $L^2$ differential forms. Any wedge metric can be deformed to one of this form and this deformation produces a homotopy of the de Rham (or signature) operators within Fredholm operators (see \cite[Remark 4.9]{AlbGel:IFFDTOP}). Thus for some purposes, e.g., defining a K-homology class, one may if convenient assume that the operator is essentially self-adjoint.

\subsection{KK-theory and the analytic signature orientation class of a Witt space} \label{sec:KKThy}

We find it convenient to make use of the unbounded description of KK-theory set out by van den Dungen and Mesland in \cite{DunMes:HEUK}.
For the reader's convenience we start by reviewing some of the relevant notions.

Let $A$ and $B$ be countably generated, $\sigma$-unital, $\bb Z_2$-graded $C^*$-algebras. A $B$-Hilbert module $\cal E$ is a right $B$-module equipped with a $B$-valued inner product
$\langle \cdot, \cdot \rangle_{\cal E}$ which is complete with respect to the norm $\lVert v \rVert_{\cal E} = \lVert \langle v, v \rangle_{\cal E} \rVert_{B}^{1/2}.$
A $B$-Hilbert module is an $(A,B)$-bimodule if $A$ acts on it via adjointable bounded operators.

Let $D:\mathrm{dom}(D)\subseteq \cal E \lra \cal E$ be a closed, densely defined, self-adjoint and regular\footnote{To say that $D$ is regular is to say that $D$ and $D^*$ are densely defined and that $\mathrm{Id}+D^*D$ has dense range.} operator and define
\begin{multline*}
	\mathrm{Lip}(D) = \{ T \text{ adjointable, bounded operator on }E: \\
	T(\mathrm{dom}(D)) \subseteq \mathrm{dom}(D) \text{ and } [D, T] \text{ is adjointable and bounded}\},
\end{multline*}
\begin{equation*}
	\mathrm{Lip}_{\cal K}(D) =
	\{ T \in \mathrm{Lip}(D): \\
	T(\mathrm{Id}+D^2)^{-1/2}, 
	T^*(\mathrm{Id}+D^2)^{-1/2} \in \cal K_B \}.
\end{equation*}
An {\bf unbounded $A$-$B$-cycle}\footnote{An unbounded Kasparov module for $(A,B),$ in the sense of Baaj-Julg \cite{BaaJul:TBKONBCH} is an $(A,B)$-bimodule $\cal E$ together with a self-adjoint regular operator $D$ on $E,$ homogeneous of degree one, such that $(1+D^2)^{-1}a$ extends to an element of $\cal K_B$ for all $a\in A$ and such that there is a dense $*$-subalgebra $\cal A\subseteq A$ contained in $\mathrm{Lip}(D).$

A significant complication of using unbounded Kasparov modules is that in order to define the direct sum of two modules we would require that the subalgebra $\cal A$ is the same for the two modules. This is avoided with unbounded $A$-$B$-cycles where instead it is required that $A\subseteq \overline{\mathrm{Lip}(D)}.$} is a $\bb Z_2$-graded Hilbert bi-module $_A\cal E_B$ together with an odd regular self-adjoint operator $D$ on $E$ such 
that the elements of $A$ act as operators in the closure of $\mathrm{Lip}_{\cal K}(D).$ Direct sum makes the set of unbounded $A$-$B$-cycles into a semi-group.
The equivalence classes of unbounded $A$-$B$-cycles with respect to operator homotopy form a group with respect to direct sum denoted
\begin{equation*}
	\overline{\mathrm{UKK}}(A,B).
\end{equation*}
Van den Dungen and Mesland show that if $A$ is separable the bounded transform induces an isomorphism
\begin{equation*}
	\overline{\mathrm{UKK}}(A,B) \cong \mathrm{KK}(A,B).
\end{equation*}

The odd KK-groups $\overline{\mathrm{UKK}}^1(A,B) \cong \mathrm{KK}^1(A,B)$ are defined by the same procedure as above but working with ungraded Hilbert bi-modules. We will denote the even groups by $\overline{\mathrm{UKK}}(A,B) \cong \mathrm{KK}(A,B)$ or $\overline{\mathrm{UKK}}^0(A,B) \cong \mathrm{KK}^0(A,B)$ interchangeably. A superindex in this setting is to be understood modulo two.

$ $\\
Roughly speaking if $X$ and $Y$ are stratified spaces (and more generally) then
\begin{equation*}
\begin{gathered}
	\mathrm K^0_c(X) := \mathrm{KK}(\bb C, C_0(X)) = \text{ ``Fredholm operators parametrized by $X$"}, \\
	\mathrm K_0^{\an}(X) := \mathrm{KK}(C_0(X), \bb C) = \text{ ``Fredholm operators on $X$"}, \\
	\mathrm{KK}(C_0(X), C_0(Y)) = \text{ ``Fredholm operators on $X$ parametrized by $Y$"}.
\end{gathered}
\end{equation*}
In particular a wedge Dirac operator on $X$ satisfying the analytic Witt condition will define a class in $\mathrm K_0^{\an}(X)$ and, 
as we will show in \S \ref{sec:Fibrations}, a family of wedge Dirac-type operators on the fibers of $X\lra Y$ satisfying the analytic Witt condition will define 
a class in $\mathrm{KK}(C_0(X), C_0(Y)).$

Let us describe the case of a single operator in detail. Suppose $X$ is a smoothly stratified space (not necessarily compact) with resolution 
$M = \mathrm{res}(X),$ that $(g_M, E \lra M, g_E, \nabla^E, \cl)$ is a wedge Clifford module on $M$ and that the associated Dirac-type operator 
$\eth^E_M$ satisfies the analytic Witt condition. Let us set
\begin{equation*}
	A =  C_{\Phi,0}(M):=  C_{\Phi}(M)\cap  C_0(M), \quad B = \bb C, \quad
	\cal E = L^2(M;E)
\end{equation*}
and endow $\cal E$ with a left $A$-action and right $B$-action, both given by multiplication, and note that the $L^2$-inner product on $\cal E$ is a
$B$-valued inner product so that $\cal E$ is a Hilbert $A$-$B$-bimodule, $\bb Z_2$-graded if $E$ is.
The operator $\eth^E_{M,VAPS}$ is closed, densely defined, self-adjoint and regular and 
$ C_{\Phi,c}^{\infty}(M):= C_{\Phi}^{\infty}(M) \cap C_c^{\infty}(M) \subseteq \mathrm{Lip}_{\cal K}(\eth^E_{M,VAPS})$ so the elements of $A$ act as operators in the closure of 
$ \mathrm{Lip}_{\cal K}(\eth^E_{M,VAPS}).$ Thus we obtain a class
\begin{equation*}
	[\eth^E_{M,VAPS}] \in \overline{\mathrm{UKK}}^{\dim X}(C_{\Phi,0}(M),\bb C) 
	\cong \mathrm{KK}^{\dim X}(C_{\Phi,0}(M),\bb C) 
	= \mathrm{KK}^{\dim X}(C_0(X),\bb C) = \mathrm K_{\dim X}^{\an}(X).
\end{equation*}

The case of most interest to us is the $K$-class of the signature operator of a Witt pseudomanifold $X,$
\begin{equation*}
	[D_M^{\mathrm{sign}}] \in \mathrm K_{\dim X}^{\an}(X).
\end{equation*}
For compatibility with the Sullivan orientation, as we will show below in \S\ref{sec.compatanalytictoporient}, it is natural to define
\begin{equation*}
	\mathrm{sign}_K(X) = 2^{-\lfloor \dim X/2 \rfloor}[D_M^{\mathrm{sign}}] \in \mathrm K_{\dim X}^{\an}(X)[\tfrac12].
\end{equation*}
We refer to this as the {\bf analytic signature orientation of a Witt space}.

\begin{remark}
The class of the signature operator in $\mathrm K_{\dim X}^{\an}(X)$ coincides with that defined in \cite[\S 6.2]{AlbLeiMazPia:SPWS} (and that defined in \cite{moscoviciwu}). See also the definition of the K-homology class of the signature operator on Cheeger spaces (generalizing Witt spaces) in \cite[\S 5.1]{AlbLeiMazPia:NCCS}.
\end{remark}

\section{Invariance properties of the analytic signature orientation}\label{sect:invariance}

In this section we consider the behavior of the analytic signature orientation under changes of the underlying structure.

\subsection{Stratified diffeomorphism  invariance}\label{subsect:diffeo}

\begin{prop} \label{prop.stratdiffeoinvariance}
The analytic signature orientation of a smoothly stratified oriented Witt space is independent of the choice of wedge metric used in its definition.

If $\phi: X \to Y$ is a stratified diffeomorphism between 
smoothly stratified oriented Witt spaces, then 
\[ \phi_* \mathrm{sign}_K(X) = \mathrm{sign}_K(Y) \in \K_*^{\an}(Y)= \mathrm{KK}^{*} (C_0(Y),\cplx). \]
\end{prop}

\begin{proof}
If $g$ and $g'$ are wedge metrics on $X$ then they are quasi-isometric and hence the space of $L^2$ sections of wedge differential forms is the same
for both metrics (though of course the $L^2$-norm of an individual element is not). The path of wedge metrics $t \mapsto tg + (1-t)g'$ lifts to an operator homotopy between the corresponding signature operators, endowed with their VAPS domain. Since operator homotopies between Kasparov bimodules do not change the KK-class, we see that the analytic signature orientation is independent of the metric.

Let $g_Y$ be a wedge metric on $Y$ and consider  $\phi^* g$ on $X$. It can be proved that this is also a wedge metric. 
This can certainly be proved directly as in \cite[Lemme 5.2]{BraHecSar:LES}. However, it also follows from
the following general principle: a wedge metric $g_Y$ extends to define  a bundle metric on the wedge cotangent bundle over the resolved 
manifold of $Y$, $\widetilde{Y}$; a stratified diffeomorphism $\phi$ lifts to a diffeomorphism $\tilde{\phi}$ of the resolutions 
and thus induces a well-defined bundle map, covering $\tilde{\phi}$, between the respective wedge-cotangent bundles. The pull-back of the
metric $g_Y$ is then a metric on the wedge-cotangent bundle of the resolved manifold $\widetilde{X}$; such a metric
defines on the interior of $\widetilde{X}$, which is the regular part of $X$, a wedge metric.

Consider any class $[H,\alpha:C(X)\to \mathcal{B}(H), F]\in \mathrm K_{*}^{\an} (X)$ and recall
that if $\psi:X\to Y$ is a continuous map, then
$$\psi_* [H,\alpha:C(X)\to \mathcal{B}(H), F]:= [H,\alpha\circ \psi^* :C(Y)\to \mathcal{B}(H), F]\in \mathrm K_*^{\an} (Y)\,.$$
A similar formula works in the unbounded picture, considering a suitable dense subalgebra $\mathcal{A}(X)$ of $C(X)$.
Consider now $H'_X:= L^2 (X^{{\rm reg}}, \Lambda^* (X^{{\rm reg}}), \phi^* g_Y)$ and the class 
$$[H'_X, \mathcal{M}_X:\mathcal{A}(X)\to \mathcal{B}(H'_X), D'_X]\in \mathrm K_*^{\an} (X)$$ with $D'_X$ the signature operator associated to
$\phi^* g_Y$, $\mathcal{A}(X)$ the dense subalgebra of Lipschitz functions 
on $X$ and $\mathcal{M}_X$ equal to the multiplication operator. 
By wedge-metric invariance 
$$[H'_X, \mathcal{M}_X:\mathcal{A}(X)\to \mathcal{B}(H'_X), D'_X]=[D_X^{\mathrm{sign}}]\in \mathrm K_*^{\an} (X)$$
On the other hand, let now $H_Y$ be the space of
$L^2$-forms on $Y$ associated to $g_Y$. Pull-back by $\phi$ defines a unitary isomorphism   $U: H_Y\to H'_X$ 
 such that $U^{-1}\circ D'_X\circ U=D_Y^{\mathrm{sign}}$ or, equivalently, $U\circ D_Y^{\mathrm{sign}}\circ U^{-1}=D'_X$.
 Thus, on the one hand,
 $$\phi_* [D_X^{\mathrm{sign}}]= \phi_* [H'_X, \mathcal{M}_X:\mathcal{A}(X)\to \mathcal{B}(H'_X), D'_X]=
 [H'_X, \mathcal{M}_X\circ \phi^*:\mathcal{A}(Y)\to \mathcal{B}(H'_X), D'_X]$$
and, on the other hand, 
$$[D_Y^{\mathrm{sign}}]=  [H_Y, \mathcal{M}_Y:\mathcal{A}(Y)\to \mathcal{B}(H_Y), D_Y^{\mathrm{sign}}].$$
We set $\mathcal{M}':= \mathcal{M}_X\circ \phi^*$; we want to show that 
$$
 [H'_X, \mathcal{M}':\mathcal{A}(Y)\to \mathcal{B}(H'_X), D'_X]= [H_Y, \mathcal{M}_Y:\mathcal{A}(Y)\to \mathcal{B}(H_Y), D_Y^{\mathrm{sign}}].$$
But for $\omega\in H_Y$ we have 
$$\mathcal{M}' (f)(U \omega)= \phi^* (f) (\phi^* \omega)= \phi^* (f\omega)= U (\mathcal{M}_Y (f)(\omega)).$$
Together with our remarks so far, this means that the 2 cycles 
$$(H'_X, \mathcal{M}':\mathcal{A}(Y)\to \mathcal{B}(H'_X), D'_X)\;\;\text{and}\;\; (H_Y, \mathcal{M}_Y:\mathcal{A}(Y)\to \mathcal{B}(H_Y), D_Y^{\mathrm{sign}})$$
are unitarily equivalent, and so define the same $K$-theory class, which is what we wanted to show.
\end{proof}

\begin{remark}
For a general wedge Clifford module $(g_M, E\lra M, g_E, \nabla^E),$ the resulting K-homology class may well depend on the metrics involved.  
The reason that the argument above does not apply is that, while a homotopy of metrics lifts to a homotopy of the corresponding Dirac-type operators, this homotopy might not take place within Fredholm operators. For the signature operator the analytic Witt condition is independent of the metric but for example the dimension of the space of harmonic spinors on a surface of genus greater than two depends on the metric \cite[Theorem 2.6]{Hit:HS}.
\end{remark}

\subsection{Witt bordism invariance}\label{subsect:witt-invariance}

The following proposition, for smooth manifolds, is Proposition 4.1 in  \cite{PedRoeWei:HIBCASMOOC} (cf. \cite[Propositions 4.4 \& 5.4]{BauDouTay:CRCAK}). Our proof follows \cite[\S5]{Hig:KONM} (cf.\cite[Proposition 11.2.15]{HigRoe:AK}). A different approach for smooth manifolds can be found in \cite{MelPia:AKMWC}. 
\begin{prop} [The boundary of Dirac is the Dirac of the boundary] \label{prop:DiracBdyDirac} $ $

Let $W$ be a smoothly stratified space with boundary $X=\pa W,$ and denote their resolutions by $N = \mathrm{res}(W),$ and $M = \mathrm{res}(X) = \pa N.$
Let $(g_N, E\lra N, g_E, \nabla^E)$ be a wedge Clifford module on $N$ whose associated Dirac-type operator $\eth_{N,VAPS}^E$ satisfies the analytic Witt condition.
If $W$ is even-dimensional, assume that $E$ is graded with grading operator $\gamma.$ Let $\nu$ be a normal vector field to $M$ and define a
wedge Clifford module $(g_N|_M, \pa E\lra M, g_E|_{\pa E}, \nabla^E|_{\pa E})$ on $M$ as follows:
\begin{enumerate}
\item If $W$ is odd-dimensional, $\pa E = E|_{M},$ graded by Clifford multiplication by $i\nu;$
\item if $W$ is even-dimensional, $\pa E$ is the $+1$ eigenspace of the involution $i\nu\gamma.$
\end{enumerate}
Then the wedge Dirac-type operator associated to this Clifford module, $\eth_{M,VAPS}^{\pa E}$ satisfies the analytic Witt condition and the boundary map in K-homology, $\pa: \mathrm{K}^\an_{\dim W}(W, \pa W) \lra \mathrm{K}^\an_{\dim \pa W}(\pa W),$ satisfies
\begin{equation*}
	\pa[\eth_{N,VAPS}^E] = [\eth_{M,VAPS}^{\pa E}]. 
\end{equation*}
\end{prop}

\begin{proof}
It is pointed out in \cite[\S5]{Hig:KONM} that if $A$ is a $C^*$-algebra and $J$ an ideal in $A,$ so that we have a short exact sequence
\begin{equation*}
	0 \lra J \lra A \lra A\diagup J \lra 0,
\end{equation*}
then the boundary map $\pa:\mathrm{KK}^j(J, \bb C) \lra \mathrm{KK}^{j-1}(A\diagup J, \bb C)$ factors as
\begin{equation}\label{eq:FactoredBdyMap}
	\mathrm{KK}^j(J, \bb C) \lra 
	\mathrm{KK}^{j}(C_0((0,1)) \otimes A\diagup J, \bb C) \lra
	\mathrm{KK}^{j-1}(A\diagup J, \bb C)
\end{equation}
where the first map is defined using the mapping cone associated to the short exact sequence and the second map is the inverse of the
Kasparov product with the class
\begin{equation*}
	[\eth_{(0,1)}] \in \mathrm{KK}^1( C_0(0,1), \bb C ) = \mathrm K^\an_1( (0,1) )
\end{equation*}
of the Dirac operator on the positively oriented open unit interval.

We wish to apply this with $A = C_0(W),$  $J = \{ f \in C_0(W): f|_{X}=0\},$ so that $A\diagup J = C_0(X).$ 
Naturality of the boundary map allows us to replace $W$ with a collar neighborhood of $X$ so instead we take $A = C_0( [0,1) \times X)$ and
$J = C_0( (0,1) \times X).$
Arguing as in the proof of \cite[Theorem 5.1]{Hig:KONM} this has the happy consequence that the first map in \eqref{eq:FactoredBdyMap} is the identity map
and hence the boundary map in K-homology is the inverse of the Kasparov product with $[\eth_{(0,1)}].$ Thus we can conclude the proof if we can show that
\begin{equation*}
	[\eth_{N,VAPS}^E] = [\eth_{(0,1)}] \otimes  [\eth_{M,VAPS}^{\pa E}] 
	\in \mathrm{K}^\an_{\dim W}([0,1)\times X, \{0\}\times X).
\end{equation*}
This can be checked using the fact (e.g., from \cite[(1)]{BaaJul:TBKONBCH}) that the Kasparov product of $[\eth_{(0,1)}]$ and $[\eth_{M,VAPS}^{\pa E}]$ is represented by $\eth_{(0,1)} \hat \otimes I + I \hat\otimes  \eth_{M,VAPS}^{\pa E}.$ This is also a particular case of Theorem \ref{thm:FunctorialityFibrations} below.
\end{proof}

In the non-singular situation, the following proposition has been established in \cite[pg. 290]{PedRoeWei:HIBCASMOOC}, \cite[Theorem 2]{RosWei:SO}.
We note that certain bordism invariance results for the
signature operator $\K$-homology class on Witt spaces have also been
obtained by Hilsum in \cite[Section 3]{Hil:PDBDKP}.

We use the notation $[D_{W,\pa W}^{\mathrm{sign}}]$ to denote the class induced by the signature operator of a wedge metric on an oriented Witt space with boundary in the K-homology group of $W$ relative to $\pa W.$

\begin{prop}$ $ \label{prop:Bordism}

\begin{enumerate}
\item \label{lem.boundaryofsignopclass}
If $W$ is an oriented smoothly stratified Witt space with boundary then the boundary map in K-homology, $\pa: \mathrm{K}^\an_{\dim W}(W, \pa W) \lra 
\mathrm{K}^\an_{\dim \pa W}(\pa W),$ satisfies
\begin{equation*}
	\pa[D_{W,\pa W}^{\mathrm{sign}}] = k [D_{\pa W}^{\mathrm{sign}}] \quad \text{ with } \quad
	k = \begin{cases} 1 & \text{ if $\dim W$ is even}\\  2 & \text{ if $\dim W$ is odd}\end{cases}
\end{equation*}
\item \label{prop.signopclassbordinoverx}
Let $X$ be a {compact} oriented smoothly stratified Witt space, 
$\mathscr Y$ any finite CW complex,
and let $f:X\to \mathscr Y$ be a continuous map.
Then $f_* [D_X^{\mathrm{sign}}] \in \mathrm K_{\dim \mathscr Y} (\mathscr Y)$ is a Witt bordism invariant of the pair $(X,f).$\\
\end{enumerate}
\end{prop}

\begin{proof}
(1) Denote the resolution of $W$ by $N = \mathrm{res}(W)$ and that of $\pa W$ by $M = \mathrm{res}(\pa W) = \pa N.$
Since the signature operator of a wedge metric on $W$ is a Dirac-type operator we know from Proposition \ref{prop:DiracBdyDirac} that $\pa[D_W^{\mathrm{sign}}]$ is the K-homology class of a Dirac-type operator $\eth$ on $\pa W.$ Without loss of generality we may assume that the wedge metric on $N$ is collared near $\pa N,$ i.e., that there is a collar neighborhood of $N$ which is isometric to $[0,1) \times M$ for some wedge metric on $M,$ and we may restrict attention to this neighborhood. Rosenberg and Weinberger computed the decomposition of the signature operator on a product of closed manifolds in \cite[Lemma 6]{RosWei:SO} but, as the computation only involves the Clifford algebras and gradings, it applies to $[0,1) \times M,$ and shows that, in K-homology, 
\begin{equation*}
	[D_{[0,1)\times \pa W}^{\mathrm{sign}}] = k \; [D_{[0,1)}^{\mathrm{sign}}] \otimes [D_{\pa W}^{\mathrm{sign}}]
	\quad \text{ with } \quad
	k = \begin{cases} 1 & \text{ if $\dim W$ is even}\\  2 & \text{ if $\dim W$ is odd}\end{cases}
\end{equation*}
Finally we know from the proof of Proposition \ref{prop:DiracBdyDirac} that the boundary map in K-homology for a product is the inverse of the Kasparov
product by the Dirac operator on the interval, so the result follows.

(2) 
Let $n = \dim X.$
Suppose that $f:X\lra \mathscr Y$ is Witt nullbordant.
Thus, there is a compact smoothly stratified oriented Witt space $W$
with boundary $\partial W =X$ and a continuous map $F: W\to \mathscr Y$ 
that extends $f$. Let $i: X \hookrightarrow W$ denote the
inclusion of the boundary and
consider the commutative diagram
\begin{equation*}
	\xymatrix{
	\K_{n+1} (W,X) \ar[r]^-{\partial_*} 
	  & \K_n^{\an} (X) \ar[rd]_{f_*} \ar[r]^{i_*} &
	    \K_n^{\an} (W) \ar[d]^{F_*} \\
	 & & \K_n^{\an} (\mathscr Y),   
	} 
\end{equation*}
whose top row is exact.
If $n$ is even, then from (1) we know $\partial [D_{W,X}^{\mathrm{sign}}] = [D_X^{\mathrm{sign}}]$ and thus
\begin{equation*}
	 f_* [D_X^{\mathrm{sign}}]
	 = F_* i_* [D_X^{\mathrm{sign}}]
	 = F_* i_* \partial [D_{W,X}^{\mathrm{sign}}]
	 = 0.
\end{equation*}
If $n$ is odd, then (1) asserts that $\partial [D_{W,X}^{\mathrm{sign}}] = 2[D_X^{\mathrm{sign}}]$. 
Therefore, $f_* [D_X^{\mathrm{sign}}]$ is either zero or has order $2$:
\begin{equation*}
	 2 f_* [D_X^{\mathrm{sign}}] = F_* i_* (2[D_X^{\mathrm{sign}}]) = F_* i_* \partial [D_{W,X}^{\mathrm{sign}}]= 0
 	\in \K_n^{\an} (\mathscr Y).
\end{equation*}
In either case, $f_* [D_X^{\mathrm{sign}}]$ vanishes in the localization $\K_n^{\an} (\mathscr Y)[\smlhf]$. However
as in \cite[Theorem 2]{RosWei:SO} it is possible to use arguments from \cite[\S 4]{PedRoeWei:HIBCASMOOC} to avoid having to invert $2.$

Specifically, Pedersen, Roe, and Weinberger point out that on a smooth manifold with boundary equipped with a unit vector field $\nu$ which is normal to the boundary (but defined on all of the manifold) the operator of right Clifford multiplication by $i\nu$ is an involution and the restriction of the signature operator
to either of its eigenspaces produces an operator whose boundary is the signature operator of the boundary on the nose. The argument above then establishes that $f_* [D_X^{\mathrm{sign}}]=0 \in \K_n^{\an} (\mathscr Y).$
A smooth manifold with boundary admits such a vector field if and only if its Euler characteristic vanishes. Rosenberg and Weinberger explain how to reduce to this case by making modifications (either a connected sum or punching out a small disk) away from the boundary.

To apply these arguments to our situation first note that if two bundles are isomorphic and one has a nowhere vanishing section then so does the other.
If $N$ is the resolution of $W,$ $N = \mathrm{res}(W),$ then we may `round out the corners' and homotope $N$ to a smooth manifold with boundary $N'.$
If $N$ has vanishing Euler characteristic then $N'$ has a smooth nowhere vanishing vector field normal to the boundary. Since the tangent bundle of $N$ is
isomorphic to the pull-back of the tangent bundle of $N',$ $N$ also has a nowhere vanishing vector field which we can assume (by how we round out the corners) is normal to the boundary hypersurfaces of $N$ that correspond to the boundary of $W.$ Next, since the wedge tangent bundle of $N$ is isomorphic (albeit not naturally isomorphic) to the ordinary tangent bundle, we see that there is a nowhere vanishing wedge vector field on $N,$ normal to the lift of the boundary of $W,$ which we can normalize to have everywhere unit length. Applying the Pedersen, Roe, and Weinberger argument as done in \cite[Theorem 2]{RosWei:SO} then establishes the result. 
\end{proof}

Recall from \S\ref{sec:WittSpaces} that $\Omega^{\Witt}_n$ denotes the bordism 
theory based on oriented PL Witt pseudomanifolds and $\Omega^{\Witt,\infty}_n$ denotes the corresponding theory based on oriented smoothly stratified Witt pseudomanifolds.
The proposition implies that the assignment
\begin{equation} \label{equ.wittbordtokprecursor}
[f:W^n \to X]\ \mapsto f_* [D_W]  
\end{equation}
constitutes a well defined homomorphism
\[ \Omega^{\Witt,\infty}_n (X) \longrightarrow
    \K^\an_n (X).  \]
Since bordism of PL stratified Witt spaces is 
isomorphic to bordism of smoothly stratified Witt spaces
(Proposition \ref{prop.bordofsmoothlystratandPLwitt}), we may also 
regard this as a homomorphism
\[ \Omega^{\Witt}_n (X) \longrightarrow
    \K^\an_n (X).  \]
This map is a natural transformation of homotopy functors, but
not of homology theories (Rosenberg-Weinberger 
\cite[Remark 4]{RosWei:SO}).
The modified map
\begin{equation}\label{equ.thetafromwittbordhlftokanhlf}
\begin{gathered}
	\theta: \Omega^{\Witt,\infty}_n (X)[\smlhf] \longrightarrow	\K^\an_n (X)[\smlhf],  \\
	\theta ([f:W^n \to X]\otimes_\intg r) := r 2^{-\lfloor n/2 \rfloor} f_* [D_W^{\mathrm{sign}}] = rf_*(\mathrm{sign}_K(W)), 
\end{gathered}
\end{equation}
\emph{is} a natural transformation of homology theories, since
the powers of $2$ correctly absorb the behavior of the 
connecting homomorphism as described in Proposition \ref{prop:Bordism} (\ref{lem.boundaryofsignopclass}).

\subsection{Witt bordism: smooth versus PL}\label{subsect:smooth-versus-PL}
In several places, particularly in Section \ref{sec.compatanalytictoporient}
where we consider a natural transformation from Witt bordism to
analytic $\K$-homology via the signature operator, 
we need to identify Siegel's bordism theory of
PL Witt spaces with bordism of smoothly stratified Witt
spaces. Let us discuss this identification in more detail.
Recall that $\Omega^\Witt_* (-)$ denotes the bordism theory of PL Witt spaces
as introduced in \cite{Sie:WSGCTKOP}. This functor is well-known to be
a homology theory.
Let $\Omega^{\Witt,\infty}_* (-)$ denote the bordism theory based on
cycles given by continuous maps on closed 
smoothly stratified Witt spaces.
This functor is a homology theory as well, since transversality is
available for smoothly stratified spaces.
\begin{prop} \label{prop.bordofsmoothlystratandPLwitt}
Triangulation induces a natural equivalence of homology theories
\[ t: \Omega^{\Witt,\infty}_* (-) \longrightarrow
       \Omega^\Witt_* (-). \]
\end{prop}
\begin{proof}
Let $\mathscr Y$ be a CW complex and
let $f: X^n \to \mathscr Y$ be a continuous map on a closed smoothly
stratified Witt space $X$. By Goresky \cite{Gor:TSO},
Thom-Mather pseudomanifolds can be triangulated compatibly with
the stratification. The triangulation can be taken to be a smooth
embedding on the interiors of simplices. Choose a Goresky-triangulation
of $X$. This choice yields a PL pseudomanifold $\overline{X}$ with the same
underlying topological space as $X$. It is not clear that different
choices of Goresky-triangulations yield PL-isomorphic PL pseudomanifolds.
Nevertheless, we claim that $(f: X\to \mathscr Y) \mapsto (f:\overline{X} \to \mathscr Y)$ 
induces a well-defined
homomorphism $t: \Omega^{\Witt,\infty}_* (\mathscr Y) \to \Omega^\Witt_* (\mathscr Y)$.
The reason is indicated by Goresky and MacPherson in Section 5.3 of
\cite{GorMac:IHT}: Goresky's techniques imply that any two
Goresky-triangulations of $X$ are concordant.
This yields a triangulation of the cylinder
$X\times [0,1]$, which is then a PL bordism between the two triangulations
of $X$. The map $f$ extends continuously over the cylinder by
$f\times \id_{[0,1]}$.
If $F:W \to \mathscr Y$ is a smoothly stratified Witt nullbordism for
$f:X=\partial W \to \mathscr Y$, then we may Goresky-triangulate $W$ to obtain
a PL Witt nullbordism $F:\overline{W} \to \mathscr Y$ for $f:\overline{X} \to \mathscr Y$,
where $\overline{X}$ is chosen to be the restriction of the Goresky-triangulation
of $\overline{W}$ to the boundary.
Hence we obtain a well-defined map $t$ and this map is a homomorphism.
It is natural in $\mathscr Y$, since maps induced by $g:\mathscr Y\to \mathscr Y'$ are just given
by composing $f$ and $g$.
For $\mathscr Y$ a point, the map $t$ was already considered by Zentarra in \cite{Zen:BTPLSWP},
who proved that it is an isomorphism over a point.
(The map is surjective, since Siegel's generators of
$\Omega^\Witt_{4k} (\pt)$ are given by Whitehead-triangulating certain compact smooth
manifolds and then coning off the boundary. The smooth manifold with cone
attached to the boundary then possesses a smooth stratification such that a
Goresky-triangulation may be taken to be the Whitehead triangulation, coned off.
The map is injective as is seen by carefully verifying that Siegel's
singular surgeries can be carried out within the class of smoothly stratified spaces.
In addition to the fact that Siegel's intersection form invariant
$w(-)\in \mathrm{Witt}(\rat)$ is a \emph{topological} invariant,
one exploits the observation that if the triangulation restricts to a 
smooth embedding on the interiors of simplices, then the skeletal
filtration determined by the triangulation is indeed a smooth Thom-Mather
stratification.)
Finally, $t$ is an equivalence, for $t$ is a natural transformation of 
homology theories which is an isomorphism on coefficient groups. 
\end{proof}

\section{Analytic transfer of the signature orientation along a fibration} \label{sec:Fibrations}

The goal of this section is to describe how to assign to every fiber bundle of smoothly stratified spaces
\begin{equation*}
	W - X \xlra{p} Y,
\end{equation*}
in which the fibers are oriented Witt pseudomanifolds, a class in the  KK-group $\mathrm{KK}(C_0(X), C_0(Y)).$
Indeed from \cite{AlbLeiMazPia:SPWS, AlbGel:IFFDTOP} the family of vertical signature operators associated to any choice of vertical wedge metric is a family
of Fredholm operators and we will see that it defines the desired class. To do this, in order to apply the theory developed in \cite{AlbGel:IFFDTOP}, we discuss the ``grid resolution" of this fiber bundle to a fiber bundle
of manifolds with corners with iterated fibration structures,
\begin{equation*}
	\mathrm{res}(W) - \mathrm{res}_{\mathrm{grid}}(X) \lra \mathrm{res}(Y).
\end{equation*}
%

\subsection{Resolutions of a stratified fiber bundle} \label{sec:ResStratFibBdles}

Starting with a fiber bundle of smoothly stratified spaces
\begin{equation*}
	W - X \xlra{p} Y,
\end{equation*}
we will discuss the partial resolutions that result from resolving 
the base or the fibers and the grid resolution that results from resolving both.
The grid resolution continues to participate in a fiber bundle. One could go further and fully resolve $X$ but,
unless one of $W$ or $Y$ is smooth, at the cost of the fiber bundle structure.

In the simplest interesting situation when $X$ is the product of two stratified spaces of positive depth, $X = W \times Y,$  endowed with the projection
onto $Y,$ $p:X \lra Y,$ the partial resolutions of base and fiber are, respectively
\begin{equation*}
	\mathrm{res}_{\mathrm{base}}(X,p) = W \times \mathrm{res}(Y), \quad
	\mathrm{res}_{\mathrm{fib}}(X,p) = \mathrm{res}(W) \times Y,
\end{equation*}
and the grid resolution of $X$ is 
\begin{equation*}
	\mathrm{res}_{\mathrm{grid}}(X,p) 
	= \mathrm{res}(W) \times \mathrm{res}(Y).
\end{equation*}
This is a smooth manifold with corners but, as neither $Y$ nor $W$ are smooth, this is not equal to the resolution of $X$ and 
it does not have a natural iterated fibration structure. Indeed in \cite{KotRoc:PMWFC} Kottke and Rochon have shown that 
$\mathrm{res}(X)$ can be obtained from $\mathrm{res}(W) \times \mathrm{res}(Y)$ by a sequence of real blow-ups of 
appropriate submanifolds of the boundary producing what they call the `ordered product' $\wt{\times}$ so that
\begin{equation*}
	\mathrm{res}(X)
	= \mathrm{res}(W) \wt{\times} \mathrm{res}(Y).
\end{equation*}
It is worth noting that the lifted projection $\mathrm{res}(X) \lra \mathrm{res}(Y)$ is no longer a fiber bundle map 
(it is a ``$b$-fibration"), as this is one of the reasons why we have use for the grid resolution.\\

While the grid resolution is a smooth manifold with corners, the partial resolutions of base and fiber above are the product of a smooth manifold with corners and a smoothly stratified space. For the purpose of regularity we consider such a product as a smoothly stratified space with the stratification obtained by endowing the manifold with corners with its natural stratification (see Remark \ref{rmk:StratMfdCorners}).

$ $\\
{\bf Prelude: When the base and fiber each have depth one.}
Since the constructions below are a bit intricate, we thought it would be useful to first describe the simplest interesting situation. This prelude can be skipped.

Suppose that $W \fib X \xlra p Y$ is a fiber bundle of smoothly stratified spaces where $W$ has singular part $W_0$ and regular part $W_1$ and $Y$ has singular part $Y_0$ and regular part $Y_1$ and $X$ has poset stratification given by the product poset
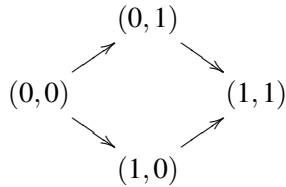
\begin{equation*}
	\xymatrix@=1em{
	& (0,1) \ar[rd] & \\
	(0,0) \ar[ru] \ar[rd] & & (1,1) \\
	& (1,0) \ar[ru] & }
\end{equation*}
where we have chosen that the first entry corresponds to the fiber and the second entry corresponds to the base.
The regular part of $X$ is $X_{\mathrm{reg}}=X_{(1,1)},$ the singular part of $X$ is $X_{\mathrm{sing}} = X_{(0,0)} \cup X_{(0,1)} \cup X_{(1,0)}.$ The singular stata $X_{(0,1)}$ and $X_{(1,0)}$ are disjoint but their closures are both obtained by the inclusion of $X_{(0,0)}.$ Note that $p^{-1}(Y_0) = X_{(0,0)} \cup X_{(1,0)}$ and $p^{-1}(Y_1) = X_{(0,1)} \cup X_{(1,1)}.$

{\em Resolving along the base:}
The resolution of the base $\mathrm{res}(Y)$ comes with a blow-down map $\beta_Y: \mathrm{res}(Y) \lra Y$ and we obtain the partial resolution of $X$ along the base by pulling-back  $p$ along $\beta_Y,$ i.e.,
\begin{equation*}
	\xymatrix{
	\mathrm{res}_{\mathrm{base}}(X,p) \ar[r] \ar[d]^-{\overline p} \pullbackcorner & X \ar[d]^-p \\
	\mathrm{res}(Y) \ar[r]^-{\beta_Y} & Y. }
\end{equation*}
The induced map $\bar p$ participates in a fiber bundle
\begin{equation*}
	W \fib \mathrm{res}_{\mathrm{base}}(X,p) \xlra{\overline p} \mathrm{res}(Y)
\end{equation*}
(and for the purposes of regularity we note that if we regard $\mathrm{res}(Y)$ with its natural stratification as a manifold with boundary, i.e., the stratification with strata $\pa\mathrm{res}(Y) = \pa_0\mathrm{res}(Y)$ and $\mathrm{res}(Y)^{\circ},$ then $\bar p$ is a smooth fiber bundle of smoothly stratified spaces).
The new base $\mathrm{res}(Y)$ is a manifold with fibered boundary with boundary fiber bundle
\begin{equation*}
	F_0\mathrm{res}(Y) \fib \pa_0\mathrm{res}(Y) \xlra{\phi_0} B_0\mathrm{res}(Y) = Y_0.
\end{equation*}
We denote the pre-image of the boundary of $\mathrm{res}(Y)$ under $\bar p$ by $\pa_0^h\mathrm{res}_{\mathrm{base}}(X,p)$ with the $h$ exponent denoting `horizontal', it fits into the restriction of the pull-back diagram above,
\begin{equation*}
	\xymatrix{
	\pa_0^h\mathrm{res}_{\mathrm{base}}(X,p) \ar[r] \ar[d]^-{\overline p} \pullbackcorner & X_{(0,0)} \cup X_{(1,0)} \ar[d]^-p \\
	\pa_0\mathrm{res}(Y) \ar[r]^-{\phi_0} & Y_0 }
\end{equation*}
since $\phi_0 = \beta_Y|_{\pa_0\mathrm{res}(Y)}.$ The unlabeled arrow in this diagram is a fiber bundle map, just as the labeled arrows are, and we denote it by
\begin{equation*}
	F_0^h \mathrm{res}_{\mathrm{base}}(X,p) \fib
	\pa_0^h \mathrm{res}_{\mathrm{base}}(X,p) \xlra{\phi_0^h}
	X_{(0,0)} \cup X_{(1,0)} =: B_0^h \mathrm{res}_{\mathrm{base}}(X,p).
\end{equation*}
Note that here $F_0^h \mathrm{res}_{\mathrm{base}}(X,p) = F_0\mathrm{res}(Y)$ is a smooth manifold and $B_0^h \mathrm{res}_{\mathrm{base}}(X,p)$ is a smoothly stratified space. We refer to the fiber bundle $\phi_0^h$ and its compatibility with the fiber bundle $\phi_0$ as a {\bf horizontal iterated fibration structure} on $\mathrm{res}_{\mathrm{base}}(X,p).$\\

{\em Resolving along the fibers:}
Returning to our fiber bundle $W \fib X \xlra p Y,$ we can use the fact that $p$ is a smooth map of smoothly stratified spaces to see that it induces a fiber bundle
\begin{equation*}
	W_0 \fib X_{(0,0)}\cup X_{(0,1)} \lra Y.
\end{equation*}
We can use this to simultaneously resolve all of the fibers of $p$ to obtain the partial resolution of $X$ with respect to the fibers, $\mathrm{res}_{\mathrm{fib}}(X,p).$ The resolution maintains the fiber bundle structure and we denote the result by
\begin{equation*}
	\mathrm{res}(W) \fib \mathrm{res}_{\mathrm{fib}}(X,p) \xlra{\widehat p} Y.
\end{equation*}
This partial resolution replaced $X_{(0,0)}\cup X_{(0,1)}$ with a new boundary face which we denote $\pa_0^v\mathrm{res}_{\mathrm{fib}}X,$ with the $v$ exponent denoting `vertical'.
The restriction of $\widehat p$ to this face is a fiber bundle over $Y$ with fiber the boundary of the resolution of $W,$
\begin{equation*}
	\pa_0\mathrm{res}(W) \fib \pa_0^v\mathrm{res}_{\mathrm{fib}}(X,p) \xlra{\widehat p} Y.
\end{equation*}
The blow-down maps of the fibers, $\beta_W: \mathrm{res}(W) \lra W,$ fit together into $\beta_{\mathrm{fib}}: \mathrm{res}_{\mathrm{fib}}(X,p)\lra X$ so that
\begin{equation*}
	\xymatrix{
	\mathrm{res}(W) \ar@{-}[r] \ar[dd]_-{\beta_W} & \mathrm{res}_{\mathrm{fib}}(X,p) \ar[rd]^-{\widehat p} \ar[dd]_-{\beta_{\mathrm{fib}}} & \\
	& & Y \\
	W \ar@{-}[r] & X \ar[ru]^-{ p} & }
\end{equation*}
and if we restrict this diagram to $\pa_0^v\mathrm{res}_{\mathrm{fib}}(X,p),$ we get
\begin{equation*}
	\xymatrix{
	\pa_0\mathrm{res}(W) \ar[dd]_-{\phi_0= \beta_W|_{\pa_0\mathrm{res}(W) }} \ar@{-}[r] 
		& \pa_0^v \mathrm{res}_{\mathrm{fib}}(X,p) \ar[dd]_-{\phi_0^v= \beta_{\mathrm{fib}}|_{\pa_0^v\mathrm{res}_{\mathrm{fib}}(X,p) }}
		\ar[rd]^-{\hat p} & \\
	& & Y \\
	W_0 \ar@{-}[r] 
		& 
		X_{(0,0)}\cup X_{(0,1)}
		\ar[ru]^-{p} & }
\end{equation*}
This shows how the boundary fiber bundles of the fibers of $\widehat p,$ 
\begin{equation*}
	F_0\mathrm{res}(W) \fib \pa_0\mathrm{res}(W) \xlra{\phi_0} B_0\mathrm{res}(W) = W_0,
\end{equation*}
fit together into 
\begin{equation*}
	F_0^v\mathrm{res}_{\mathrm{fib}}(X,p) \fib \pa_0^v \mathrm{res}_{\mathrm{fib}}(X,p) \xlra{\phi_0^v} B_0^v\mathrm{res}_{\mathrm{fib}}(X,p) 
	= X_{(0,0)}\cup X_{(0,1)}.
\end{equation*}
We refer to the fiber bundle $\phi_0^v$ and its compatibility with the fiber bundles $\phi_0$ and $\widehat p$ as a {\bf vertical iterated fibration structure} on $\mathrm{res}_{\mathrm{fib}}(X,p).$\\
{\em Resolving both the base and the fibers:} If we resolve the fiber bundle $W \fib X \xlra p Y$ along the base we again have a stratified fiber bundle whose fiber is $W$ and so we may proceed as above and resolve along the fibers. We refer to the result as the grid resolution of $X,$
\begin{equation*}
	\mathrm{res}_{\mathrm{grid}}(X) = \mathrm{res}_{\mathrm{fib}}(\mathrm{res}_{\mathrm{base}}(X,p),\overline p).
\end{equation*}
The result is a smooth manifold with corners that participates in a smooth fiber bundle
\begin{equation*}
	\mathrm{res}(W) \fib \mathrm{res}_{\mathrm{grid}}(X) \xlra{\wt p} \mathrm{res}(Y)
\end{equation*}
where the base and the fiber are smooth manifolds with fibered boundary. The only displeasing aspect of $\mathrm{res}_{\mathrm{grid}}(X)$ is that it does not have a natural structure of manifold with fibered corners. The problem is that the two boundary hypersurfaces of $\mathrm{res}_{\mathrm{grid}}(X)$ correspond to the labels $(0,1)$ and $(1,0)$ and, since these are non-comparable elements of the poset, their corresponding boundary hypersurfaces should not intersect. This is easily fixed, we need only blow-up their intersection to obtain the resolution of $X,$
\begin{equation*}
	\mathrm{res}(X) = [\mathrm{res}_{\mathrm{grid}}(X); \pa_0^v\mathrm{res}_{\mathrm{grid}}(X) \cap \pa_0^h\mathrm{res}_{\mathrm{grid}}(X) ].
\end{equation*}
As this blow-up does not treat all of the fibers of $\wt p$ equally, the fiber bundle structure of $p$ does not survive beyond the grid resolution.\\
$ $\\

Returning to the general case, let $W \fib X \xlra{p} Y$ be a fiber bundle of smoothly stratified spaces, so that the poset of $X$ satisfies
\begin{equation*}
	S(X) = S(W) \times S(Y)
\end{equation*}
(with the obvious notation). By resolving the strata corresponding to $(\max S(W)) \times S(Y)$ of $X$ we will obtain the partial base resolution 
of $X,$ which will fiber over the resolution of $Y$ with typical fiber $W,$ and similarly by resolving the strata corresponding to 
$S(W) \times (\max S(Y))$ of $X$ we will obtain the partial fiber resolution of $X,$ which will fiber over $Y$ with typical fiber the resolution of $W.$
In general neither of these partial resolutions will result in a smooth manifold with corners. Once we perform both resolutions (in either order) we will 
obtain the grid resolution of $X,$ $\mathrm{res}_{\mathrm{grid}}(X),$ which will be a smooth manifold with corners but it will not have a natural iterated
fibration structure. Instead it will have two structures, one inherited from the base $Y$ and the other inherited from the fiber $W.$ We will start by defining
these structures, then we will construct the grid resolution, and then we will show that it has the advertised properties.

\begin{defn} $ $ \label{def:HorVertIFS}

\begin{enumerate}
\item Let $W \fib  X \xlra{\psi} N$ be a fiber bundle of smoothly stratified spaces in which the base is a smooth manifold with corners.
A {\bf horizontal iterated fibration structure} consists of an iterated fibration structure on the base $N,$ associated to the boundary 
stratification $N \lra S(N),$ such that:
\begin{itemize}
	\item For all $\alpha\in S(N),$ with corresponding fiber bundle $F_\alpha N \fib \pa_\alpha N \xlra {\phi_\alpha } B_\alpha N,$ 
	$\pa_\alpha ^hX:= \psi^{-1}(\pa_\alpha  N)$ 
	is either a disjoint union of connected codimension one strata of $X$ or the regular part of $X,$
	and participates in a fiber bundle of smoothly stratified spaces
	\begin{equation*}
		F_\alpha^hX \fib \pa_\alpha^hX \xlra{\phi_\alpha^h} B_a^hX.
	\end{equation*}
	If $\alpha ,\gamma \in S(N)$ are such that $\pa_\alpha^hX \cap \pa_\gamma^hX\neq \emptyset$ and $\alpha<\gamma,$ 
	there is a commutative diagram of fiber bundles of smoothly stratified spaces
	\begin{equation*}
		\xymatrix{
		\pa_\alpha^h X \cap \pa_\gamma^h X \ar[rr]^-{\phi_\gamma^h} \ar[rd]^-{\phi_\alpha^h} & & 
			\pa_\gamma^hB_\alpha^hX \ar[ld]_-{\phi_{\gamma\alpha}^h} 
			\\ & B_\alpha^hX & }
	\end{equation*}
	where $\pa_\gamma^hB_\alpha^hX$ is a disjoint union of connected codimension one strata of $B_\alpha^hX.$
	\item For all $\alpha \in S(N),$ the restriction of $\psi$ to a map $\pa_\alpha^h X  \lra \pa_\alpha N$ is the pull-back of a fiber bundle 
	$B_\alpha^hX \xlra{\psi_\alpha^h} B_\alpha N,$ i.e., 
	\begin{equation}\label{eq:defHorPullBackDiag}
		\xymatrix{
		\pa_\alpha^h X \ar[r]^-{\phi_\alpha^h} \ar[d]_-{\psi} \pullbackcorner & B_\alpha^hX \ar[d]^-{\psi_\alpha^h} \\ 
		\pa_\alpha N \ar[r]^-{\phi_\alpha} & B_\alpha N }
	\end{equation}
	is a pull-back diagram.
\end{itemize}

\item Let $L \fib X \xlra{\psi} Y$ be a fiber bundle of smoothly stratified spaces in which the fiber is a smooth manifold with corners. 
A {\bf vertical iterated fibration structure}\footnote{When the base is smooth, this is referred to as a {\bf locally trivial family of manifolds with corners and iterated fibration structures} in \cite[Definition 1.3]{AlbGel:IFFDTOP}.} consists of an iterated fibration structure on the fiber $L,$ associated to the boundary
stratification $L \lra S(L),$ such that:
\begin{itemize}
	\item For all $a\in S(L),$ with corresponding fiber bundle $F_aL \fib \pa_a L \xlra {\phi_a} B_aL,$ there is $\pa_a^vX \subseteq X,$ either a disjoint
	union of connected codimension one strata of $X$ or the regular part of $X,$ such that $\psi$ restricted to $\pa_a^vX$ fibers over $Y$ with fiber 
	$\pa_aL.$ Each $\pa_a^vX$ participates in a fiber bundle of smoothly stratified spaces
	\begin{equation*}
		F_a^vX \fib \pa_a^vX \xlra{\phi_a^v} B_a^vX.
	\end{equation*}
	If $a,b \in S(L)$ are such that $\pa_a^vX \cap \pa_b^vX\neq \emptyset$ and $a<b,$ there is a commutative diagram of fiber bundles of smoothly
	stratified spaces
	\begin{equation*}
		\xymatrix{
		\pa_a^v X \cap \pa_b^v X \ar[rr]^-{\phi_b^v} \ar[rd]^-{\phi_a^v} & & \pa_b^vB_a^vX \ar[ld]_-{\phi_{ba}^v} \\ & B_a^vX & }
	\end{equation*}
	where $\pa_b^vB_a^vX$ is a disjoint union of connected codimension one strata of $B_a^vX.$
	\item For all $a\in S(L),$ there is a fiber bundle map $B_aL \fib B_a^vX \xlra{\psi_a^v} Y$ participating in the commutative diagram
	\begin{equation*}
		\xymatrix{
		\pa_a^vX \ar@{^(->}[r] \ar[d]_-{\phi^v_a} & X \ar[dd]^-{\psi}\\
		B_a^vX \ar[dr]_-{\psi_a^v} & \\
		& Y }
	\end{equation*}
\end{itemize}
$ $
\item
We say that a fiber bundle carries a {\bf grid iterated fibration structure} if it is equipped with both a horizontal and a vertical iterated fibration structure.
\end{enumerate}

\end{defn}

Starting with a fiber bundle $W \fib X \xlra p Y$ of smoothly stratified spaces, we obtain the partial resolution corresponding to the base by 
considering the pull-back of $p$ along the blow-down map from the resolution of $Y,$
\begin{equation}\label{eq:ResBaseDef}
	\mathrm{res}_{\mathrm{base}}(X,p):=\beta_Y^*X, 
	\quad \text{ where } \quad
	\xymatrix{ 
	\beta_Y^*X  \ar[d]_-{\overline p} \ar[r] \pullbackcorner & X \ar[d]^-p \\
	\mathrm{res}(Y) \ar[r]^-{\beta_Y} & Y. }
\end{equation}
The partial resolution corresponding to the fiber is obtained by resolving the strata of $X$ corresponding to $S(W) \times (\max S(Y)).$ As
these strata are transverse to the fibers of $p,$ the composition of the blow-down map with $p$ is again a fiber bundle map, $\widehat p,$
\begin{equation*}
	\xymatrix{
	\mathrm{res}_{\mathrm{fib}}(X,p) \ar[rr]^-{\beta_{\mathrm{fib}}} \ar[rd]_-{\widehat p} & & X \ar[ld]^-p \\
	& Y &}
\end{equation*}
with typical fiber $\mathrm{res}(W).$ 
Indeed, given an open cover of $Y,$ $\{ \cal U_\ell \}$ that is trivializing for $p$ and the corresponding transition functions 
$\phi_{\ell,\ell'}: \cal U_{\ell}\cap \cal U_{\ell'} \lra \mathrm{Diff}(W)$ (valued in the stratified diffeomorphisms of $W$), we can lift these to
$\wt \phi_{\ell,\ell'}: \cal U_{\ell}\cap \cal U_{\ell'} \lra \mathrm{Diff}(\mathrm{res}(W))$ (valued in the fibered diffeomorphisms of $\mathrm{res}(W)$) 
and obtain transition functions for $\widehat p.$

Finally the grid resolution of $W \fib X \xlra{p} Y$ is obtained by performing both of these partial resolutions,
\begin{equation*}
	\mathrm{res}_{\mathrm{grid}}(X,p) = \mathrm{res}_{\mathrm{fib}}(\mathrm{res}_{\mathrm{base}}(X,p),\bar p)
\end{equation*}
and so fits into the diagram 
\begin{equation}\label{eq:DefnResolvp}
	\xymatrix{ 
	\mathrm{res}_{\mathrm{grid}}(X,p) \ar[r]^-{\bar\beta_X} \ar@{-->}[rd]_-{\wt p} &
	\beta_Y^*X \ar[d]_-{\bar p} \ar[r] \pullbackcorner & X \ar[d]^-p \\
	& \mathrm{res}(Y) \ar[r]^-{\beta_Y} & Y. }
\end{equation}

\begin{thm} \label{thm:ResolvedFibrations}
If $W \fib X \xlra{p} Y$ is a fiber bundle of smoothly stratified spaces then its partial resolutions corresponding to the base and fiber inherit
a horizontal and vertical iterated fibration structure, respectively. Its grid resolution thus has a grid iteration fibration structure.
\end{thm}

\begin{proof}
To simplify notation we assume during this proof that two strata are comparable if and only if their closures intersect. Thus, e.g., the closure of $W_a$ is $W_{\leq a}.$\\

Let us start by establishing the vertical iterated fibration structure of the partial resolution corresponding to the fiber.

Without loss of generality let us assume that $Y$ is connected and let us denote by $\circ$ the maximum on $S(Y),$ so that $Y_{\circ}$
is equal to the regular part of $Y.$ 
Every stratum $W_a,$ $a \in S(W),$ corresponds to a stratum $X_{a\times \circ},$ where $a\times \circ \in S(W) \times S(Y) = S(X),$
and by restricting $p$ we obtain a fiber bundle of smoothly stratified spaces
\begin{equation*}
	W_{\leq a} \fib Z_{\leq a}:=X_{(\leq a) \times (\leq \circ)} \xlra{p_{\leq a} := p|_{X_{(\leq a)\times(\leq \circ)}}} Y.
\end{equation*}
Just as discussed above for $p,$ $p_{\leq a}$ lifts to the resolution of $Z_{\leq a}$ corresponding to the fibers,
\begin{equation*}
	\widehat p_{\leq a}: \mathrm{res}_{\mathrm{fib}}(Z_{\leq a}, p_{\leq a}) \lra Y.
\end{equation*}
Now if we recall from Theorem \ref{thm:ResSASS} (see \eqref{eq:paares}) that $\mathrm{res}(W_{\leq a})$ is the base of the boundary fiber bundle of $W$ corresponding to $a,$ $\pa_a\mathrm{res}(W) \xlra{\phi_a} B_a\mathrm{res}(W),$ then since resolving the stratum $X_{a \times \circ}$ of $X$ resolves the corresponding stratum
of each fiber, we see that we have a commutative diagram
\begin{equation*}
	\xymatrix{ 
	\pa_a\mathrm{res}_{\mathrm{fib}}(X,p) \ar@{^(->}[r] \ar[d]_-{\phi_a} & \mathrm{res}_{\mathrm{fib}}(X) \ar[dd]^-{\widehat p}\\
	\mathrm{res}(Z_{\leq a}) \ar[rd]_-{\widehat p_{\leq a}} & \\
	& Y
	}
\end{equation*}
and hence the vertical iterated fibration structure we expected.

Next let us establish the horizontal iterated fibration structure of the partial resolution corresponding to the base.
Let us momentarily introduce the notation
\begin{equation*}
	\wt X = \mathrm{res}_{\mathrm{base}}(X,p).
\end{equation*}

For each $\alpha \in S(Y)$ the stratum $Y_{\alpha}$ of $Y$ has closure $Y_{\leq \alpha}$ and we denote the restriction of $p$ to this closure by
\begin{equation*}
	\xymatrix{
	X|_{\leq\alpha} \ar@{^(->}[r] \ar[d]_-{ p|_{\leq \alpha}} & X \ar[d]^-{ p} \\
	Y_{\leq \alpha}  \ar@{^(->}[r] & Y }
\end{equation*}
where $X|_{\leq\alpha} = p^{-1}(Y_{\leq\alpha}).$ Next define
\begin{equation*}
	B_\alpha^h\wt X:= \mathrm{res}_{\mathrm{base}}(X|_{\leq\alpha},p|_{\leq\alpha}):=\beta_{Y_{\leq\alpha}}^*X|_{\leq\alpha}, 
	\quad \text{ where } \quad
	\xymatrix{ 
	\beta_{Y_{\leq\alpha}}^*X|_{\leq\alpha}  \ar[d]_-{\overline p_{\leq \alpha}} \ar[r] \pullbackcorner & X|_{\leq \alpha} \ar[d]^-{p|_{\leq\alpha}} \\
	\mathrm{res}(Y_{\leq \alpha}) \ar[r]^-{\beta_{Y_{\leq\alpha}}} & Y_{\leq\alpha}. }
\end{equation*}

Recall that $\mathrm{res}(Y_{\leq\alpha})$ is the base of a fiber bundle projection in the iterated fibration structure of $\mathrm{res}(Y),$
\begin{equation*}
	\pa_\alpha \mathrm{res}(Y) \xlra{\phi_\alpha} B_{\alpha}\mathrm{res}(Y) = \mathrm{res}(Y_{\leq\alpha}),
\end{equation*}
and let us define $\pa_{\alpha}^h\wt X:=\overline p^{-1}(\pa_\alpha \mathrm{res}(Y) ),$ where $\overline p$ is the map in \eqref{eq:ResBaseDef}.
We claim that there is a fiber bundle map $\phi_{\alpha}^h: \pa_{\alpha}^h\wt X \lra B_\alpha^h\wt X$ that participates in a pull-back diagram
\begin{equation}\label{eq:ResHasHorIFS}
	\xymatrix{
	\pa_{\alpha}^h\wt X \ar[r]^-{\phi_{\alpha}^h} \ar[d]_-{\overline p} \pullbackcorner & B_\alpha^h\wt X \ar[d]^-{\overline p_{\leq\alpha}}\\
	\pa_\alpha \mathrm{res}(Y) \ar[r]^-{\phi_\alpha} & B_{\alpha}\mathrm{res}(Y). }
\end{equation}

To prove the claim,
we will use that $\mathrm{res}(Y)$ is obtained from $Y$ by performing a radial blow-up of the singular strata of $Y$ in any non-decreasing order.
If we denote by $\wt Y_1$ the space obtained from $Y$ by blowing-up all of the strata that come before $Y_{\alpha},$ then the closure of the 
pre-image of $Y_{\alpha}$ in $\wt Y_1$ is precisely $\mathrm{res}(Y_{\leq\alpha})=B_{\alpha}\mathrm{res}(Y)$ 
since we have blown-up all of the singular strata of $Y_{\leq\alpha}.$ Let us denote the partial blow-down map $\wt Y_1 \lra Y$ by $\gamma_1,$ so we have
\begin{equation*}
	\xymatrix{
	B_{\alpha}\mathrm{res}(Y) \ar@{^(->}[r] \ar[d]_-{\beta_{Y_{\leq\alpha}}} & \wt Y_1 \ar[d]^-{\gamma_1} \\
	Y_{\leq\alpha} \ar@{^(->}[r] & Y.}
\end{equation*}
Next let $\wt Y_2$ be the space obtained from $\wt Y_1$ by blowing-up $B_{\alpha}\mathrm{res}(Y),$ 
let $\gamma_2:\wt Y_2 \lra \wt Y_1$ denote the blow-down map, and let 
$\pa_\alpha\wt Y_2 = \gamma_2^{-1}(B_{\alpha}\mathrm{res}(Y))$ be the boundary hypersurface resulting from the blow-up.
The restriction of $\gamma_2$ to $\pa_\alpha\wt Y_2$ is a fiber bundle projection onto $B_{\alpha}\mathrm{res}(Y)),$ which we denote by $\gamma_{2,\alpha}.$
Finally, as $\mathrm{res}(Y)$ is obtained by performing blow-ups on $\wt Y_2,$ there is a blow-down map $\gamma_3:\mathrm{res}(Y) \lra \wt Y_2$ 
whose restriction to $\pa_\alpha\mathrm{res}(Y)$ is a blow-down map $\gamma_{3,\alpha}: \pa_\alpha\mathrm{res}(Y) \lra \pa_\alpha \wt Y_2.$
Thus we have
\begin{equation*}
	\xymatrix{
	\mathrm{res}(Y) \ar@/^2pc/[rrr]_-{\beta_Y} \ar[r]_-{\gamma_3} & \wt Y_2 \ar[r]_-{\gamma_2} & \wt Y_1 \ar[r]_-{\gamma_1} & Y \\
	\pa_\alpha \mathrm{res}(Y) \ar@/_2pc/[rr]^-{\phi_\alpha} \ar@{^(->}[u] \ar[r]^-{\gamma_{3,\alpha}} & \pa_\alpha\wt Y_2 \ar@{^(->}[u] \ar[r]^-{\gamma_{2,\alpha}} 
		& B_{\alpha}\mathrm{res}(Y)  \ar@{^(->}[u] \ar[r]^-{\beta_{Y_{\leq\alpha}}} & Y_{\leq \alpha}\phantom{\wt A}  \ar@{^(->}[u]}
\end{equation*}
Let us denote the composition of the maps along the bottom of this diagram by $\psi_{\alpha}:\pa_{\alpha}\mathrm{res}(Y) \lra Y_{\leq \alpha}.$
Now notice that both of the fiber bundle maps
\begin{equation*}
	\bar p: \beta_Y^*X \lra \mathrm{res}(Y) \quad  \text{ and } \quad
	\overline{p}_{\leq\alpha}: \beta_{Y_{\leq\alpha}}^*X|_{\leq\alpha} \lra \mathrm{res}(Y_{\leq\alpha}) = B_{\alpha}\mathrm{res}(Y)
\end{equation*}
are obtained by pulling-back $p: X \lra Y$ in the former case and $p$ restricted to $Y_{\leq\alpha}$ in the latter case.
Thus they are related by 
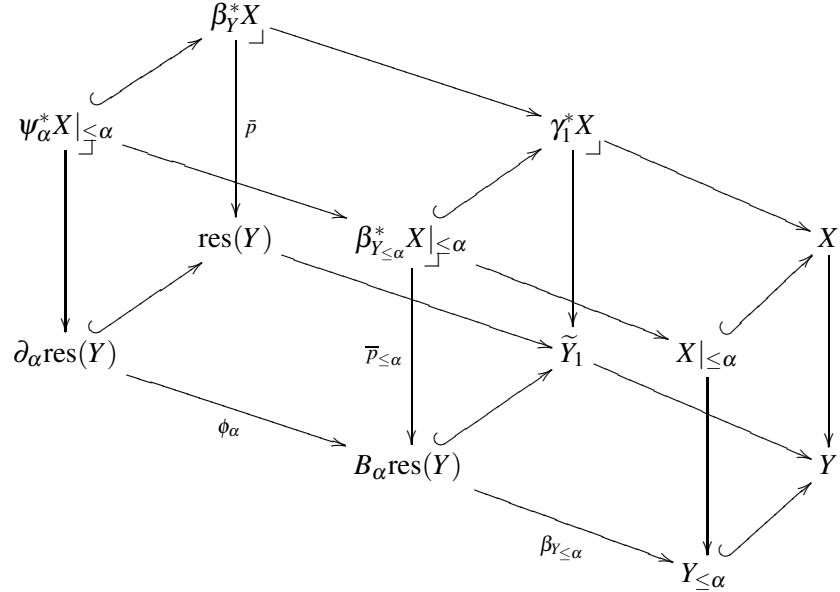
\begin{equation*}
	\xymatrix{
	& \beta_Y^*X \ar[dd]^-{\bar p} \ar[rrd] \pullbackcorner & & & & \\
	\psi_\alpha^*X|_{\leq\alpha} \ar@{^(->}[ru] \ar[dd] \ar[rrd] \pullbackcorner & & & \gamma_1^*X \ar[dd] \ar[rrd] \pullbackcorner &  &\\
	& \mathrm{res}(Y) \ar[rrd] & \beta_{Y_{\leq\alpha}}^*X|_{\leq\alpha} \ar@{^(->}[ru] \ar[dd]_-{\overline{ p}_{\leq\alpha}}  \ar[rrd] \pullbackcorner & & & X \ar[dd] \\
	\pa_\alpha\mathrm{res}(Y) \ar@{^(->}[ru]  \ar[rrd]_-{\phi_\alpha} & & & \wt Y_1 \ar[rrd] & X|_{\leq\alpha} \ar@{^(->}[ru] \ar[dd] & \\
	& & B_{\alpha}\mathrm{res}(Y) \; \ar@{^(->}[ru] \ar[rrd]_-{\beta_{Y_{\leq\alpha}}} & & & Y \\ 
	& & & & Y_{\leq\alpha} \; \ar@{^(->}[ru] & }
\end{equation*}
from which we recognize that the restriction of $\bar p: \beta_Y^*X \lra \mathrm{res}(Y)$ to the pre-image of $\pa_\alpha \mathrm{res}(Y)$ 
is equal to the pull-back of $\overline{p}_{\leq\alpha}: \beta_{Y_{\leq\alpha}}^*X|_{\leq\alpha} \lra B_{\alpha}\mathrm{res}(Y)$ along $\phi_\alpha.$ That is
\begin{equation*}
	\xymatrix{
	\beta_Y^*X|_{\pa_\alpha\mathrm{res}(Y)} \ar[r] \ar[d]^-{\overline p} \pullbackcorner & \beta_{Y_{\leq\alpha}}^*X|_{\leq\alpha}  \ar[d]^-{\overline{p}_{\leq \alpha}}\\
	\pa_\alpha\mathrm{res}(Y) \ar[r]^-{\phi_\alpha} & B_{\alpha}\mathrm{res}(Y) }
\end{equation*}
is, as indicated, a pull-back diagram. Recognizing that this diagram is the same as \eqref{eq:ResHasHorIFS} establishes the claim and produces the desired horizontal iterated fibration structure.
\end{proof}

It will happen that we will start out with a fiber bundle of smoothly stratified spaces
\begin{equation*}
	W \fib X \xlra{p} Y
\end{equation*}
then we will resolve it to some extent in order to carry out analytic constructions but we will then want to make conclusions that refer back
to the original stratified spaces. For this reason it is convenient to be able to recognize the continuous functions on the stratified spaces among
the continuous functions on the resolutions.

Analogously to how we defined $\cal C_{\Phi}(\mathrm{res}(X))$ in \eqref{eq:DefCPhi}, we can define
\begin{equation}\label{eq:DefCPhi-v}
	\cal C_{\Phi-v}(\mathrm{res}_{\mathrm{fib}}(X)) = \{ f \in \cal C(\mathrm{res}_{\mathrm{fib}}(X)): \\
	\text{ for all } a\in S(W), \text{ we have } f|_{\pa_a^v\mathrm{res}_{\mathrm{fib}}(X)} \in (\phi_a^v)^*\cal C(B_a^vX) \},
\end{equation}
\begin{equation*}
	\cal C_{\Phi-h}(\mathrm{res}_{\mathrm{base}}(X)) = \{ f \in \cal C(\mathrm{res}_{\mathrm{base}}(X)): \\
	\text{ for all }\alpha\in S(Y), \text{ we have } f|_{\pa_\alpha^h\mathrm{res}_{\mathrm{base}}(X)} \in (\phi_\alpha^h)^*\cal C(B_\alpha^hX) \Big\},
\end{equation*}
and the analogous $\cal C_{\Phi-h, \Phi-v}(\mathrm{res}_{\mathrm{grid}}(X)),$ and these too will be $*$-isomorphic to $\cal C(X).$ 
$ $\\
Finally we point out that the Kottke-Rochon result mentioned above, \cite[Theorem 5.1]{KotRoc:PMWFC}, extends to fiber bundles. That is to say
\begin{equation*}
	\mathrm{res}(X) = [\mathrm{res}_{\mathrm{grid}}(X); \text{ codim $2$ corners }]
\end{equation*}
where the codimension two corners are blown-up in any order consistent with the partial order on $S(X) = S(W) \times S(Y).$ Indeed the proof of this result
in {\em loc cit} is coordinate based so it continues to hold in this setting.

Similarly, by choosing a connection for $\mathrm{res}_{\mathrm{grid}}(X) \xlra{\wt p} \mathrm{res}(Y)$ we can consider a Riemannian metric on the interior of
$\mathrm{res}_{\mathrm{grid}}(X)$ of the form 
\begin{equation*}
	g_{\mathrm{res}_{\mathrm{grid}}(X) / \mathrm{res}(Y)} \oplus \wt p^*g_{\mathrm{res}(Y)}
\end{equation*}
where the first summand is a family of wedge metrics on the fibers of $\wt p$ and the second summand is the lift of a wedge metric on $\mathrm{res}(Y).$
The argument in \cite[Alternate proof of Theorem 6.8]{KotRoc:PMWFC} shows that this metric lifts to $\mathrm{res}(X)$ to be a non-degenerate and non-singular 
bundle metric on
\begin{equation*}
	\beta_{\mathrm{res}(X)}^*\Big( 
	{}^wT\mathrm{res}_{\mathrm{grid}}(X) / \mathrm{res}(Y) \oplus {}^wT\mathrm{res}(Y)
	\Big),
\end{equation*}
where $\beta_{\mathrm{res}(X)}: \mathrm{res}(X) \lra \mathrm{res}_{\mathrm{grid}}(X)$ is the blow-down map.
As pointed out in {\em loc cit}, it follows that this pull-back bundle is isomorphic to the wedge tangent bundle of $\mathrm{res}(X),$
\begin{equation}\label{eq:WedgePullBackBundle}
	{}^wT\mathrm{res}(X) = 
	\beta_{\mathrm{res}(X)}^*\Big( 
	{}^wT\mathrm{res}_{\mathrm{grid}}(X) / \mathrm{res}(Y) \oplus {}^wT\mathrm{res}(Y)
	\Big).
\end{equation}
(Note that $\mathrm{res}_{\mathrm{grid}}(X)$ generally does not have an iterated fibration structure and hence does not have a wedge tangent bundle.)

In fact, Kottke-Rochon show that if the vertical and horizontal wedge metrics are product-type (which they call rigid), then so is the lift to $\mathrm{res}(X).$
It follows that the same holds for totally-geodesic wedge metrics.

\subsection{Vertical wedge Dirac-type operators and the analytic transfer map} \label{sec:AnTransferMap}

In this section we gather together the final pieces needed to assign a KK-class to a fiber bundle of oriented Witt pseudomanifolds.
More generally, we explain how a family of wedge Dirac-type operators defines a KK-class. Starting with a fiber bundle of
stratified spaces, we pass to the grid resolution to apply results from \cite{AlbGel:IFFDTOP} but we want the KK-class
to correspond to the initial stratified spaces. This is where the structure developed above comes in as it allows us to show
that the $C^*$-algebras of functions that we use on the resolved spaces can be identified with the continuous functions on
the stratified spaces.\\

Our first step is to describe the structure of wedge differential forms on a fiber bundle of manifolds with corners $L \fib M \xlra\psi N$ endowed
with a grid iteration fibration structure.
We shall have need of the {\bf vertical wedge cotangent bundle}, defined as follows.
First recall that the vertical tangent bundle, which we denote $TM/N \lra M,$ is the vector bundle given by the null space of $D\psi$ and its restriction to any fiber 
$L_q = \psi^{-1}(q)$ is equal to its tangent bundle $TL_q.$ The vertical cotangent bundle, which we will denote $T^*M/N \lra M,$ 
is the dual bundle to the vertical tangent bundle.
Finally, to define the vertical wedge cotangent bundle we proceed as above to first describe the vertical wedge one-forms,
\begin{equation*}
	\cal V_{w}^*(M/N) 
	= \{ \theta \in \cal C^{\infty}(M; T^*M/N): \text{ for each $a\in S(L)$ and $y \in B_a^vM,$} \quad
		\theta|_{\phi_a^{-1}(y)}=0\},
\end{equation*}
and then appeal to the Serre-Swan theorem to obtain the associated bundle, denoted 
\begin{equation*}
	{}^w T^*M/N \lra M.
\end{equation*}
Next we consider how this construction relates to the horizontal iterated fibration structure.

\begin{prop}\label{prop:WedgeFormsFormHorSys}
Let $L \fib M \xlra\psi N$ be a fiber bundle of manifolds with corners  endowed with a grid iteration fibration structure. 

For every $\alpha \in S(N)$ and every $k\geq0,$ we have a pull-back diagram
\begin{equation*}
	\xymatrix{ 
	\Lambda^k ({}^w T^*M/N)|_{\pa_\alpha N} \pullbackcorner \ar[r] \ar[d] & \Lambda^k ( {}^w T^*  B_{\alpha}^hM / B_{\alpha}N )\ar[d]\\
	\pa_\alpha^hM \ar[r]^-{\phi_\alpha^h} & B_\alpha^hM }
\end{equation*}
where $\phi_{\alpha}^h: \pa_{\alpha}^hM \lra B_{\alpha}^hM$ is the fiber bundle from Definition \ref{def:HorVertIFS}(1).
\end{prop}

\begin{proof}
We first note that commutativity of the diagram \eqref{eq:defHorPullBackDiag} implies that the differential of $\phi_\alpha^h,$ 
$D\phi_\alpha^h:T\pa_\alpha^hM \lra TB_{\alpha}^hM $
restricts to a map between the vertical tangent bundles,
\begin{equation*}
	T(\pa_\alpha^hM/\pa_\alpha N)\lra T(B_{\alpha}^hM / B_{\alpha}N ).
\end{equation*}
As \eqref{eq:defHorPullBackDiag} is a pull-back diagram, $D\phi_\alpha^h$ sends each fiber of the vertical bundle of $\psi$ 
isomorphically onto the corresponding fiber of the vertical bundle of $\psi_{\alpha}^h.$ Indeed, if 
$\cal U \subseteq B_{\alpha}N$ is an open set such that we may trivialize $\psi_{\alpha}^h$ over $\cal U$ and 
$\overline p$ over $\phi_\alpha^{-1}(\cal U)$ then the diagram restricts to
\begin{equation}\label{eq:LocalizedPullback}
	\xymatrix{
	\phi_\alpha^{-1}(\cal U) \times L  \ar[rr]^-{\phi_\alpha \times \mathrm{id}} \ar[d]  &  & \cal U \times L \ar[d] \\
	\phi_\alpha^{-1}(\cal U) \ar[rr]^-{\phi_\alpha} & & \cal U }
\end{equation}
where $L$ is the typical fiber of $\psi_{\alpha}^h$ (and $\psi$), and then the claimed behavior of $D\phi_{\alpha}^h$ is clear.
It follows that the vertical tangent bundles participate in a pull-back diagram
\begin{equation}\label{eq:TangentPullback}
	\xymatrix{
	T(\pa_\alpha^hM/\pa_\alpha N) \ar[r] \ar[d] \pullbackcorner & T(B_{\alpha}^hM / B_{\alpha}N)  \ar[d]\\
	\pa_\alpha^hM \ar[r]^-{\phi_\alpha^h} & B_{\alpha}^hM, }
\end{equation}
i.e., $(\phi_\alpha^h)^*T(B_{\alpha}^hM / B_{\alpha}N) = T(\pa_\alpha^hM/\pa_\alpha N).$

We can then obtain a bundle morphism 
\begin{equation*}
	(\phi_\alpha^h)^*T^*(B_{\alpha}^hM / B_{\alpha}N) \lra ((\phi_\alpha^h)^*T(B_{\alpha}^hM / B_{\alpha}N))^*
	= T^*(\pa_\alpha^hM/\pa_\alpha N)
\end{equation*}
by noting that any element 
\begin{equation*}
	\theta \in (\phi_\alpha^h)^*T^*(B_{\alpha}^hM / B_{\alpha}N))|_{\zeta \in \pa_\alpha^hM}
\end{equation*}
defines a functional on $T(\pa_\alpha^hM/\pa_\alpha N)|_{\zeta\in \pa_\alpha^hM}$
since we have identified it  with $T(B_{\alpha}^hM / B_{\alpha}N) |_{\phi_\alpha^h(\zeta)\in B_{\alpha}^hM}.$ 
This bundle morphism restricts to an isomorphism between fibers and hence justifies the pull-back diagram
\begin{equation*}
	\xymatrix{
	T^*(\pa_\alpha^hM/\pa_\alpha N) \ar[r] \ar[d] \pullbackcorner & T^*(B_{\alpha}^hM / B_{\alpha}N) \ar[d]\\
	\pa_\alpha^hM \ar[r]^-{\phi_\alpha^h} & B_{\alpha}^hM. }
\end{equation*}
Again appealing to \eqref{eq:LocalizedPullback}, we see that this identification restricts to a bijection between the vertical wedge one-forms on $\pa_\alpha^hM$
and the vertical wedge one-forms on $B_{\alpha}^hM$ and so induces a bundle isomorphism  
\begin{equation*}
	(\phi_\alpha^h)^*({}^w T^*(B_{\alpha}^hM / B_{\alpha}N))  \cong {}^w T^*(\pa_\alpha^hM/\pa_\alpha N).
\end{equation*}
Finally, since pull-back commutes with taking exterior products, the statement of the proposition follows.
\end{proof}

A useful consequence of this proposition is that we can define vertical wedge bundles of the partial resolution corresponding to the fibers of a fiber 
bundle of stratified spaces.

\begin{remark}
Let us consider what this says in the setting of the prelude from section \ref{sec:ResStratFibBdles}. We start with $W \fib X \xlra p Y,$ a fiber bundle of smoothly stratified spaces where the base and fiber have stratifications of depth one and we denote its grid resolution by
\begin{equation*}
	\wt W = \mathrm{res}(W) \fib \wt X =  \mathrm{res}_{\mathrm{grid}}(X) \xlra{\wt p} \wt Y = \mathrm{res}(Y).
\end{equation*}
As this is a fiber bundle of smooth manifolds, its vertical tangent bundle is $T\wt X \diagup \wt W = \mathrm{Ker}(D\wt p) \subseteq T\wt X.$ One consequence of (the proof of) Proposition \ref{prop:WedgeFormsFormHorSys} that we want to point out is that there is also vertical tangent bundle even if we only resolve the fibers. Indeed, there is a natural map $\mathrm{res}_{\mathrm{grid}}(X) \xlra{\beta^h_X} \mathrm{res}_{\mathrm{fib}}(X)$ which collapses the fibers of $\pa_0^h\wt X,$ i.e., replaces $\pa_0^h\wt X$ with $B_0^h\wt X,$ and \eqref{eq:TangentPullback} shows that the vertical tangent bundle is the pull-back of a vector bundle on $\mathrm{res}_{\mathrm{fib}}(X)$ along $\beta^h_X.$ The proposition points out that this is true also for the wedge cotangent bundle.
\end{remark}

\begin{defn} [Vertical cotangent bundle of vertical iterated fibration structures] $ $

Let  $L \fib X \xlra{p} Y$ be a fiber bundle of smoothly stratified spaces, in which the fiber is a smooth manifold with corners, endowed with a vertical
iterated fibration structure, let
\begin{equation*}
	L \fib M = \mathrm{res}_{\mathrm{base}}(X,p) \xlra{\overline p} N = \mathrm{res}(Y)
\end{equation*}
be the partial resolution corresponding to the base, and let $\beta_X: M \lra X$ denote the blow-down map. Since $M\xlra{\overline p} N$ has a grid
iterated fibration structure, it follows from Proposition \ref{prop:WedgeFormsFormHorSys} that there is a vector bundle over $X,$ which we will call the {\bf vertical wedge cotangent bundle of $X \xlra{p} Y$} and denote ${}^w T^*X/Y \lra X,$ which pulls-back along $\beta_X$ to the vertical wedge cotangent
bundle of  $M \xlra{\overline p} N,$
\begin{equation*}
	\xymatrix{ 
	{}^w T^*M/N \pullbackcorner \ar[r] \ar[d] &{}^w T^*X/Y \ar[d]\\
	M \ar[r]^-{\beta_X} & X. }
\end{equation*}
The bundle ${}^w T^*X/Y \lra X$ is smooth as a fiber bundle of smoothly stratified spaces and restricts to each fiber of $p$ to the cotangent bundle
of $L,$
\begin{equation*}
	{}^w T^*X/Y|_{p^{-1}(y)} = {}^w T^*L_y.
\end{equation*}
\end{defn}

Of course one can similarly define the vertical wedge tangent bundle and the bundles of vertical wedge differential forms of a vertical iterated
fibration structure.\\

Our next objective is to obtain KK-classes from vertical families of Dirac-type operators associated to vertical iterated fibration structures, which
we will accomplish by applying results from \cite{AlbGel:IFFDTOP}. As this paper worked in the setting of fiber bundles of smooth manifolds with
corners we will also make use of grid iterated fibration structures.

\begin{defn} \label{def:CliffStr} $ $

\begin{enumerate} 
\item 
Let  $L \fib X \xlra{p} Y$ be a fiber bundle of smoothly stratified spaces, in which the fiber is a smooth manifold with corners, endowed with a vertical
iterated fibration structure, let $g_{X/Y}$ be a smooth family of (totally geodesic) wedge metrics on the fibers of $p,$ i.e., a bundle metric
on the vertical wedge tangent bundle, ${}^w TX/Y,$ which restricts to each fiber to be a totally geodesic wedge metric.
A {\bf $p$-vertical wedge Clifford structure} on $X$ consists of 
a smooth complex vector bundle $E \lra X$ endowed with a Hermitian metric $g_E,$ a metric connection $\nabla^E,$
and a bundle homomorphism (referred to as Clifford multiplication)
\begin{equation*}
	\cl:\bb C \otimes \mathrm{Cl}({}^wT^*X/Y, g_{X/Y}) \lra \mathrm{End}(E)
\end{equation*}
compatible with the metric and the connection.

\item Given a fiber bundle of smooth manifolds with corners, $L \fib M \xlra{\psi} N,$ endowed with a grid iterated fibration structure,
a {\bf horizontal system of vertical wedge Clifford structures} consists of a $\psi$-vertical wedge Clifford structure, 
$(g_{M/N}, E_M \lra M, g_{E_M}, \nabla^{E_M}, \cl),$ and, for each $\alpha \in S(N)$ with associated pull-back diagram of fiber bundles
\begin{equation*}
	\xymatrix{
	\pa_{\alpha}^hM \ar[r]^-{\phi_\alpha^h} \ar[d]_-\psi \pullbackcorner & B_{\alpha}^hM \ar[d]^-{\psi_\alpha^h} \\
	\pa_\alpha N \ar[r]^-{\phi_\alpha} & B_{\alpha}N, }
\end{equation*}
a $\psi_\alpha^h$-vertical wedge Clifford structure $(g_{B_\alpha^hM/B_{\alpha}N}, E_\alpha \lra B_{\alpha}^hM, g_{E_\alpha}, \nabla^{E_\alpha}, \cl),$ such that these satisfy
\begin{equation*}
	(\phi_\alpha^h)^*(g_{B_{\alpha}^hM/B_{\alpha}N}, E_\alpha \lra B_{\alpha}^hM, g_{E_\alpha}, \nabla^{E_\alpha}, \cl) 
	= (g_{M/N}, E_M \lra M, g_{E_M}, \nabla^{E_M}, \cl)|_{B_{\alpha}^hM}.
\end{equation*}
\end{enumerate}
\end{defn}

Note that Proposition \ref{prop:WedgeFormsFormHorSys} is implicitly used in asking for the Clifford multiplication on $\pa_\alpha^hM$ 
to be the pull-back of the Clifford multiplication on $B_{\alpha}^hM.$ The same proposition also shows that the bundle of wedge 
differential forms makes up part of a horizontal system of vertical wedge Clifford structures, which we can complete by
picking a suitable vertical wedge metric. The fact that this horizontal system is pulled-back from a vertical iterated fibration structure is no coincidence.
Indeed, immediately from the definition we see that if $L \fib X \xlra{p} Y$ has a vertical iterated fibration structure, 
\begin{equation*}
	L \fib M = \mathrm{res}_{\mathrm{base}}(X,p) \xlra{\overline p} N = \mathrm{res}(Y)
\end{equation*}
is the partial resolution corresponding to the base, and $\beta_X: M \lra X$ is the blow-down map, then any $p$-vertical wedge Clifford structure 
$(g_{X/Y}, E \lra X, g_{E}, \nabla^{E}, \cl)$ can be pulled-back along $\beta_X$ to produce a horizontal system of vertical wedge Clifford structures.
It is also easy to see that all horizontal systems can be obtained in this way.

Associated to a vertical iterated fibration structure, $L \fib X \xlra{p} Y,$ endowed with a $p$-vertical wedge Clifford structure is a family of wedge Dirac-type operators, 
\begin{equation*}
	Y \ni y \mapsto \eth^E_{L_y}
\end{equation*}
defined as in \S \ref{sec:WedgeMetsDiracOps}, which we denote $\eth^E_{X/Y}.$ We denote the realization of this family acting on $L^2$ sections of $E$ with the vertical APS domain by $\eth^E_{X/Y, VAPS}.$

Similarly if $L \fib M \xlra{\psi} N$ is a fiber bundle of smooth manifolds with corners endowed with a grid iterated fibration structure and a 
horizontal system of vertical wedge Clifford structures, we can define a family of wedge Dirac-type operators in the same way. We denote the resulting
family by $\eth^E_{M/N}$ and the realization of this family acting on $L^2$ sections of $E$ with the vertical APS domain by $\eth^E_{M/N, VAPS}.$

In either case we say that the family of Dirac-type operators satisfies the analytic Witt condition 
if the individual Dirac-type operators satisfy the condition
(as described in \S \ref{sec:WedgeMetsDiracOps}).

\begin{thm} {\cite[Theorem 1]{AlbGel:IFFDTOP}}
Let $L \fib M \xlra{\psi} N$ be a fiber bundle of smooth manifolds with corners endowed with a $\psi$-vertical iterated fibration structure and a $\psi$-vertical wedge Clifford structure. If the corresponding family of Dirac-type operators $\eth^E_{M/N,VAPS}$ satisfies the analytic Witt condition then it is a family of self-adjoint Fredholm operators with compact resolvent.
\end{thm}

This theorem does not require a grid iterated fibration structure but it does require that the spaces involved 
are all smooth manifolds with corners. However by applying it to
fiber bundles $L \fib M \xlra{\psi} N$  of smooth manifolds with corners that are endowed with a grid iterated fibration structures and a 
horizontal system of vertical wedge Clifford structures, we can obtain the following corollary.

\begin{cor} 
Let $L \fib X \xlra{p} Y$ be a fiber bundle of smoothly stratified spaces, in which the fibers are smooth manifolds with corners, endowed with a $p$-vertical wedge Clifford structure. 
If the corresponding family of Dirac-type operators $\eth^E_{X/Y,VAPS}$ satisfies the analytic Witt condition then it is a family of self-adjoint Fredholm operators with
compact resolvent.
\end{cor}

This corollary is the result we have been aiming for. Our aim has been to assign a KK-class to a fiber bundle of smoothly stratified Witt spaces by means of the vertical signature operators (to be realized in Corollary \ref{cor:SignKKClass} below). This leads us to using the grid resolution for a couple of reasons: first, defining Clifford structures requires tangent bundles and, second, the results (and proofs) in \cite{AlbGel:IFFDTOP} assume fiber bundles of smooth manifolds with corners. What we have shown is that we can start with a fiber bundle of smoothly stratified spaces, resolve only the fibers, and still make sense of and apply the results of \cite{AlbGel:IFFDTOP}. We will now carry out the usual repackaging of these results to obtain an unbounded cycle and define a KK-class. (The final result will be nicest if we start with a fiber bundle in which all of the space involved are potentially singular spaces and we assume that the Clifford structure is defined on the partial resolution of the fibers.)

Thus let $Z \fib X \xlra{p} Y$ be a fiber bundle of smoothly stratified spaces and let
\begin{equation*}
	L:= \mathrm{res}(Z) \fib \widehat X := \mathrm{res}_{\mathrm{fib}}(X,p) \xlra{\hat p} Y
\end{equation*}
be the partial resolution with respect to the fibers. 
Suppose that $L \fib \widehat X \xlra{\hat p} Y$ is endowed with a $p$-vertical wedge Clifford structure 
$(g_{\widehat X/Y}, E \lra \widehat X, g_{E}, \nabla^{E}, \cl)$ whose associated Dirac-type operators satisfy the analytic Witt condition.
Let us endow
\begin{equation*}
	\cal E' = \cal C_c(\hat X_{\mathrm{reg}};E)
\end{equation*}
with the $\cal C_0(Y)$-valued inner product defined by
\begin{equation*}
	\xymatrix @R=1pt
	{
	\langle \cdot, \cdot \rangle_{\cal E'}: & \cal E' \times \cal E' \ar[r] & \cal C_0(Y) \\
	& (f,h) \ar@{|->}[r] & \Big( y \mapsto \int_{p^{-1}(y)} g_E(f(\zeta), h(\zeta))\; \mathrm{dvol}_{g_{L_y}} \Big)
	}
\end{equation*}
and then define $\cal E$ to be the closure of $\cal E'$ with respect to the resulting norm.
Since $\cal E$ is a Hilbert module over $\cal C_0(Y)$ it can be identified with the space of continuous sections vanishing at infinity of a continuous field
of Hilbert spaces over $Y.$ In this case, of the  Hilbert bundle
\begin{equation*}
	L^2(L;E) - L^2(\widehat X/Y;E) \lra Y
\end{equation*}
induced by $p,$ so that we have
\begin{equation*}
	\cal E = \cal C_0(Y;L^2(\widehat X/Y;E)).
\end{equation*}
We let $\cal C_{\Phi-v}(\widehat X)$ (defined in \eqref{eq:DefCPhi-v}) act on $\cal E$ by left multiplication and $\cal C(Y)$ act on $\cal E$ by right
multiplication and we obtain a $\cal C_{\Phi-v}(\widehat X)$-$\cal C(Y)$ Hilbert bi-module, $\bb Z_2$-graded if $E$ is $\bb Z_2$-graded.
Define $\eth^{\cal E}_{\widehat X/Y}$ to be the operator on $\cal E$ given by 
\begin{multline}\label{eq:DefethCalE}
	(\eth^{\cal E}_{\widehat X/Y}\omega)(y) = \eth_{L_y,VAPS}(\omega(y)), \\
	\text{ for all }
	\omega \in \mathrm{dom}(\eth^{\cal E}_{\widehat X/Y}) = \{ \omega \in \cal E: \omega(y) \in \cal D_{VAPS}(\eth^E_{L_y})
		\; \& \; \eth^{\cal E}_{\widehat X/Y}\omega \in \cal E\}.
\end{multline}
Since $\eth^E_{\widehat X/Y,VAPS}$ is closed, densely defined and $\mathrm{dom}(\eth^{\cal E}_{\widehat X/Y})|_{y\in Y}$ is dense in 
$\mathrm{dom}(\eth^E_{L_y})$ (in fact they are equal), it follows from \cite[Proposition 2.9]{Hil:FKBPVL} that $\eth^{\cal E}_{\widehat X/Y}$ 
is self-adjoint and regular. 

Next note that elements $h\in \cal C^1_{\Phi-v}(\widehat X)$ are such that $[\eth^{\cal E}_{\widehat X/Y},h]$ is adjointable and bounded
and such that multiplication by $h$ preserves the VAPS-domain. Hence 
$\cal C^1_{\Phi-v}(\widehat X)\subseteq \mathrm{Lip}(\eth^{\cal E}_{\widehat X/Y})$ and 
$\mathrm{Lip}(\eth^{\cal E}_{\widehat X/Y}) = \mathrm{Lip}_{\cal K}(^{\cal E}_{\widehat X/Y})$ since the resolvent of each $\eth^E_{L_y,VAPS}$ is 
compact. Thus the closure of $\mathrm{Lip}_{\cal K}(^{\cal E}_{\widehat X/Y})$ contains the multiplication operators acting by elements of 
$\cal C_{\Phi-v}(\widehat X).$ This shows that these data define an unbounded $\cal C_{\Phi-v}(\widehat X)$-$\cal C(Y)$-cycle,
\begin{equation*}
	[\cal C_0(Y;L^2(\widehat X/Y;E)), \eth^{\cal E}_{\widehat X/Y}] \in \overline{\mathrm{UKK}}^{\dim Z}(\cal C_{\Phi-v}(\widehat X), \cal C(Y)).
\end{equation*}
If $X$ is not compact, but the links of $Z$ are compact, then as discussed at the end of \S\ref{sec:WedgeMetsDiracOps} we can define a KK-class using
functions that vanish at infinity.
Finally, since $\cal C_{\Phi-v}(\widehat X)$ is $*$-isomorphic to $\cal C(X)$ and the unbounded $\mathrm{KK}$-groups coincide with those defined
by Kasparov we can summarize this as folllows.

\begin{cor}\label{cor:GralKKClass}
If $Z \fib X \xlra{p} Y$ is a fiber bundle of smoothly stratified spaces, $\mathrm{res}_{\mathrm{fib}}(X,p) \xlra{\hat p} Y$ is endowed with a $\hat p$-vertical wedge Clifford structure whose associated Dirac-type operators satisfy the analytic Witt condition and $\cal E,$ $\eth^{\cal E}_{\widehat X/Y}$ are defined as above then these data define
\begin{equation*}
	[\eth^{\cal E}_{\widehat X/Y}] \in \mathrm{KK}^{\dim X/Y}(\cal C_0( X), \cal C_0(Y)).
\end{equation*}
\end{cor}

The case of most interest to us comes from specializing this corollary to the signature operator.

\begin{cor}\label{cor:SignKKClass}
To any smooth oriented fiber bundle $p: X \lra Y$ of stratified spaces, whose typical fiber $Z$ is a Witt pseudomanifold, there is associated a KK-class
\begin{equation*}
	[D^{\mathrm{sign}}_{X/Y}] \in \mathrm{KK}^{\dim X/Y}(\cal C_0(X), \cal C_0(Y))
\end{equation*}
given by the KK-class of the family of signature operators associated to any totally geodesic vertical wedge metric.
\end{cor}

\begin{proof}
For any choice of totally geodesic vertical wedge metric Corollary \ref{cor:GralKKClass} defines a KK-class from the family of signature operators
of the resolutions of the fibers of $p,$ so we only need to justify that this class is independent of the choice of metric, and so it suffices to note that 
any two appropriate totally geodesic vertical wedge metric can be connected by a homotopy which induces a homotopy of the corresponding families of
signature operators and hence does not affect the KK-class.
\end{proof}

Just as we defined the analytic orientation by normalizing the K-homology class of the signature operator after inverting $2,$ so too do we define the analytic transfer class in KK-theory as a normalized version of the class of the vertical family of signature operators after inverting $2.$ This has two advantages: first, product formulas will hold `on the nose' without occasional factors of two (see Corollary \ref{cor:FunctorialityFibrations}) and secondly, to be consistent with the K-homology class of the signature operator when $Y$ is a point (see Proposition \ref{prop.gysinelemissignopclass}). 

\begin{defn} \label{def:AnalyticTransferClass}
Let $p: X \lra Y$ be a smooth oriented fiber bundle of stratified spaces whose typical fiber $Z$ is a Witt pseudomanifold. We define {\bf the analytic transfer class associated to $p$ to be\footnote{NB: In the setting of Lipschitz manifolds, Hilsum \cite[\S 2.4]{Hil:FKBPVL} uses the same notation with a slightly different meaning: for him $\Sigma(p)$ denotes the class of $[D^{\mathrm{sign}}_{X/Y}].$}
\begin{equation*}
	\Sigma(p) = 2^{-\lfloor \dim Z /2 \rfloor} [D^{\mathrm{sign}}_{X/Y}] 
	\in \mathrm{KK}^{\dim X/Y}(\cal C_0(X), \cal C_0(Y))[\tfrac12].
\end{equation*}
}
\end{defn}

Another important special case is when the partial resolution corresponding to the fibers of the fiber bundle $Z \fib X \xlra{p} Y,$
\begin{equation*}
	\hat X = \mathrm{res}_{\mathrm{fib}}(X,p) \xlra{\hat p} Y
\end{equation*}
admits a vertical wedge spin structure; that is, when the vertical wedge tangent bundle ${}^w T\hat X/Y$ is spin.
Choosing a spin bundle and a vertical wedge metric we obtain a family of spin Dirac operators and if they satisfy the analytic Witt condition then we obtain a KK-class. 

\begin{cor} \label{cor:DefSpinKK}
To any smooth fiber bundle $p: X \lra Y$ of smoothly stratified spaces, whose vertical wedge tangent bundle ${}^w T\mathrm{res}_{\mathrm{fib}}(X,p)/Y$ is endowed with a spin structure, there is for each vertical wedge metric $g$ an associated a KK-class
\begin{equation*}
	[D^{\mathrm{spin}}_{X/Y}] \in \mathrm{KK}^{\mathrm{dim X/Y}}(\cal C_0(X), \cal C_0(Y))
\end{equation*}
given by the KK-class of the corresponding family of spin Dirac operators.
\end{cor}

Unlike for smooth manifolds, the class of the spin Dirac operator in KK-theory will depend on the choice of wedge metric. Indeed, whether or not the spin Dirac operator is Fredholm will depend on the choice of metric.

A natural condition that ensures that the analytic Witt condition holds is known as `psc-Witt'. Suppose that $(X,g)$ is a smoothly stratified pseudomanifold all of whose strata are spin, endowed with wedge metric $g,$ we say that the {\bf psc-Witt condition} holds if the induced metrics on the links of $X$ have positive scalar curvature (see \cite{BotPiaRos:PSCSPFGSI, BotPiaRos:PSCSCSP}). Since all of the strata are spin, the boundary family at each stratum consists of the spin Dirac operators of the links and hence the Lichnerowicz formula shows that the psc-Witt condition implies the analytic Witt condition.
A stronger sufficient condition is to assume that the wedge metric has non-negative scalar curvature near the singularities; this implies that the operators are also essentially self-adjoint (see, e.g., \cite[\S 7]{AlbGel:IDOIES}).
(One can analogously treat spin-c structures; the analogue of the psc-Witt condition is the `generalized psc-Witt condition' see \cite{BotRos:GPSCSM}.)\\

Note that if in the fiber bundle $Z \fib X \xlra{p} Y$ the base $Y$ consists of a single point then the construction of the unbounded 
$\cal C_{\Phi-v}(\widehat X)$-$\cal C(Y)$-cycle coincides with the construction of the KK-class associated to a single Dirac-type operator in 
\S\ref{sec:KKThy}. The special case of the signature operator is important enough for us that we label it as a proposition.

\begin{prop} \label{prop.gysinelemissignopclass}
Let $X^n$ be an orientable compact closed $n$-dimensional smoothly stratified Witt pseudomanifold. Then the Gysin element of the constant map
$h:X\to \pt$ agrees with the class of the signature operator,
\[ \Sigma(h) =  \mathrm{sign}_K(X) \;\;\text{in}\;\;\mathrm{K}_n(X)[\tfrac12] \]
\end{prop}

\begin{defn}
The {\bf analytic transfer maps} associated to any smooth oriented fiber bundle $p: X \lra Y$ of stratified spaces, whose typical fiber $Z$ is an oriented Witt pseudomanifold, are the maps induced by Kasparov product with $\Sigma(p).$
We denote the corresponding maps in K-homology by
\begin{equation*}
	\xymatrix @R=1pt {
	\mathrm K_*^{\an}(Y)[\tfrac12] 
	= \mathrm{KK}^{*}(C_0(Y), \bb C)[\tfrac12]  \ar[r]^-{p^!} & \mathrm{KK}^{*}(C_0(X),\bb C)[\tfrac12]  
	= \mathrm K_*^{\an}(X)[\tfrac12] \\
	[A] \ar@{|->}[r] & \Sigma(p) \otimes [A], }	
\end{equation*}
and the corresponding map in topological K-theory by
\begin{equation*}
	\xymatrix @R=1pt {
	\mathrm K^{*}(Y)[\tfrac12]  
	= \mathrm{KK}^{*}(\bb C, C_0(X))[\tfrac12]  \ar[r]^-{p_!} & \mathrm{KK}^{*}(\bb C, C_0(Y))[\tfrac12]  
	= \mathrm K^{*}(Y)[\tfrac12]  \\
	[A] \ar@{|->}[r] & [A] \otimes \Sigma(p) }	
\end{equation*}
where $\otimes$ denotes the Kasparov product.
\end{defn}

\subsection{Functoriality for fibrations}\label{subsect:functoriality-for-fibrations}

In this section our aim is to start with a commutative diagram of fiber bundles of smoothly stratified spaces 
\begin{equation}\label{eq:3fibrations}
	\xymatrix{
	X_1 \ar[r]^-{p_{12}} \ar[rd]_-{p_{13}} & X_2 \ar[d]^-{p_{23}} \\
	& X_3}
\end{equation}
and enough data to define appropriate Dirac-type operators and then to understand the relation between the corresponding 
$\mathrm{KK}$-classes.
In particular we will generalize \cite[Lemma 6]{RosWei:SO} and relate the analytic signature class of $p_{13}$ with the Kasparov product of those of $p_{12}$ and $p_{23}$ in Corollary \ref{cor:FunctorialityFibrations}.
An important special case corresponds to taking $X_3$ to be a single point: 

\begin{thm} \label{thm.bundletransferpreservesksignature}
If $X \xlra{p} Y$ is an oriented fiber bundle
map of Witt pseudomanifolds, then
\begin{equation*}
	\mathrm{sign}_K(X) = \Sigma(p) \otimes \mathrm{sign}_K(Y) = p^! \mathrm{sign}_K(Y) \quad
	\text{ in } \mathrm K_{\dim X}^{\mathrm{an}}(X)[\tfrac12]
\end{equation*}
i.e., the analytic transfer map takes the signature orientation of the base $Y$ to the signature orientation of the total space $X.$
\end{thm}

A special case of this latter result, when $X\lra Y$ is a real oriented vector bundle over a Witt pseudomanifold, was worked out by Hilsum in 
\cite[Th\'eor\`eme 4.4, pg. 187]{Hil:PDBDKP}.

Our proof of the functoriality, just as proofs of similar statements for smooth manifolds in, e.g., \cite{KaaSui:RSFDO,Dun:KPSOM},
consists of relating the corresponding Dirac-type operators and then using Kucerovsky's unbounded version of the Connes-Skandalis
characterization of the Kasparov product \cite{Kuc:KUM}.

$ $\\
We will show something more general than the relation between the analytic signature classes, namely that data sufficient to define a family of Dirac-type operators for $p_{12}$ and $p_{23}$ are also sufficient to define a family of Dirac-type operators for $p_{13}$ and that the resulting KK-classes are related by Kasparov product. To state this more precisely let us start by denoting the typical fiber of the map $p_{ij}$ by $F_{ij}$ so that
\begin{equation*}
\begin{gathered}
	\mathrm{res}(F_{12}) \fib \widehat X_1:= \mathrm{res}_{\mathrm{fib}}(X_1,p_{12}) \xlra{\widehat p_{12}} X_2, \\
	\mathrm{res}(F_{23}) \fib \widehat X_2:= \mathrm{res}_{\mathrm{fib}}(X_2,p_{23}) \xlra{\widehat p_{23}} X_3, \\
	\mathrm{res}(F_{13})= \mathrm{res}(F_{12}) \wt\times \mathrm{res}(F_{23}) \fib  
		\wt X_1:= \mathrm{res}_{\mathrm{fib}}(X_1,p_{13}) \xlra{\widehat p_{13}} X_3,
\end{gathered}
\end{equation*}
where $\wt\times$ denotes the ordered product of Kottke-Rochon \cite{KotRoc:PMWFC} mentioned in \S\ref{sec:ResStratFibBdles}.
Locally over an open set $\cal U \subseteq X_3$ on which we can trivialize all of the $p_{ij}$ we have natural maps
\begin{equation*}
	\xymatrix{
	\mathrm{res}(F_{12}) \wt\times \mathrm{res}(F_{23}) \times \cal U \ar[r] \ar[d] 
		& \mathrm{res}(F_{12}) \times F_{23} \times \cal U \ar[d] \\
	\mathrm{res}(F_{23}) \times \cal U \ar[r] & F_{23} \times \cal U \ar[d] \\ & \cal U }
\end{equation*}
and these maps fit together into the diagram
\begin{equation*}
	\xymatrix{
	\mathrm{res}_{\mathrm{fib}}(X_1, p_{13}) \ar[r]^-\alpha \ar[d]_-\beta & \mathrm{res}_{\mathrm{fib}}(X_1, p_{12}) \ar[d] \\
	\mathrm{res}_{\mathrm{fib}}(X_2, p_{23}) \ar[r] & X_2 \ar[d] \\ & X_3. }
\end{equation*}
The bundle 
\begin{equation}\label{eq:TX1/X3}
	 \alpha^* {}^wT^*\widehat X_1/X_2 \oplus \beta^* {}^wT^*\widehat X_2/X_3 \lra \mathrm{res}_{\mathrm{fib}}(X_1, p_{13})
\end{equation}
is, by \cite[Alternate proof of Theorem 6.8]{KotRoc:PMWFC} mentioned above \eqref{eq:WedgePullBackBundle}, the vertical wedge cotangent
bundle of $\mathrm{res}_{\mathrm{fib}}(X_1, p_{13}) \lra X_3,$ ${}^w T^*\wt X_1/X_3.$
Thus vertical wedge metrics for $\widehat X_1 \lra X_2$ and $\widehat X_2 \lra X_3$ induce a vertical wedge metric for $\wt X_1 \lra X_3.$

If we are given 
\begin{itemize}
	\item a vertical wedge Clifford structure 
		$(g_{\widehat X_1/X_2}, E_{12} \lra \widehat X_1, g_{E_{12}}, \nabla^{E_{12}}, \cl),$
		for $\widehat X_1 = \mathrm{res}_{\mathrm{fib}}(X_1,p_{12}) \xlra{\widehat p_{12}} X_2,$
	\item a vertical wedge Clifford structure 
		$(g_{\widehat X_2/X_3}, E_{23} \lra \widehat X_2, g_{E_{23}}, \nabla^{E_{23}}, \cl),$
		for $\widehat X_2 = \mathrm{res}_{\mathrm{fib}}(X_2,p_{23}) \xlra{\widehat p_{23}} X_3,$
\end{itemize}
then we define a bundle $E_{13} \lra \mathrm{res}_{\mathrm{fib}}(X_1, p_{13})$ by
\begin{equation*}
	E_{13} = \alpha^*E_{12} \widehat\otimes \beta^*E_{23},
\end{equation*}
where $\widehat\otimes$ denotes the graded tensor product\footnote{The graded tensor product of Clifford modules is discussed in, e.g., 
\cite[\S6]{AtiBotSha:CM}.}. 
There is an induced bundle metric $g_{E_{13}} = \alpha^*g_{E_{12}} \otimes \beta^*g_{E_{23}},$ an induced connection
\begin{equation*}
	\nabla^{E_{13},\oplus} = \alpha^*\nabla^{E_{12}} \otimes 1 + 1 \otimes \beta^*\nabla^{E_{23}},
\end{equation*}
and an induced Clifford action: at each $\zeta \in \wt X_1,$
\begin{equation*}
\begin{gathered}
	\text{ for any }\theta \otimes \eta \in 
	\bb Cl(\alpha^* {}^wT^*_{\alpha(\zeta)}\widehat X_1/X_2) \widehat\otimes \bb Cl( \beta^* {}^wT^*_{\beta(\zeta)}\widehat X_2/X_3), \\
	\text{ and }\sigma \otimes \tau \in \alpha^*(E_{12})_{\alpha(\zeta)} \widehat\otimes \beta^*(E_{23})_{\beta(\zeta)}, \\
	\cl(\theta\otimes\eta)(\sigma \otimes \tau) = (-1)^{\mathrm{deg}(\omega)\mathrm{deg}(\sigma)}\cl(\theta)\sigma \otimes \cl(\eta)\tau.
\end{gathered}
\end{equation*}
Thanks to the identification of ${}^w T^*\wt X_1/X_3$ with $\alpha^* {}^wT^*\widehat X_1/X_2 \oplus \beta^* {}^wT^*\widehat X_2/X_3,$ this
defines an action of the complexified Clifford algebra of ${}^w T^*\wt X_1/X_3$ on $E_{13}.$

Thus we almost have an induced vertical Clifford structure for $\wt X_1 \lra X_3.$ The only issue is that $\nabla^{E_{13},\oplus}$ is a Clifford connection
with respect to the connection
\begin{equation*}
	\nabla^{\wt X_1/X_3, \oplus} = \alpha^*\nabla^{\hat X_1/X_2} \otimes 1 + 1 \otimes \beta^*\nabla^{\hat X_2/X_3},
\end{equation*}
and this is generally not equal to the vertical Levi-Civita connection determined by $g_{\wt X_1/X_3},$ $\nabla^{\wt X_1/X_3}.$
If we denote the difference between the two connections by
\begin{equation*}
	\nabla^{\wt X_1/X_3} - \nabla^{\wt X_1/X_3, \oplus} = \omega
\end{equation*}
then $\omega$ is an one-form valued in the endomorphisms of vertical tangent bundle.
We can regard $\omega$ as a $T^*\wt X_1$-valued vertical wedge two form using the identification
\begin{equation*}
	(\omega(V_1, V_2), V_3) = g_{\wt X_1/X_3}( (\nabla^{\wt X_1/X_3}_{V_3} - \nabla^{\wt X_1/X_3, \oplus}_{V_3}) V_1, V_2),
\end{equation*}
where $V_3 \in \cal C^{\infty}(\wt X_1; T\wt X_1)$ and $V_1, V_2 \in \cal C^{\infty}(\wt X_1; {}^w T\wt X_1/X_3),$ and then define
\begin{equation*}
	\cl(\omega) \in \cal C^{\infty}(\wt X_1; T^*\wt X_1 \otimes \bb Cl({}^w T^*\wt X_1/X_3)), \quad
	\cl(\omega) = \tfrac12 \sum_{a,b} \omega(e_a, e_b) \otimes \cl(e^a)\cl(e^b),
\end{equation*}
where $e_a$ runs over an orthonormal basis of ${}^w T\wt X_1/X_3,$
and it follows from \cite[Proposition 10.12]{BerGetVer:HKDO} that
\begin{equation*}
	\nabla^{E_{13}} := \nabla^{E_{13},\oplus} + \cl(\omega)
\end{equation*}
is a Clifford connection with respect to $\nabla^{\wt X_1/X_3}.$ We refer to the data defined in this way,
$(g_{\wt X_1/X_3}, E_{13} \lra \wt X_1, g_{E_{13}}, \nabla^{E_{13}}, \cl),$ as
the induced vertical Clifford structure for the fiber bundle $\mathrm{res}_{\mathrm{fib}}(X_1, p_{13}) \lra X_3.$

Finally, if the vertical wedge Clifford structures for $\hat X_1 \xlra{\hat p_{12}} X_2$ and $\hat X_2 \xlra{\hat p_{23}} X_3$ are $\bb Z_2$-graded then we endow the vertical wedge Clifford structure for $\hat X_1 \lra{\hat p_{13}} X_3$ with the natural $\bb Z_2$-grading. If neither of these is $\bb Z_2$-graded then  we endow the vertical wedge Clifford structure for $\hat X_1 \lra{\hat p_{13}} X_3$ with the $\bb Z_2$-grading obtained by considering the ungraded Clifford structures as $\bb Cl(1)$ structures in the usual way (see, e.g., \cite[Appendix A]{SuiVer:IUKPESIES} and \cite[Examples 2.38-2.40]{BraMesSui:GTSTUKP}).

\begin{thm}\label{thm:FunctorialityFibrations}
Suppose we are given fiber bundles of smoothly stratified spaces as in \eqref{eq:3fibrations} and vertical Clifford structures for 
$\mathrm{res}_{\mathrm{fib}}(X_1, p_{12}) \lra X_2$ and $\mathrm{res}_{\mathrm{fib}}(X_2, p_{23}) \lra X_3,$ and suppose that the families of Dirac-type operators of these two families, as well as that of the induced vertical Clifford structure for $\mathrm{res}_{\mathrm{fib}}(X_1, p_{13}) \lra X_3,$  satisfy the analytic Witt condition. Then the associated $\mathrm{KK}$-classes satisfy 
\begin{equation*}
	[\eth^{\cal E_{13}}_{X_1/X_3}] = [\eth^{\cal E_{12}}_{X_1/X_2}] \otimes [\eth^{\cal E_{23}}_{X_2/X_3}] 
		\in \mathrm{KK}^{\dim X_1/X_3}(C_0(X_1), C_0(X_3)).
\end{equation*}
\end{thm}

\begin{proof}
For simplicity, let us first assume that $\dim X_1/X_2$ and $\dim X_2/X_3$ are even.

Let $(\cal E_{ij}, \eth^{\cal E_{ij}}_{X_i/X_j}),$ with $(i,j) \in \{ (1,2), (2,3), (1,3) \},$ denote the three unbounded cycles. Recall that 
$\cal E_{ij}$ is defined as the closure of $\cal C_c(X_{i,\mathrm{reg}};E_{ij})$ with respect to the $\cal C_0(X_j)$-valued inner product induced by 
integation along the fibers $X_i/X_j.$\\

\noindent
{\bf Step 1}: We can assume without loss of generality that, for each $\zeta \in X_2,$
\begin{equation}\label{eq:ThickWittDomain}
	\cal D_{VAPS}(\eth_{F_{12,\zeta}}^E) = \rho H^1_e(F_{12,\zeta};E).
\end{equation}
Indeed, it is pointed out in \cite[Remark 4.9]{AlbGel:IDOIES} that any family of wedge Dirac-type operators satisfying the analytic Witt condition is homotopic
through Fredholm families of wedge Dirac-type operators satisfying the analytic Witt condition to one whose ``indicial roots" are as large as desired.\footnote{As
noted in \cite[Remark 4.9]{AlbGel:IDOIES} the deformation is through Dirac-type operators but other structures are not necessarily preserved. For example
it is not necessarily possible to arrange \eqref{eq:ThickWittDomain} for the signature operator. Indeed a simple computation such as \cite[Lemma 5.6]{AlbLeiMazPia:SPWS} shows that \eqref{eq:ThickWittDomain} for the signature operator requires a ``thick Witt condition": 
\begin{equation*}
	\mathrm{IH}_{\overline m}^k(Z) =0 \text{ if } |\dim Z - 2k|\leq \tfrac12.
\end{equation*}
(This seems to have been overlooked in \cite{HarLesVer:DDLTOSS} and this hypothesis should be included in their Theorem 1.1.)
However it is possible to deform the signature operator to a Dirac-type operator satisfying \eqref{eq:ThickWittDomain}.}
Since this does not change the KK-class of the family of Dirac-type operators we do not lose any generality by assuming that this condition holds.\\

\noindent
{\bf Step 2}: We can identify $\cal E_{12} \otimes_{\cal C(X_2)} \cal E_{23}$ with $\cal E_{13}.$\\

Recall (e.g., from \cite[Chapter 4]{Lan:HC}) that the interior tensor product used here is defined by first endowing the algebraic tensor product 
$\cal E_{12} \otimes_{\mathrm{alg}} \cal E_{23}$ with the inner product given on simple tensors by
\begin{equation*}
	\langle f \otimes h, f' \otimes h' \rangle = \langle h, \langle f,f' \rangle h'\rangle,
\end{equation*}
modding out by $N = \{ z \in \cal E_{12} \otimes_{\mathrm{alg}} \cal E_{23}: \langle z, z\rangle = 0 \}$ and then completing with respect to this inner product.
In our setting this inner product is given on simple tensors with $f,f' \in \cal E_{12}$ and $h,h' \in \cal E_{23}$ by
\begin{multline*}
	X_3 \ni q \mapsto
	\langle f \otimes h, f' \otimes h' \rangle(q)
	= \int_{r\in \widehat p_{23}^{-1}(q)} g_{E_{23}}( h(r), \langle f,f' \rangle(r) h'(r) ) \; \mathrm{dvol}_{\mathrm{res}(F_{23,q})} \\
	= \int_{r\in \widehat p_{23}^{-1}(q)} g_{E_{23}}\Big( h(r), 
	\big( \int_{s\in \widehat p_{12}^{-1}(r)} g_{E_{12}}(f(s), f'(s))\; \mathrm{dvol}_{\mathrm{res}(F_{12,r})} \big)
 	 h'(r) \Big) \; \mathrm{dvol}_{\mathrm{res}(F_{23,q})}.
\end{multline*}
Notice, directly from the definition of the $\cal E_{ij},$ that we obtain the same completion if we start with the algebraic tensor product  
$\cal C_c(X_{1,\mathrm{reg}};E_{12}) \otimes_{\mathrm{alg}} \cal C_c(X_{2,\mathrm{reg}};E_{23})$ and that then we may rewrite the inner product
as an integral over the fibers of $\hat p_{13}: \wt X_1 \lra X_3,$
\begin{multline*}
	X_3 \ni q \mapsto
	\langle f \otimes h, f' \otimes h' \rangle(q) \\
	= \int_{\hat p_{13}^{-1}(q)} g_{E_{12}}(\alpha^*f, \alpha^*f') g_{E_{23}}(\beta^*h, \beta^*h') 
	\; \mathrm{dvol}_{F_{12,p_{23}^{-1}(q)}} \; \beta^*\mathrm{dvol}_{F_{23,q}} \\
	= \int_{\hat p_{13}^{-1}(q)} g_{\alpha^*E_{12}\otimes \beta^*E_{23}}(f \otimes \beta^*h, f'\otimes \beta^*h') 
	\; \mathrm{dvol}_{F_{13,q}},
\end{multline*}
where we have used that $\alpha|_{\wt X_{1,\mathrm{reg}}}=\mathrm{id}.$ Thus the completion with respect to this inner product produces 
$\cal C(X_3, L^2(X_{1,\mathrm{reg}}/X_3)) = \cal C(X_3, L^2(\wt X_1/X_3)) = \cal E_{13},$ as promised.\\

\noindent
{\bf Step 3}: We obtain a local expression relating the Dirac-type operators.\\

This is straightforward since we may, by localizing on $X_3$ and restricting to the regular parts, reduce to relating the the Dirac-type operator 
on the total space of the fiber bundle of smooth manifolds, 
\begin{equation*}
	F_{12,\mathrm{reg}} \fib F_{13,\mathrm{reg}} \lra F_{23,\mathrm{reg}},
\end{equation*}
with the family of Dirac-type operators on the fibers and a lift of the Dirac-type operator on the base and this is carried out in Proposition 10.12 and 
Theorem 10.19 in \cite{BerGetVer:HKDO} (see also \cite[Theorem 22]{KaaSui:RSFDO}, \cite[Theorem 30]{KaaSui:FDOAFSM}).
Thus we may write, as differential operators acting on smooth sections over the regular part of $X_1,$
\begin{equation*}
	\eth^{E_{13}}_{X_1/X_3}
	= \eth^{E_{12}}_{X_1/X_2} + \wt\eth^{E_{23}}_{X_2/X_3} -\tfrac14 \sum_{a<b}\sum_c g_{X_1/X_2}([f_a,f_b],e_c) \cl(f^a)\cl(f^b)\cl(e^c)
\end{equation*}
where $\wt\eth^{E_{23}}_{X_2/X_3}$ is a lift of $\eth^{E_{23}}_{X_2/X_3}$ to $X_1,$ and in the final sum $(f_a)$ runs over an orthonormal frame
for ${}^w TX_2/X_3,$ lifted to $X_1$ and $(e_c)$ runs over an orthonormal frame for ${}^wTX_1/X_2.$
Let us define
\begin{equation*}
	D^{E_{13}}_{X_1/X_3}
	= \eth^{E_{12}}_{X_1/X_2} + \wt\eth^{E_{23}}_{X_2/X_3}
\end{equation*}
and point out that $D^{E_{13}}_{X_1/X_3}$ is a relatively compact perturbation of $\eth^{E_{13}}_{X_1/X_3}$ when both are endowed with the VAPS
domain of the latter since from, e.g., \cite[\S1.2]{AlbGel:IFFDTOP}, the final sum in the expression for $\eth^{E_{13}}_{X_1/X_3}$ is a uniformly bounded
section of the endomorphism bundle of $E_{13}$ and, from \cite[Theorem 1]{AlbGel:IFFDTOP}, the VAPS domain includes compactly into $L^2.$
In particular, since the domain for $\eth^{\cal E_{13}}_{X_1/X_3}$ is pointwise equal to the VAPS domain of $\eth^{E_{13}}_{X_1/X_3}$ we can replace
this operator with $D^{E_{13}}_{X_1/X_3},$ with the same domain, and obtain the same class in $\mathrm{KK}^*(C_0(X_1), C_0(X_3)).$\\

\noindent
{\bf Step 4}: We use the characterization of the Kasparov product from \cite{Kuc:KUM}.\\

To prove the theorem it is sufficient, following {\em loc cit}, to show 
\begin{itemize}
\item [i)] (connection) for all $\xi$ in a dense subset of $\cal C(X_1)\cdot \cal C(X_2, L^2(X_1/X_2;E_{12})),$ the operator
\begin{equation*}
	\left[ \begin{pmatrix} 0 & T_\xi^* \\ T_{\xi} & 0 \end{pmatrix},
	 \begin{pmatrix} \eth^{\cal E_{23}}_{X_2/X_3} & 0 \\ 0 & D^{\cal E_{13}}_{X_1/X_3} \end{pmatrix} \right],
	 \text{ where }
	 T_{\xi}(e) = \xi\otimes e
\end{equation*}
is bounded on $\mathrm{dom}(\eth^{\cal E_{23}}_{X_2/X_3} \oplus D^{\cal E_{13}}_{X_1/X_3}),$
\item [ii)] (compatibility) there is a dense submodule $\cal W$ of $\cal E_{13}$ such that, for any $\mu_{13}, \mu_{12} \in \bb R\setminus\{0\},$
\begin{equation*}
	\eth^{\cal E_{12}}_{X_1/X_2}(i\mu_{13} + D^{\cal E_{13}}_{X_1/X_3})^{-1}(i\mu_{12}+\eth^{\cal E_{12}}_{X_1/X_2})^{-1}
\end{equation*}
is defined on $\cal W.$
\item [iii)] (positivity) for any $\mu_{13}, \mu_{12} \in \bb R\setminus\{0\},$ there exists  $\lambda >0$ such that
\begin{equation*}
	\langle (\eth^{\cal E_{12}}_{X_1/X_2} \otimes 1) \xi, D^{\cal E_{13}}_{X_1/X_3} \xi \rangle
	+ \langle D^{\cal E_{13}}_{X_1/X_3}\xi , (\eth^{\cal E_{12}}_{X_1/X_2} \otimes 1)\xi \rangle
	\geq -\lambda \langle \xi,\xi \rangle,
\end{equation*}
for all $\xi \in (i\mu_{13} + D^{\cal E_{13}}_{X_1/X_3})^{-1}(i\mu_{12}+\eth^{\cal E_{12}}_{X_1/X_2})^{-1}\cal W.$
\end{itemize}

To establish the connection condition (i), we choose to use $\cal C^{\infty}_c(X_{1,\mathrm{reg}};E_{12})$ as the dense subset of 
$\cal C(X_1)\cdot \cal C(X_2, L^2(X_1/X_2;E_{12}))$ and then, as we are working over the regular part, we may appeal to the computation done in 
\cite[Theorem 23]{KaaSui:FDOAFSM} to see that (i) comes down to the boundedness of the map
\begin{equation*}
	\cal E_{13} \ni r\mapsto (D_{X_1/X_2}\xi) \otimes r + (-1)^{|\xi|}i \sum_a \nabla_{f_a}^{X_1/X_3}\xi \otimes \cl(f_a)r
\end{equation*}
for each fixed $\xi,$ and this is manifest.

To establish condition (ii) it is sufficient, by \cite[Lemma 10]{Kuc:KUM} to establish that
\begin{equation*}
	\mathrm{dom}(D^{\cal E_{13}}_{X_1/X_3}) \subseteq \mathrm{dom}(\eth^{\cal E_{12}}_{X_1/X_2})
\end{equation*}
and, from the description of these domains in \eqref{eq:DefethCalE}, to establish this over each point $\zeta \in X_3.$
The domain of $D^{\cal E_{13}}_{X_1/X_3}|_{\zeta \in X_3}$ is defined by requiring Sobolev regularity with respect to all edge vector fields on $F_{13}$
and $\rho^{1/2}$ decay at all boundary hypersurfaces of $F_{13},$ whereas the domain of $\eth^{\cal E_{12}}_{X_1/X_2})|_{\zeta\in X_3}$ only imposes
these requirements with respect to the $p_{12}$-horizontal edge vector fields and boundary hypersurfaces, so this inclusion is clear.

Finally to establish (iii) we can appeal to the computation done in \cite[Lemma 17]{KaaSui:FDOAFSM} and the proof of 
\cite[Theorem 23]{KaaSui:FDOAFSM} to see that it suffices to show that there exists a constant $C>0$ such that for any $\xi$ in the domain of 
$\eth^{\cal E_{12}}_{X_1/X_2},$
\begin{equation*}
	\lVert \; | \nabla^{X_1/X_2} \xi|_{g_{X_1/X_2}} \rVert_{L^2}
	\leq C(\lVert \xi \rVert_{L^2} + \lVert \eth^{E_{12}}_{X_1/X_2} \xi \rVert_{L^2} ).
\end{equation*}
While the analogous inequality holds for edge Sobolev spaces (by, e.g.,  the bounded geometry of edge metrics) this inequality does not hold in general
for elliptic wedge differential operators. This is why in Step 1 we arranged for the VAPS domain of each operator in the family $\eth^{E_{12}}_{X_1/X_2}$ is $\rho H^1_e(\hat p_{12}^{-1}(\zeta);E)$ as this inequality does hold for elements in this space.

The only difference when either $\dim X_1/X_2$ or $\dim X_2/X_3$ is odd is the treatment of the gradings. This aspect can be handled just as when the underlying spaces are smooth and the fiber bundles are trivial. We refer to \cite[Examples 2.38-2.40]{BraMesSui:GTSTUKP} and especially \cite{Wah:PFAICHS} for a thorough discussion.
\end{proof}

\begin{cor} \label{cor:FunctorialityFibrations}
Given fiber bundles of smoothly stratified spaces as in \eqref{eq:3fibrations}, if the fibers of $p_{12}$ and $p_{23}$ are smooth oriented Witt pseudomanifolds then so are the fibers of $p_{13}$ and the associated analytic signature classes satisfy 
\begin{multline*}
	[D^{\mathrm{sign}}_{X_1/X_3}] = \ell\big( [D^{\mathrm{sign}}_{X_1/X_2}]\otimes [D^{\mathrm{sign}}_{X_2/X_3}] \big)
	\text{ in } \mathrm{KK}^*(C_0(X_1), C_0(X_3)), \\
	\text{ with } \ell = \begin{cases} 2 & \text{ if } \dim X_1/X_2 \text{ and } \dim X_2/X_3 \text{ are odd }\\ 1 & \text{ otherwise } \end{cases}
\end{multline*}
Correspondingly, we have
\begin{equation*}
	\Sigma(p_{13})= \Sigma(p_{12}) \otimes \Sigma(p_{23}) \text{ in } \mathrm{KK}^*(C_0(X_1), C_0(X_3))[\tfrac12].
\end{equation*}
\end{cor}

\begin{proof}
The fact that the fibers of $p_{13}$ satisfy the topological Witt condition reduces to the fact that the product of Witt spaces is a Witt space, which is proved, e.g., in Proposition 9.1.25 of \cite{Fri:SIH}.

The theorem establishes that the product of $[D^{\mathrm{sign}}_{X_1/X_2}]$ and $[D^{\mathrm{sign}}_{X_2/X_3}]$ is the KK-class of a family of Dirac-type operators associated to $p_{13},$ and the construction above specializes to show that this operator is $d+\delta$ acting on the vertical wedge differential forms, so it is only necessary to work out the grading. That this results in the formula in the statement of the corollary is explained in \cite[Lemma 6]{RosWei:SO} (see also \cite{Wah:PFAICHS} and \cite[Proposition 2.29]{Ebe:TCHS}).

For the analytic transfer classes, if we define
\begin{equation*}
	i = \dim X_1-\dim X_2, \quad j = \dim X_2 - \dim X_3
\end{equation*}
then $\Sigma(p_{13})$ is equal to 
$2^{-\lfloor (i+j)/2 \rfloor} [D^{\mathrm{sign}}_{X_1/X_3}]$
so it suffices to note that
\begin{equation*}
	\ell \big( 2^{-\lfloor (i+j)/2 \rfloor} \big)
	=  2^{-\lfloor i/2 \rfloor} 2^{-\lfloor j/2 \rfloor}, 
	\text{ with } \ell = \begin{cases} 2 & \text{ if }i \text{ and } j \text{ are odd }\\ 1 & \text{ otherwise } \end{cases}
\end{equation*}
\end{proof}

It is instructive to consider the case of Theorem 
\ref{thm.bundletransferpreservesksignature} where $p$ is the factor projection
of a product bundle, which  generalizes
\cite[Lemma 6, p. 51]{RosWei:SO} from the manifold case to
singular spaces.

\begin{cor} \label{cor.signoponproductofwitt}
Let $X^n$ and $Y^m$ be closed smoothly stratified Witt spaces.
Then the product $X\times Y$ is a closed smoothly stratified Witt space, and 
the $\K$-homology class of the signature operator on $X\times Y$ satisfies
\begin{equation*}
	[D^{\mathrm{sign}}_{X \times Y}] = \ell\big( [D^{\mathrm{sign}}_{X}]\boxtimes [D^{\mathrm{sign}}_{Y}] \big)
	\text{ in } \mathrm{K}_{n+m}^{\mathrm{an}}(X \times Y), \\
	\text{ with } \ell = \begin{cases} 2 & \text{ if } \dim X \text{ and } \dim Y \text{ are odd }\\ 1 & \text{ otherwise } \end{cases}
\end{equation*}
where $\boxtimes$ denotes the external Kasparov product
$\K^\an_n (X) \otimes \K^\an_m (Y) \to \K^\an_{n+m} (X\times Y)$.
Correspondingly we have
\begin{equation*}
	\mathrm{sign}_K(X\times Y) = 
	\mathrm{sign}_K(X) \boxtimes \mathrm{sign}_K(Y)
	\text{ in } \mathrm{K}_{n+m}^{\mathrm{an}}(X \times Y)[\tfrac12].
\end{equation*}
\end{cor}
\begin{proof}
Apply the theorem to the case in which $X_1 = X \times Y,$ $X_2 = Y,$ $X_3 = \mathrm{pt}$ and $p_{12}$ is the natural projection $X\times Y \lra Y.$
\end{proof}

Another interesting case concerns spin Dirac operators. If the vertical Clifford structures on $p_{12}$ and $p_{23}$ correspond to spin structure on their fibers, then the induced vertical Clifford structure on $p_{13}$ does as well and we have the following result.

\begin{cor}
Given fiber bundles of smoothly stratified spaces as in \eqref{eq:3fibrations}, whose vertical wedge tangent bundles are spin, equipped with vertical wedge metrics such that the associated vertical families of spin Dirac operators satisfy the analytic Witt condition, the associated KK-classes satisfy
\begin{equation*}
	[D^{\mathrm{spin}}_{X_1/X_3}] = [D^{\mathrm{spin}}_{X_1/X_2}] \otimes [D^{\mathrm{spin}}_{X_2/X_3}] 
		\in \mathrm{KK}^{\dim X_1/X_3}(C_0(X_1), C_0(X_3)).
\end{equation*}
\end{cor}

\begin{remark}
The corollary applies, for example, if the fibers of $p_{12},$ $p_{23},$ and $p_{13}$ all satisfy the psc-Witt condition (described after Corollary \ref{cor:DefSpinKK}). 

Notice that, unlike the topological Witt condition, assuming that the fibers of $p_{12}$ and $p_{23}$ satisfy this condition does not imply that the fibers of $p_{13}$ will satisfy the psc-Witt condition. 
Indeed, it is not difficult to see that the product of two cones satisfying the psc-Witt condition need not satisfy the psc-Witt condition.
\end{remark}

\subsection{Base change}\label{subsect:base-change}
In this subsection we show that the Gysin homomorphism associated to a fiber bundle is compatible with pull-back along proper smooth maps of stratified spaces.

Specifically suppose that $p: X \lra Y$  is a smooth oriented fiber bundle of stratified spaces, whose typical fiber $Z$ is a Witt pseudomanifold, $W$ is another smoothly stratified space and $f: W \lra Y$ is a proper smooth stratified map, and then consider the Cartesian diagram
\begin{equation*}
	\xymatrix{
	f^* X \ar[r]^-g \ar[d]_-q \pullbackcorner & X \ar[d]^-p \\
	W \ar[r]^-{f} & Y }
\end{equation*}
with $f^*X\xlra q W$ the pull-back bundle.
We assume that $X$ is endowed with a vertical totally geodesic wedge metric and then we endow $f^*X$ with the pull-back metric along $g.$
We have two bivariant classes 
$$[D_{X/Y}^{\mathrm{sign}}]\in \mathrm{KK}(C_0(X),C_0(Y)),\,\quad [D_{f^*X/W}^{\mathrm{sign}}]\in \mathrm{KK}(C_0(f^* X), C_0(W)),$$
where $D_{X/Y}^{\mathrm{sign}}$ denotes the vertical family of signature operators endowed with their VAPS domain, and similarly for $D_{f^*X/W}^{\mathrm{sign}},$
and by Kasparov multiplication they define two homomorphisms
\begin{equation*}
	[D_{X/Y}^{\mathrm{sign}}] \otimes \cdot: \mathrm K_*^{\mathrm{an}}(Y) \lra \mathrm K_*^{\mathrm{an}}(X), \quad
	[D_{f^*X/W}^{\mathrm{sign}}] \otimes \cdot: \mathrm K_*^{\mathrm{an}}(W) \lra \mathrm K_*^{\mathrm{an}}(f^*X).
\end{equation*}
We will deduce base change from the following:
\begin{thm} \label{thm.basechangeforstratfiberbundles}
If $f,$ $g,$ $p,$ and $q$ are above then the following holds
\begin{equation*}
	[D_{X/Y}^{\mathrm{sign}}] \otimes f_*\alpha = g_*([D_{f^*X/W}^{\mathrm{sign}}] \otimes \alpha), \quad \forall\alpha\in \mathrm K_*^{\mathrm{an}}(W).
\end{equation*}
\end{thm}

\begin{proof}
Since  $f$ is proper, it defines $f^*: C_0(Y)\to C_0(W)$, a $C^*$-algebra homomorphism, and 
$$f_* (\alpha)= [f^*]\otimes \alpha,\quad\text{with}\quad [f^*]\in \mathrm{KK}(C_0(Y),C_0(W))$$
and, since $g$ is also proper, similarly for $g_*$, with 
$[g^*]\in \mathrm{KK}(C_0(X), C_0(f^* X))$. Here we are using the well known fact that if $A$ and $B$ are two $C^*$-algebras
and $\phi:A\to B$ is a homomorphism then there is a natural KK-class
$$[\phi]:= [B,\phi:A \to B,0]\in \mathrm{KK}(A,B)$$
where we regard $B$ as a $B$-Hilbert module with the $B$-valued inner product equal to $\langle b, b^\prime \rangle :=
(b^\prime)^* b$. This class $[\phi]\in \mathrm{KK}(A,B)$ has the property that 
$\phi^*: \mathrm{KK}(B,D)\to \mathrm{KK}(A,D)$ is given by $[\phi]\otimes (-) : \mathrm{KK}(B,D)\to \mathrm{KK}(A,D)$ for any $C^*$-algebra $D$.\\
Hence, by associativity of the Kasparov product, we see that the theorem is equivalent to the
equality
$$[D_{X/Y}^{\mathrm{sign}}]\otimes [f^*]= [g^*]\otimes [D_{f^*X/W}^{\mathrm{sign}}]\quad\text{in}\quad \mathrm{KK}(C_0(X),C_0(W))\,.$$

\noindent
In general if 
$${\bf x}=[E_1,\phi_1: A\to \mathbb{B} (E_1),F_1]\in \mathrm{KK}(A,D)$$ and 
if $\lambda:D\to B$ is a $C^*$-homomorphism and 
$${\bf y}=[B,\lambda:D\to B,0]\in \mathrm{KK}(D,B)$$ then  
$${\bf x}\otimes {\bf y}=[E= E_1\otimes_\lambda B, \phi_1\otimes_\lambda {\rm Id}, F\otimes {\rm Id}]\in \mathrm{KK}(A,B).$$
One should observe that $\lambda$ induces a group homomorphism:
$$\lambda_* : \mathrm{KK}(-,D)\to \mathrm{KK}(-,B)$$
and 
$${\bf x}\otimes {\bf y}= \lambda_* {\bf x}.$$
Similarly if 
$${\bf x}=[D, \mu: A\to D,0]\in \mathrm{KK}(A,D)$$
and
$${\bf y}=[E_2,\phi_2: D\to \mathbb{B} (E_2),F_2]\in \mathrm{KK}(D,B)$$ 
then 
$${\bf x}\otimes {\bf y}=[E= D\otimes_{\phi_2}  E_2\stackrel{\Phi}{\simeq} E_2, f\otimes_{\phi_2} {\rm Id}, \Phi^* F_2 \Phi]\in \mathrm{KK}(A,B).$$
Notice that $\mu$ induces
$$\mu^*: \mathrm{KK}(D,-) \to \mathrm{KK}(A,-)$$
and it holds that
$$\mu^* {\bf y}= {\bf x}\otimes {\bf y}\,.$$
We go back to our specific situation and write down the relevant cycles.\\
Let $E$ be the bundle of exterior powers of wedge vertical differential forms on $X,$
\begin{equation*}
	E = \Lambda^* ({}^wT^*X/Y).
\end{equation*}
We denote by $L^2 (X/Y;E)$ the bundle of Hilbert spaces over $Y$ with fiber at $y\in Y$ given by
$L^2 (X_y, E_y)$ with $X_y=p^{-1}(y)$ and $E_y$ the restriction of $E$ to $X_y$.
Then
\begin{equation*}
	[D_{X/Y}^{\mathrm{sign}}]:= [H_{X/Y}:=C_0(Y, L^2 (X/Y;E)),\phi: C_0(X)\to \mathbb{B}(H_{X/Y}), D_{X/Y}^{\mathrm{sign}} ]\\
	\in \mathrm{KK}(C_0(X),C_0(Y))
\end{equation*}
where $C_0(Y, L^2 (X/Y;E))$ denotes continuous sections vanishing at infinity of the bundle of Hilbert spaces 
$L^2 (X/Y;E).$
We also have $[f^*]\in \mathrm{KK}(C_0(Y), C_0(W))$ given by
$$[C_0(W), f^*:C_0(Y)\to C_0(W),0]\in \mathrm{KK}(C_0(Y), C_0(W)).$$
Next we write the elements appearing on the right hand side of the base change formula.
These are
$$[g^*]=[C_0(f^* X), g^*: C_0(X)\to C_0(f^* X),0]\in \mathrm{KK}(C_0(X), C_0(f^* X)),$$
whereas $[D_{f^*X/W}^{\mathrm{sign}}]$ is given by
\begin{equation*}
	[ H_{f^*X/W}:=C_0(W,L^2 (f^*X/W, g^*E)), \tilde{\phi}:  C_0(f^*X)\to 
	\mathbb{B}(H_{f^*X/W}), D_{f^*X/W}^{\mathrm{sign}}] \\
	\in \mathrm{KK}(C_0(f^* X), C_0(W)).
\end{equation*}
\begin{lemma}
There is an isomorphism 
$$L^2 (f^*X/W, g^*E)= f^* L^2 (X/Y;E)$$
of bundles of Hilbert spaces over $W$.
\end{lemma}
\begin{proof}
By our definition of fiber bundle of smoothly stratified spaces, we know that there is an open cover of $Y,$ $\{ U_{i}\}_{i \in \cal A},$ together with corresponding local trivializations of $p: X \to Y$ of the form
\begin{equation*}
	\xymatrix{
	p^{-1}( U_{i}) \ar[rr]^-{\varphi_{i}} \ar[rd]_-p & & Z \times  U_{i} \ar[ld] \\
	&  U_{i} & }
\end{equation*}
where the right arrow is the projection onto the right factor and where $\varphi_{i}$ is a stratified diffeomorphism.
Whenever $i, j \in \cal A$ are such that $ U_{ij} :=  U_{i} \cap \cal U_{j} \neq \emptyset$ the map
\begin{equation*}
	Z \times  U_{ij} \xlra{\varphi_{j}\circ \varphi_{i}^{-1}} Z \times  U_{ij}
\end{equation*}
necessarily has the form $(z,u) \mapsto (\rho_{ij}(u)(z), u)$ with transition functions $\rho_{ij},$ a family of stratified diffeomorphisms of $Z$ parametrized by points of $ U_{ij}.$ (This family is smooth in the sense that they fit together into the smooth map $\varphi_{j}\circ \varphi_{i}^{-1}.$ ) 
If $i, j, k \in \cal A$ are such that $\cal U_{i} \cap \cal U_{j} \cap \cal U_{k} \neq \emptyset$ then, on this set, the transition functions satisfy the cocycle condition $\rho_{ij} = \rho_{kj} \circ \rho_{ij}.$

At a point $u \in U_{ij},$ the  stratified diffeomorphism
$\rho_{ij} (u): Z \to Z$ induces an isomorphism of Hilbert spaces
$\rho^*_{ij} (u): L^2 (Z) \to L^2 (Z)$ by pulling back an $L^2$-function on $Z$
via $\rho_{ij}(u)$, i.e.
$(\rho^*_{ij} (u))(h) = h \circ (\rho_{ij} (u))$
for $h\in L^2 (Z)$.
This yields a system 
$\{ \rho^*_{ij}: U_{ij} \to \mathrm{GL}(L^2 (Z)) \}$ of transition functions satisfying
the cocycle condition. This system is the system of transition functions
associated to the bundle $L^2 (X/Y)$ of Hilbert spaces.

The pulled-back smooth stratified fiber bundle $q: f^* X \to W$ is trivialized over the open cover $\{V_i = f^{-1}(U_i)\}_{i \in \cal A}$ of $W,$ with local trivializations 
\begin{equation*}
	\xymatrix{
	q^{-1}(V_{i}) \ar[rr]^-{\psi_i:=(\pi_L\circ\varphi_{i}\circ g, q) } \ar[rd]_-p & & Z \times  V_{i} \ar[ld] \\
	&  V_{i} & }
\end{equation*}
where $\pi_L: Z \times U_i \lra Z$ is the projection onto the left factor. 
Whenever $i, j \in \cal A$ are such that $V_{ij} :=  V_{i} \cap V_{j} \neq \emptyset$ the map
\begin{equation*}
	Z \times  V_{ij} \xlra{\psi_{j}\circ \psi_{i}^{-1}} Z \times  V_{ij}
\end{equation*}
is given by $(z,u) \mapsto (\rho_{ij}(f(u))(z), u),$ i.e., the transition functions of $q$ are $\tau_{ij} = \rho_{ij} \circ f.$
At a point $v \in V_{ij},$ the stratified diffeomorphism
$\tau_{ij} (v): Z \to Z$ induces an isomorphism of Hilbert spaces
$\tau^*_{ij} (v): L^2 (Z) \to L^2 (Z)$ by pulling back an $L^2$-function on $Z$
via $\tau_{ij}(v)$, i.e.
$(\tau^*_{ij} (v)) (h) = h \circ (\tau_{ij} (v))$
for $h\in L^2 (Z)$.
This yields a system 
$\{ \tau^*_{ij}: V_{ij} \to \mathrm{GL}(L^2 (Z)) \}$ of transition functions satisfying
the cocycle condition which is the system of transition functions
for the Hilbert space bundle $L^2 (f^* X/W)$.
On the other hand,
the transition functions of the pullback Hilbert space bundle
$f^* L^2 (X/Y)$ over $W$ are given by composing the transition
functions of $L^2 (X/Y)$ with $f$, that is,
$f^* L^2 (X/Y)$ has transition functions
$\sigma_{ij}: V_{ij} \to \mathrm{GL} (L^2 (Z))$,
$\sigma_{ij} (v) = \rho^*_{ij} \circ f$.
Now, on a function $h\in L^2 (Z),$
\[ \sigma_{ij} (v) (h)
   = \rho^*_{ij} (f(v)) (h)
   = h \circ (\rho_{ij} (f(v)))
   = h \circ (\tau_{ij} (v))
   = (\tau^*_{ij} (v)) (h). \]
This shows that $f^* L^2 (X/Y)$ and $L^2 (f^* X/W)$ have the same
transition functions with respect to $\{ V_{ij} \}$.
\end{proof}

Consequently,
$$[D_{f^*X/W}^{\mathrm{sign}}]= [ H_{f^*X/W}:=C_0(W,f^* L^2 (X/Y;E)), \tilde{\phi}:  C_0(f^*X)\to 
 \mathbb{B}(H_{f^*X/W}), D_{f^*X/W}^{\mathrm{sign}}]$$
Notice that the left hand side of base change, $[D_{X/Y}^{\mathrm{sign}}]\otimes [f^*]$,  is also equal to 
\begin{equation*}
	(f^*)_* [H_{X/Y}:=C_0(Y, L^2 (X/Y;E)),\phi: C_0(X)\to \mathbb{B}(H_{X/Y}), D_{X/Y}^{\mathrm{sign}} ]\\
	\in \mathrm{KK}(C_0(X),C_0(W))
\end{equation*}
The definition we have given thus gives for
$[D_{X/Y}^{\mathrm{sign}}]\otimes [f^*]$ the following element
$$
 [H_{X/Y}\otimes_{f^*} C_0(W),\phi\otimes_{f^*} {\rm Id}: C_0(X)\to \mathbb{B}(H_{X/Y}\otimes_{f^*} C_0(W)), D_{X/Y}^{\mathrm{sign}}\otimes {\rm Id} ]$$
The right hand side, 
on the other hand, is equal to
\begin{equation*}
	(g^*)^*  [ H_{f^*X/W}:=C_0(W,L^2 (f^*X/W, g^*E)), \tilde{\phi}:  C_0(f^*X)\to 
	 \mathbb{B}(H_{f^*X/W}), D_{f^*X/W}^{\mathrm{sign}}]\\
	 \in \mathrm{KK}(C_0(X),C_0(W))
\end{equation*}
which is in turn equal to
 $$  [ H_{f^*X/W}:=C_0(W,L^2 (f^*X/W, g^*E)),  \tilde{\phi}\circ (g^*):  C_0(X)\to 
 \mathbb{B}(H_{f^*X/W}), D_{f^*X/W}^{\mathrm{sign}}]$$
 that we can write, thanks to the lemma, as
 $$  [ C_0(W,f^* L^2 (X/Y, E)),  \tilde{\phi}\circ (g^*):  C_0(X)\to 
 \mathbb{B}(C_0(W,f^* L^2 (X/Y, E))), D_{f^*X/W}^{\mathrm{sign}}]$$
and this can be written as 
$$  [ C_0(W,f^* L^2 (X/Y, E)),  \tilde{\phi}\circ (g^*):  C_0(X)\to 
 \mathbb{B}(C_0(W,f^* L^2 (X/Y, E))), f^* D_{X/Y}^{\mathrm{sign}}]
$$
because the vertical wedge metric on $f^*X$ is the pull-back of the corresponding metric on $X.$

This finishes the proof of Theorem \ref{thm.basechangeforstratfiberbundles}. 
\end{proof}

\begin{remark}\label{rmk: on-base-change}
It is convenient to note, as this will show up when we use the geometric description of K-homology below, that 
Theorem  \ref{thm.basechangeforstratfiberbundles} holds also when $f$ is a continuous map from a smooth manifold into a smoothly stratified space.
In this case we define $[D_{f^*X/W}^{\mathrm{sign}}]$ by 
$$  [ C_0(W,f^* L^2 (X/Y, E)),  \tilde{\phi}:  C_0(f^*X)\to 
 \mathbb{B}(C_0(W,f^* L^2 (X/Y, E))), f^* D_{X/Y}^{\mathrm{sign}}]
$$
(so that only the continuity of $f$ is required)
and the proof above shows that for any $\alpha\in \mathrm{KK}(C_0(W),\cplx),$
$$[D_{X/Y}^{\mathrm{sign}}]\otimes (f_* (\alpha))=g_* ([D_{f^*X/W}^{\mathrm{sign}}]\otimes \alpha).$$
\end{remark}

An immediate corollary, obtained by passing to analytic transfer classes, is base change.
\begin{cor}
If $p: X \lra Y$  is a smooth oriented fiber bundle of stratified spaces, whose typical fiber $Z$ is a Witt pseudomanifold, $W$ is another smoothly stratified space and $f: W \lra Y$ is a proper smooth stratified map and $q: f^*X \lra W$ is the induced fiber bundle over $W,$ then
\begin{equation*}
	\Sigma(p) \otimes f_*\alpha = g_*( \Sigma(q) \otimes \alpha), \quad \forall\alpha\in \mathrm K_*^{\mathrm{an}}(W)[\tfrac12].
\end{equation*}
Equivalently, 
\begin{equation*}
	p^! (f_* (\alpha))= g_* (q^! (\alpha))\quad \forall\alpha\in \mathrm K_*^{\mathrm{an}}(W)[\tfrac12].
\end{equation*}
\end{cor}

\section{Analytic transfer of the signature orientation along a normally non-singular inclusion} \label{sec:Inclusions}

\begin{defn}
An inclusion of smoothly stratified spaces $j: X \lra Y$ is said to be {\bf normally non-singular} (nns) if there exist
\begin{itemize}
\item an open neighborhood
$\cal U \subseteq Y$ of $j(X)$,
\item an oriented vector bundle $\pi: N_{X}Y\to X,$ and
\item a stratified diffeomorphism $\psi: N_{X}Y\to \cal U$ such that
 \[ \xymatrix{
 N_X Y \ar[r]^-\psi_-\simeq & \cal U \\
 X \ar@{^{(}->}[u] \ar[r]_-j^-\simeq & j(X) \ar@{^{(}->}[u]
 }
 \]
commutes, where the left hand vertical map is the zero section inclusion.
\end{itemize}
Here $\cal U$ has the stratification inherited from $Y$ and $N_{X}Y$ has the stratification pulled-back
from $X$ through $\pi$.
\end{defn}

\noindent
We shall denote by $\phi: \cal U \to N_{X}Y$ the inverse of $\psi$.

Let us start by recalling the situation in the smooth setting.
If $j:M \lra M'$ is an embedding of smooth manifolds then, without any orientation assumptions, we will construct a `wrong-way map' between the K-theory of their tangent bundles,
\begin{equation*}
	\mathrm K^0(TM) \lra \mathrm K^0(TM').
\end{equation*}
If $M$ and $M'$ are K-oriented then this induces the Gysin map $j_!: \mathrm K^0(M) \lra \mathrm K^0(M')$ (in fact it suffices for $j$ to be K-oriented).

To define the former, we start by noting that 
any embedding of smooth manifolds is normally non-singular since the image $j(M)$ has a tubular neighborhood $\cal U,$
\begin{equation*}
	j(M) \subseteq \cal U \subseteq M',
\end{equation*}
which is diffeomorphic to the total space of the normal bundle $N=N_MM'$ of $j(M)$ in $M'$ through a diffeomorphism
$\varphi: \cal U \to N$.
The normal bundle of $Tj(M)$ in $TM'$ is $TN$ and this equals $N \oplus N$ lifted to $Tj(M).$ Since $N \oplus N$ can be identified with $N \otimes \bb C$
we see that $TN \lra Tj(M)$ is a complex vector bundle so there is a KK-class corresponding to the fiberwise Dolbeault operators
\begin{equation*}
	[\overline\partial_{TN/Tj(M)}] \in \mathrm{KK}(C_0(TN), C_0(Tj(M)))
\end{equation*}
and multiplication by this class implements the Thom isomorphism so there is also a class
\begin{equation*}
	[\overline\partial_{TN/Tj(M)}]^{-1} \in \mathrm{KK}(C_0(Tj(M)), C_0(TN))
\end{equation*}
We can identify $T\cal U$ with $TN$ through $d\varphi: T\cal U\to TN$. This induces
by pull-back an algebra isomorphism $\Phi: C_0(TN)\to C_0(T\cal U)$.
As $T\cal U$ is an open subset of $TM',$ the natural extension by zero  defines a homomorphism of $C^*$-algebras $C_0(T\cal U)  \lra C_0(TM')$.
Precomposing with $\Phi$ we finally obtain a $C^*$-algebra homomorphism  $j': C_0(TN)  \lra C_0(TM')$.  This defines a KK-class,
\begin{equation*}
	[j'] \in \mathrm{KK}(C_0(TN), C_0(TM')).
\end{equation*}
Putting these together we obtain 
\begin{equation*}
	\xymatrix{
	\mathrm K^0(TM)= \mathrm{KK}(\bb C, C_0(TM)) \ar[rd]_-{\otimes [\overline\partial_{TN/Tj(M)}]^{-1}\phantom{xxxx}} \ar[rr] 
	& &\mathrm{KK}(\bb C, C_0(TM'))= \mathrm K^0(TM')	\\
	& \mathrm{KK}(\bb C, C_0(TN)) \ar[ru]_-{\otimes [j'] } & } 
\end{equation*}

\begin{remark}
This construction produces the topological index of Atiyah-Singer in the case where $M'=\bb R^n.$
In that case $\mathrm K^0(TM')= \mathrm K^0(\bb R^{2n})= \mathrm K^0(\bb C^n) = \bb Z$ and so we have produced a map
$\mathrm K^0(TM)\lra \bb Z.$
\end{remark}

If $M$ is K-oriented, i.e., if there is a spin-c structure on its tangent bundle, then the fiberwise spin-c Dirac operators $\eth_{TM/M}$ determine a KK-class 
\begin{equation*}
	[\eth_{TM/M}] \in \mathrm{KK}(C_0(TM), C(M))
\end{equation*}
and multiplication by this class implements the Thom isomorphism
\begin{equation*}
	\mathrm K^0(TM) = \mathrm{KK}(\bb C, C_0(TM)) \lra \mathrm{KK}(\bb C, C(M)) = \mathrm K^0(M).
\end{equation*}
Thus if both $M$ and $M'$ are K-oriented we obtain the Gysin map
\begin{equation*}
	j_!: \mathrm K^0(M) \lra \mathrm K^0(M')
\end{equation*}
as anticipated.

For our purposes it is useful to note that the Gysin map can be defined without reference to the tangent bundles assuming only that the inclusion $j:M \lra M'$ is K-oriented or, equivalently, that the normal bundle $N_MM'$ has a spin-c structure. In this case there is a 
Thom isomorphism 
\begin{equation*}
	\mathrm K^0(N_MM') = \mathrm{KK}(\bb C, C_0(N_MM')) \lra \mathrm{KK}(\bb C, C_0(M)) = \mathrm K^0(M)
\end{equation*}
resulting from the Kasparov product with the class of the fiberwise spin-c Dirac operators (see \cite[\S 5, Theorem 8]{Kas:OKECA} \cite[\S 19.9.4]{Bla:KOA}).
Multiplying the inverse of this isomorphism with the class induced by the $C^*$-homomorphism  $i: C_0(N_MM') \lra C_0(M')$ (obtained by identifying $N_MM'$ with $\cal U$ as above) produces the Gysin map
\begin{equation*}
	j_!: \mathrm K^0(M) \lra \mathrm K^0(M')
\end{equation*}
and exhibits it as Kasparov product with
\begin{equation*}
	j! := [\eth_{N_MM'/M}]^{-1} \otimes [i]  \in \mathrm{KK}(C_0(M), C_0(M')).
\end{equation*}
We will refer to it as the {\bf spin-c Gysin map}, to emphasize its origin and the assumption that
$j$ is $K$-oriented (that is, the normal bundle has a  spin-c structure). We shall give a similar
construction in K-homology using the signature operator and we shall see that such a construction also holds in the context of normally non-singular inclusions of
Witt spaces.

\bigskip
\noindent
Let $j: X \lra Y$ be a nns inclusion of smoothly stratified spaces and recall  that, by definition,  $N_YX$ is oriented. 
We have denoted the vector bundle projection by  $\pi:N_YX\lra X$.  Then as a special case of the theory developed in the previous sections, the family of signature operators on the fibers of $N_XY$ determines an analytic transfer class,
\begin{equation*}
	\Sigma(\pi) = 2^{-\lfloor \ell/2 \rfloor}
	[D_{N_YX/X}^{\mathrm{sign}}] \in \mathrm{KK}^{\ell}(C_0(N_YX), C_0 (X))[\tfrac12]\,,\;\;\;\ell=\dim Y-\dim X.
\end{equation*}

\begin{lemma} \label{lem.hilsumpsigmaisinvertible} (Hilsum, see \cite[Theoreme 3.10]{Hil:FKBPVL})\\
There exist a unique  element $\Sigma_{N_YX}\in \mathrm{KK}^{\ell}(C_0 (X),C_0(N_Y X))[\tfrac12]$ such that 
\[ \Sigma_{N_YX} \otimes \Sigma(\pi) = 1 \in \KK^0 (C_0 (X), C_0 (X))[\tfrac12], \]
\end{lemma}
\begin{proof}
(Throughout the proof we abbreviate  $\Sigma_{N_YX}$ as $ \Sigma_{N}.$)
We follow closely Hilsum but use our Theorem \ref{thm:FunctorialityFibrations} at a crucial point. 
Consider a vector bundle $\nu\to X$ such that $N_X Y\oplus \nu$ 
is the trivial bundle $X\times \bb R^m$ (up to bundle isomorphisms), with $m$ even. We obtain a fibration 
$p: X\times \bb R^m\to N_X Y$ and thus a class $[D_{(X\times \bb R^m)/N_XY}^{\mathrm{sign}}] \in \mathrm{KK}^{m-\ell}(C_0(X\times \bb R^m), C_0(N_X Y)) $.
Recall that Kasparov in \cite[Thereom 7.5]{Kas:OKECA} constructed elements $\alpha_m\in \mathrm{KK}^{m}(C_0 (\bb R^m), \bb C)$ and $\beta_m \in  \mathrm{KK}^{m}(\bb C,C_0 (\bb R^m))$ 
such that $\beta_m \otimes \alpha_m={\rm Id}_{\bb C}\in  \mathrm{KK}^{0}
(\bb C,\bb C)$. Moreover, as pointed out by Hilsum, 
$[D_{\bb R^m}^{\mathrm{sign}}] = 2^m \alpha_m$.
Consider  $\beta_m\in \mathrm{KK}^m(\bb C, C_0 (\bb R^m))$ and denote by 
$\tau_X (\beta_m)$ 
the element in the group  $\mathrm{KK}^m (C_0 (X), C_0 (X)\otimes C_0 (\bb R^m))\equiv 
\mathrm{KK}^m (C_0 (X), C_0(X\times \bb R^m))$ obtained by external Kasparov product 
of $\beta_m$ with the KK element corresponding 
to the identity homomorphism on $C_0 (X)$:
$$ \tau_X (\beta_m): = \beta_m\otimes {\rm Id}_{C_0 (X)}\in \mathrm{KK}^m (C_0 (X), C_0 (X)\otimes C_0 (\bb R^m)).$$
Then 
$$\tau_X (\beta_m)\otimes [D_{(X\times \bb R^m)/N_XY}^{\mathrm{sign}}] \in \mathrm{KK}^{2m-\ell}(C_0 (X), C_0(N_X Y))$$
Let us prove that 
$$(\tau_X (\beta_m)\otimes [D_{(X\times \bb R^m)/N_XY}^{\mathrm{sign}}])\otimes [D_{N_YX/X}^{\mathrm{sign}}] =2^m \operatorname{Id}_{C_0 (X)} \in \mathrm{KK}^0 (C_0 (X), C_0 (X))\,.$$
We consider the projection of the trivial bundle $X\times \bb R^m\to X$ and we denote it by $q$. Then, clearly,
$q=\pi\circ p$ and thus 
\begin{equation}\label{sigmaq}
	[D_{(X\times \bb R^m)/X}^{\mathrm{sign}}] = [D_{(X\times \bb R^m)/N_XY}^{\mathrm{sign}}] \otimes [D_{N_YX/X}^{\mathrm{sign}}] 
\end{equation} by our Theorem  \ref{thm:FunctorialityFibrations}.
But it is also true that $[D_{(X\times \bb R^m)/X}^{\mathrm{sign}}] = \tau_X (\alpha_m)$ and so 
$$\tau_X(\beta_m)\otimes[D_{(X\times \bb R^m)/X}^{\mathrm{sign}}] =2^m {\rm Id}_{C_0 (X)}\,.$$
Thus multiplying \eqref{sigmaq} on the left by $\tau_X(\beta_m)$ 
and using associativity we obtain 
$$2^m {\rm Id}_{C_0 (X)}= (\tau_X (\beta_m)\otimes [D_{(X\times \bb R^m)/N_XY}^{\mathrm{sign}}])\otimes [D_{N_YX/X}^{\mathrm{sign}}]\;\;\text{in}\;\; \mathrm{KK}^0 (C_0 (X), C_0 (X))\,.$$
which is what we wanted to establish.

If we rewrite this in terms of the analytic transfer forms, using the notation $j=\dim X,$ $k=\dim Y$ (so $\ell=k-j$),
it becomes
\begin{equation*}
	2^m {\rm Id}_{C_0 (X)}= 
	(\tau_X (\beta_m)\otimes 2^{\lfloor (m-k)/2 \rfloor}\Sigma(\pi))\otimes 2^{\lfloor (k-j)/2\rfloor}\Sigma(p)\;\;
	\text{in}\;\; \mathrm{KK}^0 (C_0 (X), C_0 (X))[\tfrac12]\,.
\end{equation*}
We set 
$$\Sigma_N:= 2^{-m + \lfloor (m-k)/2 \rfloor + \lfloor (k-j)/2\rfloor} \tau_X (\beta_m)\otimes\Sigma(\pi)$$
and we notice that this is indeed an element in $\mathrm{KK}^{\ell}(C_0 (X),C_0(N_Y X))[\tfrac12]$ and by construction it satisfies
the equation  $\Sigma_N \otimes \Sigma(\pi) = 1 \in \KK^0 (C_0 (X), C_0 (X))[\tfrac12]$. 

\end{proof}

\noindent
We denote the element $\Sigma_{N_YX}$ as $\Sigma(\pi)^{-1}$.

\smallskip
\noindent
The above result is in fact valid in greater generality. Let $V\xrightarrow{\pi} B$ be an orientable
real vector bundle of rank $\ell$ on a (possibly non-compact) Witt space $B$. 
We know that there exists a well-defined element $[D_{V/B}^{\mathrm{sign}}] 
\in \mathrm{KK}^{\ell}(C_0(V), C_0(B))\,,\;\ell={\rm rank}\; V$, a corresponding analytic transfer class
\begin{equation*}
	\Sigma(\pi) = 2^{-\lfloor \ell/2\rfloor}[D_{V/B}^{\mathrm{sign}}] \in \mathrm{KK}^{\ell}(C_0(V), C_0(B))[\tfrac12]
\end{equation*}
and that this element induces a Gysin map by left Kasparov product
$$\pi^! : \mathrm{KK}^{j}(C_0(B), \cplx)[\tfrac12] \rightarrow \mathrm{KK}^{j+\ell}(C_0(V), \cplx)[\tfrac12].$$
Moreover $\pi^! \mathrm{sign}_K(B)=\mathrm{sign}_K(V) \in \mathrm{KK}^{\dim B+\ell}(C_0(V), \cplx)[\tfrac12]$.\\
Proceeding as above we understand that there exists $\Sigma_V\in \mathrm{KK}^{\ell}(C_0 (B),C_0(V))[\tfrac12]$  such that  $\Sigma_V \otimes \Sigma(\pi) = 1 \in \KK^0 (C_0 (B), C_0 (B))[\tfrac12]$.
We denote the element $\Sigma_V$ by $\Sigma (\pi)^{-1}$. 
Thus if $j:B \hookrightarrow V$ is the zero embedding, then there exists a homomorphism 
$$j^!: \KK^i (C_0 (V), \cplx)[\tfrac12] \lra \KK^{i+\ell} (C_0 (B), \cplx)[\tfrac12]$$ 
given by
\[ j^! (-) := \Sigma (\pi)^{-1}\otimes -. \]
which is the inverse of $\pi^!$. Morover it is clear from the equality $\pi^! \mathrm{sign}_K(B)= \mathrm{sign}_K(V)$
that $j^! \mathrm{sign}_K(V) = \mathrm{sign}_K(B)$ in  $\KK^{\dim B} (C_0 (B), \cplx)[\tfrac12]$ or, equivalently,
\begin{equation}\label{gysinressignopzerosectemb-0}
 \Sigma (\pi)^{-1}\otimes \mathrm{sign}_K(V) = \mathrm{sign}_K(B)\;\;\text{in}\;\; \KK^{\dim B} (C_0 (B), \cplx)[\tfrac12]\,.
 \end{equation}

\smallskip
We go back to our normally non-singular inclusion  of Witt spaces $X\hookrightarrow Y$.
We have seen that there is then an orientable  normal  bundle
$\pi:N_Y X\to X$, an open neighborhood $\mathcal{U} \subset Y$ of $X$, and a
stratified diffeomorphism $\varphi: \mathcal{U}  \longrightarrow X$.
We proceed in parallel to Brodzki, Mathai, Rosenberg, Szabo
\cite[Example 3.3]{BroMatRosSza:NCDDBK}.
From the above discussion, see \eqref{gysinressignopzerosectemb-0}, we record that  for the normal bundle $N_Y X\xrightarrow{\pi}
X$ it holds that 
\begin{equation}\label{gysinressignopzerosectemb}
\Sigma (\pi)^{-1}\otimes \mathrm{sign}_K(N_YX) = \mathrm{sign}_K(X);\;\text{in}\;\; \KK^{\dim X} (C_0 (X), \cplx)[\tfrac12] \,.
\end{equation}
The open inclusion $i: \mathcal{U} \subset Y$ defines
an element 
\[ i! \in \KK^m  (C_0 (U), C_0 (Y)),  \;\;\;\forall m\in \mathbb{N},\]
corresponding to the  homomorphism of $C_0 (U)$ in $C_0 (Y)$ obtained by  extension by zero.
(This works for any open subset $U$ of a locally compact space $Y$.)
This defines by Kasparov multiplication on the left a {\em restriction homomorphism}
$ \KK^\ell  (C_0 (Y), \cplx)\to \KK^\ell  (C_0 (\mathcal{U} ), \cplx)$, $\forall \ell\in \mathbb{N}$.
Note that $\mathcal{U} $ is a Witt space and hence
 there exists a signature class $[D_{\mathcal{U}}^{\mathrm{sign}}] \in \KK^{\dim Y} (C_0 (\mathcal{U} ),\cplx)$
 and a corresponding analytic orientation class
\begin{equation*}
	\mathrm{sign}_K(\cal U) = 2^{-\lfloor \dim Y/2 \rfloor}[D_{\mathcal{U}}^{\mathrm{sign}}] 
	\in \KK^{\dim Y} (C_0 (\mathcal{U} ),\cplx)[\tfrac12].
\end{equation*}

\begin{lemma} \label{lem.restrtoopen}
Restriction to open subsets preserves $\KK$-classes of signature operators
of Witt spaces, i.e.
\[ i! \otimes [D_Y^{\mathrm{sign}}] = [D_{\mathcal{U}}^{\mathrm{sign}}] \in \KK^{\dim Y} (C_0 (\mathcal{U} ),\cplx). \]
Hence we also have
\[ i! \otimes \mathrm{sign}_K(Y) = \mathrm{sign}_K(\cal U) \in \KK^{\dim Y} (C_0 (\mathcal{U} ),\cplx)[\tfrac12]. \]
\end{lemma}
\begin{proof}
 This is discussed in  \cite[Rem. 2.17, p. 419]{Hil:FKBPVL} (take $N=point$ there).
\end{proof}

Next, 
 the stratified diffeomorphism
$\varphi: \mathcal{U}  \longrightarrow N_Y X$
gives an invertible element
\[ [\varphi] \in \KK^0  (C_0 (N_Y X), C_0 (\mathcal{U} ))  \]
which induces an  isomorphism $\KK^j (C_0 (\mathcal{U} ), \cplx) \to \KK^j (C_0 (N_Y X), \cplx)$
by left Kasparov product.
By  Proposition 
\ref{prop.stratdiffeoinvariance} (stratified diffeomorphism invariance 
of the signature class on a Witt space)
we have 
\begin{lemma} \label{lem.homeophipressignops}
It holds that $$[\varphi] \otimes [D_{\mathcal{U}}^{\mathrm{sign}}] = [D_{N_Y X}^{\mathrm{sign}}].$$
\end{lemma}

We are in the position to define the Gysin restriction homomorphism
$j^!$ associated to the normally non-singular inclusion
$j: X\hookrightarrow Y$ of our Witt spaces.

\begin{defn}\label{def:nns-gysin}
Let $j: X\hookrightarrow Y$ be a normally non-singular inclusion of Witt spaces.
 Let $\ell$ be the codimension of the inclusion. Then
Kasparov multiplication
\[ \KK^\ell (C_0 (X), C_0 (N_Y X))[\smlhf] \otimes \KK^0 (C_0 (N_Y X), C_0 (\mathcal{U}))[\smlhf] 
    \otimes \KK^0 (C_0 (\mathcal{U}), C_0 (Y))[\smlhf] \]
\[   \longrightarrow
    \KK^\ell (C_0 (X), C_0 (Y))[\smlhf]  \]
defines an element
\begin{equation} \label{equ.gysinrestrelement}  
	\Sigma(j) := \Sigma(\pi)^{-1} \otimes [\varphi] \otimes i! 
        \in \KK^\ell (C_0 (X), C_0 (Y))[\tfrac12],
\end{equation}         
which we refer to as the {\bf analytic transfer map of the nns inclusion $j$}.
Kasparov multiplication with this element on the left defines the analytic Gysin restriction
\[
j^!: \KK^{m+\ell} (C_0 (Y),\cplx)[\tfrac12] \longrightarrow \KK^m (C(X),\cplx)[\tfrac12], 
\]
\[ j^! (-) := \Sigma(j) \otimes -. \]
\end{defn}

We  have, finally, the following result:

\begin{thm} \label{thm.gysinresofsignop}
Let $X,Y$ be Witt spaces 
and let $j: X \hookrightarrow Y$ be a normally non-singular inclusion of codimension $\ell$.
Then Gysin restriction
on analytic K-homology 
$$j^!: \KK^{*+\ell} (C_0 (Y),\cplx)[\tfrac12] \longrightarrow \KK^{*} (C_0 (X),\cplx)[\tfrac12],$$ 
sends the analytic signature orientation class of $Y$ to the
analytic signature orientation class of $X$, i.e.
\begin{equation}\label{j-signature} 
j^! \mathrm{sign}_K(Y) = \mathrm{sign}_K(X).
\end{equation}
\end{thm}
\begin{proof}
Using  \ref{gysinressignopzerosectemb}
and Lemmas \ref{lem.restrtoopen}, \ref{lem.homeophipressignops}, we have
\begin{align*}
j^! \mathrm{sign}_K(Y)
&= \Sigma(\pi)^{-1} \otimes [\varphi] \otimes i! \otimes \mathrm{sign}_K(Y) \\
&= \Sigma(\pi)^{-1} \otimes [\varphi] \otimes \mathrm{sign}_K(\cal U) \\
&=  \Sigma(\pi)^{-1} \otimes \mathrm{sign}_K(N_Y X) \\
&= \mathrm{sign}_K(X).
\end{align*}
\end{proof}

\begin{remark}
Notice that even assuming the normal bundle $N_Y X$ to be spin-c, we would not be able to use the
spin-c Gysin map in order to
obtain the analogue of \eqref{j-signature}  for the spin-c Dirac operator; indeed, we do not know, in general, that the spin-c Dirac operator defines a K-homology class in the stratified setting.
\end{remark}  

Functoriality in this case is simple to establish (see \cite[\S 4.1]{Hil:FKBPVL}).

\begin{thm}
If $X,$ $Y,$ and $Z$ are Witt spaces and
\begin{equation*}
	X \xhookrightarrow{\phantom{x}i\phantom{x}}
	Y \xhookrightarrow{\phantom{x}j\phantom{x}}
	Z
\end{equation*}
are nns inclusions, then
\begin{equation*}
	\Sigma(j\circ i) = \Sigma(i) \otimes \Sigma(j) \in \mathrm{KK}(C_0(X), C_0(Z))[\tfrac12]
\end{equation*}
\end{thm}

Similarly, we can deduce base change from work we have already carried out.

\begin{prop} \label{lem.basechangeforgysinrestrictiontozero}
(Base Change for Gysin Restriction to Zero-Sections)
Let $Y$ be a smoothly stratified compact Witt space and $p: X\to Y$ an
oriented real vector bundle of rank $r$ over $Y$. 
Suppose that $W$ is a compact smoothly stratified space and $f:W\to Y$ a  smooth stratified
map. Consider the Cartesian diagram
\begin{equation*}
	\xymatrix{
	X_W \ar[d]_{p_W} \ar[r]^g & X \ar[d]^p \\
	W \ar[r]_f & Y.
	}	
\end{equation*}
Then the base change formula
\begin{equation*}
	f_* \circ j^!_W = j^! \circ g_* 	
\end{equation*}
holds for the Gysin restrictions 
\begin{equation*}
	j^!: \KK^n (C_0 (X),\cplx)[\tfrac12] \to \KK^{n-r} (C(Y),\cplx)[\tfrac12], \quad
	j^!_W: \KK^n (C_0 (X_W),\cplx)[\tfrac12] \to \KK^{n-r} (C(W),\cplx)[\tfrac12] 
\end{equation*}
associated to the zero-section inclusions
$j: Y\hookrightarrow X$ and $j_W: W\hookrightarrow X_W$.
\end{prop}

\begin{proof}
According to Theorem \ref{thm.basechangeforstratfiberbundles},
$p^! f_* = g_* p^!_W$.
Recall that for vector bundles, $j^!$ and $p^!$ are mutually inverse
(Thom-) isomorphisms. 
An application of $j^!$ from the left on both sides yields
\[ f_* = j^! p^! f_* = j^! g_* p^!_W. \]
Applying $j^!_W$ from the right on both sides, we obtain
\[ f_* j^!_W = j^! g_* p^!_W j^!_W = j^! g_*. \]
\end{proof}

\section{Compatibility of Analytic and Topological Orientations}
\label{sec.compatanalytictoporient}

\subsection{Prelude: The Four-Sphere}

We illustrate the orientation classes considered in this work
for the case of a $4$-sphere, where these classes are determined
by their Chern characters.
Let $\nu$ be the trivial real 4-plane bundle over a point.
Thus the total space of $\nu$ is $\real^4$ and its Thom space
is $\Th (\nu) = S^4$, the $4$-sphere. 
We shall describe the (real) Sullivan orientation
\[ \Delta (\nu) \in \wKO^4 (\Th (\nu))[\smlhf] = \wKO^4 (S^4)[\smlhf] \]
of the bundle $\nu$, the Sullivan orientation 
\[ \Delta (S^4) \in \KO_4 (S^4)[\smlhf] \]
of the manifold $S^4$, and its relation to the $\K$-homology class
of the signature operator on $S^4$.
Let 
\[  \beta = [H] - [1] \in \wK^0 (S^2) = \wK^{-2} (S^0)= \pi_2 (\K) \]
be the Bott generator. The complex line bundle $H$ on $S^2 = \cplx P^1$
is holomorphic and dual to the Hopf line bundle whose fiber over
a point in $\cplx P^1$ is the complex line in $\cplx^2$ represented by
the point.
Let $u := c_1 (H) \in H^2 (S^2;\intg)$ be the canonical generator
given by the first Chern class of $H$.
The Chern character is an isomorphism
\[ \ch: \wK^0 (S^2) \stackrel{\simeq}{\longrightarrow}
H^2 (S^2;\intg)\cong \intg \]
given by $\ch (\beta) = u$.
We shall use the external products
\[ \Smash: \widetilde{H}^* (X) \otimes \widetilde{H}^* (Y)
\longrightarrow \widetilde{H}^* (X\Smash Y), \quad 
  \Smash: \widetilde{\K} (X) \otimes \widetilde{\K} (Y)
\longrightarrow \widetilde{\K} (X\Smash Y). \]   
Note that for
$X=S^2,$ $Y=S^2,$ we have $S^2 \Smash S^2 = S^4.$
The map
\[ \widetilde{H}^* (S^2;\intg)
\xlra{-\Smash u}
\widetilde{H}^* (S^4;\intg) \]
is an isomorphism, and by Bott periodicity, the Bott map
\[ \Per_\cplx: \widetilde{\K} (S^2)
\xlra{-\Smash \beta}
\widetilde{\K} (S^4) \]
is an isomorphism as well.   
As the diagram  
\[ \xymatrix{
	\widetilde{\K} (S^2) \ar[d]_{\ch} \ar[r]^{-\Smash \beta} &
	\widetilde{\K} (S^4) \ar[d]^{\ch} \\
	\widetilde{H}^* (S^2;\intg) \ar[r]^{-\Smash u} &
	\widetilde{H}^* (S^4;\intg)    
} \]
commutes (Husemoller \cite[p. 308]{husemoller}), we have
$\ch (\beta^2) = \ch (\beta \Smash \beta) =u \Smash u = u^2.$
Let $c: \KO \to \K$ be complexification as a morphism of ring spectra,
inducing homomorphisms
\[  c_*: \KO^* (X) \to \K^* (X), \quad 
c_*: \KO_* (X) \to \K_* (X).  \]
The Pontrjagin character $\ph = \ch \circ c_*$ localizes to 
$\ph [\smlhf]$ given by the composition
\[ \xymatrix{
	\widetilde{\KO}^0 (S^4)[\smlhf] \ar[dr]_{\ph [\smlhf]}^\simeq 
	\ar[r]^{c_* [\smlhf]}_\simeq
	& \widetilde{\K}^0 (S^4)[\smlhf] \ar[d]^{\ch [\smlhf]}_\simeq \\
	& H^4 (S^4;\intg)[\smlhf].
} \]     
Here we have used suspension isomorphisms to identify
\[ \xymatrix@C=10pt{  
	\widetilde{\KO}^0 (S^4) \ar[d]_{c_*} \ar@{=}[r]^\sim &
	\widetilde{\KO}^{-4} (S^0) \ar@{=}[r] &
	\KO^{-4} (\pt) \ar@{=}[r] &
	\pi_4 (\KO) \ar[d]^{c_*} \\
	\widetilde{\K}^0 (S^4) \ar@{=}[r]^\sim &
	\widetilde{\K}^{-4} (S^0) \ar@{=}[r] &
	\K^{-4} (\pt) \ar@{=}[r] &
	\pi_{4} (\K).       
} \]  
(Note that $c_*$ commutes with suspension isomorphisms, since it is
induced by a morphism of spectra.)
On $\pi_4$, complexification induces multiplication by two,
$c_* =2: \pi_4 (\KO)=\intg \to \intg = \pi_4 (K)$.
Let 
\[ a \in \pi_4 (\KO)[\smlhf] = \wKO^{-4} (S^0)[\smlhf]= 
\wKO^0 (S^4)[\smlhf] = \intg [\smlhf] \]
be the unique element with
\[ c_* (a) = \beta^2. \]
This is the generator used in
\cite{Ban:TSSKSS} and
by Oscar Randal-Williams in \cite{randalfamilysignature}.
Its Pontrjagin character satisfies
$\ph [\smlhf] (a) = \ch [\smlhf] (\beta^2) = u^2.$
The generator $u^2 \in H^4 (S^4;\intg)$ agrees with
the Thom class $u (\nu) \in H^4 (S^4;\intg)$ of the trivial rank
$4$-bundle $\nu$ over a point.
By construction of the Sullivan orientation 
(see also \cite[p. 13, Lemma 3.8]{Ban:TSSKSS}),
\[
	\ph [\smlhf] (\Delta (\nu))
	 = L^{-1} (\nu) \cup u (\nu) 
	= u(\nu) \\
	 = u^2 = \ph [\smlhf] (a).
\]
Since $\ph [\smlhf]$ is an isomorphism,
\[ \Delta (\nu) = a. \]  
This is Lemma 3.9 in \cite{Ban:TSSKSS} for the case of a $4$-plane bundle.
Since $\Delta (S^4) \in \KO_4 (S^4)[\smlhf]$ is an orientation,
we have
\begin{equation} \label{equ.deltas4issuspofone}
	\Delta (S^4) = \sigma_* (1), 
\end{equation}
$\sigma_*: \wKO_0 (S^0)[\smlhf] \stackrel{\simeq}{\longrightarrow}
\wKO_4 (S^4)[\smlhf]$
is the suspension isomorphism and 
$1 \in \pi_0 (\KO [\smlhf]) = \wKO_0 (S^0)[\smlhf]$ the unit of the ring
$\pi_* (\KO [\smlhf])$.
Another way to see this is as follows:
The map $\Delta_\SO: \MSO \to \KO[\smlhf]$ has been constructed
in \cite{Ban:TSSKSS} as a morphism of ring spectra. In particular, the diagram
\[ \xymatrix{
	\widetilde{\MSO}_4 (S^4) \ar[r]^{\Delta_*} &
	\wKO_4 (S^4)[\smlhf] \\
	\widetilde{\MSO}_0 (S^0) \ar[u]^{\sigma_*}_\simeq \ar[r]^{\Delta_*} &
	\wKO_0 (S^0)[\smlhf] \ar[u]_{\sigma_*}^\simeq
} \]
commutes and $\Delta_* (1)=1 \in \wKO_0 (S^0)[\smlhf]$.
The reduced bordism groups $\widetilde{\MSO}_n (X)$ are given by the kernel of
$\MSO_n (X) \to \MSO_n (\pt)),$
so that $[f:M^n \to X] \in \MSO_n (X)$ is an element of
$\widetilde{\MSO}_n (X)$ if and only if $[M]=0\in \MSO_n$. 
Since $[S^4]=0\in \MSO_4$,
it follows that $[\id_{S^4}] \in \widetilde{\MSO}_4 (S^4)$.
The $\MSO$-orientation of a smooth oriented closed manifold $M^n$ 
is given by the identity map $[\id_M] \in \MSO_n (M)$ and 
$[\id_{S^4}] = \sigma_* (1) \in \widetilde{\MSO}_4 (S^4).$
By definition, $\Delta (S^4) = \Delta_* [\id_{S^4}].$
It follows that indeed
$\Delta (S^4) = \Delta_* \sigma_* (1)
= \sigma_* \Delta_* (1) = \sigma_* (1).$  
Let $\Psi^2: \KO[\smlhf] \to \KO [\smlhf]$ be the stable Adams operation,
constructed as a morphism of $\mathbb{E}_\infty$-ring spectra.
Thus the diagram
\[ \xymatrix{
	\wKO_4 (S^4)[\smlhf] \ar[r]^{\Psi^2}_\simeq &
	\wKO_4 (S^4)[\smlhf] \\
	\wKO_0 (S^0)[\smlhf] \ar[u]^{\sigma_*}_\simeq \ar[r]^{\Psi^2}_\simeq &
	\wKO_0 (S^0)[\smlhf] \ar[u]_{\sigma_*}^\simeq
} \]
commutes and $\Psi^2 (1)=1 \in \wKO_0 (S^0)[\smlhf]$.
Using (\ref{equ.deltas4issuspofone}), we deduce that
\begin{equation} \label{equ.psi2deltas4} 
	\Psi^2 \Delta (S^4) = \Psi^2 \sigma_* (1)
	= \sigma_* \Psi^2 (1) = \sigma_* (1) = \Delta (S^4).
\end{equation}

\begin{remark}
	The $4$-fold periodicity of $\KO [\smlhf]$ does \emph{not} commute
	with the stable Adams operation $\Psi^2$.
	Indeed, Sullivan's periodicity map
	\[ \Per_\real: \KO^0 (X)[\smlhf]
	\xlra{-\Smash a}
	\KO^0 (X \wedge S^4)[\smlhf] \]  
	is given by multiplication with $a = \Delta (\nu)$,
	see Sullivan \cite[p. 202]{Sul:GTP}.
	The stable Adams operation
	$\Psi^2: \pi_{4k} (\KO [\smlhf]) \to \pi_{4k} (\KO [\smlhf])$
	acts on the preferred generator by
	$\Psi^2 (a^k) = 4^k \cdot a^k.$
	Hence
	$\Psi^2 \Per_\real (1) = \Psi^2 (a) = 4a,$
	which does not equal
	$\Per_\real \Psi^2 (1) = \Per_\real (1) = a.$
\end{remark}

The $4$-sphere is not an almost complex manifold, but since it
is oriented and $w_2 (S^4)=0$, it is a spin manifold.
The spin structures on $S^4$ are in one-to-one correspondence
with elements of $H^1 (S^4;\intg/_2)=0$. Thus $S^4$ has a unique
spin structure, which determines in particular a canonical $\Spin^c$-structure 
on $S^4$.
Therefore, $S^4$ has an integral Atiyah-Bott-Shapiro orientation
\[ [S^4]_\ABS \in \K_4 (S^4).   \]
Since $c_*$, being induced by a morphism $c$ of ring spectra, 
is a natural transformation of homology theories, the diagram
\[ \xymatrix{
	\wKO_4 (S^4)[\smlhf] \ar[r]^{c_*} &
	\wK_4 (S^4)[\smlhf] \\
	\wKO_0 (S^0)[\smlhf] \ar[u]^{\sigma_*}_\simeq \ar[r]^{c_*} &
	\wK_0 (S^0)[\smlhf] \ar[u]_{\sigma_*}^\simeq
} \]
commutes and $c_* (1)=1 \in \wK_0 (S^0)[\smlhf]$.
Since $[S^4]_\ABS$ is an orientation, the equation
\[  [S^4]_\ABS = \sigma_* (1)  \]
holds; see Atiyah-Bott-Shapiro \cite[p. 30]{AtiBotSha:CM} for the cohomological case.
Using (\ref{equ.psi2deltas4}), we compute
\[ c_* (\Psi^2)^{-1} \Delta (S^4)
= c_* \Delta (S^4) = c_* \sigma_* (1)
= \sigma_* c_* (1) = \sigma_* (1), \] 
which implies
\begin{equation} \label{equ.deltaandabs}
	c_* (\Psi^2)^{-1} \Delta (S^4) = [S^4]_\ABS. 
\end{equation}

Now let $D^\sign = d+d^*$ denote the signature operator with respect to some Riemannian 
metric on $S^4$ and let
\[ [D^\sign] \in \K^\an_4 (S^4) \]
be the associated class in analytic $\K$-homology. 
(This class exists integrally, not just
away from $2$.)
The Atiyah-Singer $\La$-classes were introduced in
\cite[p. 577]{asindellop3}.
The unstable cohomological Atiyah-Singer class $\La^*_u$ is associated to the
power series 
\[ x \coth (x/2) =
2 + \frac{1}{6} x^2 
- \frac{1}{360} x^4 
+ \frac{1}{15120} x^6
- \frac{1}{604800} x^8 \pm \ldots.  \] 
(The stable Atiyah-Singer class $\La^*_s$ is 
associated to $(x/2) \coth (x/2)$, but will not be used here.
The Hirzebruch $L$-class is associated to $x \coth x$.
In \cite[p. 200]{stong}, Stong denotes $\La^*_u$ by $\delta$.)
For a $4$-dimensional manifold $M^4,$
$\La^*_u (M^4) = 4 + \frac{1}{3} p_1 (TM).$
Since the $4$-sphere is stably parallelizable,
$\La^*_u (S^4) = 4.$
The homological unstable Atiyah-Singer class $\La^u_*$ is the Poincar\'e
dual $\La^u_* = \La^*_u \cap [M]_H$, 
where $[M]_H \in H_n (M;\intg)$ denotes the fundamental class
in ordinary integral homology of a closed, $H\intg$-oriented $n$-dimensional
manifold $M$. 
Thus
\[ \La^u_* (S^4) = 4[S^4]_H. \]
In fact, for an even-dimensional smooth manifold, 
the unstable Atiyah-Singer $\La$-class is given by
\[ \La^u_* (M) = \sum_j 2^j L_{2j} (M) \in H_* (M;\rat),  \]
where the classes $L_{2j} (M) \in H_{2j} (M;\rat)$ are the components of the
Poincar\'e dual $L_* (M) = L^* (TM)\cap [M]_H$ of the Hirzebruch $L$-class
$L^* (TM)$, see also Moscovici-Wu \cite[p. 14]{moscoviciwu}.
Under the Chern character, the $\K$-class of the signature operator
is a refinement of the unstable Atiyah-Singer class. 
This follows from the Atiyah-Singer index theorem, see also
Rosenberg-Weinberger \cite[p. 48]{RosWei:SO}.
Thus
$\ch [D^\sign] = \La^u_* (S^4)$
and it follows that
\[ \ch [D^\sign] = 4[S^4]_H. \]
Note that since the signature of $S^4$ vanishes, $[D^\sign]$ is actually a
class in reduced $\K$-homology, $[D^\sign] \in \wK_4 (S^4)$.
Under the standard isomorphism
\[ \varphi: \K^\geo_* (X) \stackrel{\simeq}{\longrightarrow}
\K^\topo_* (X), \quad 
\varphi [M,E,f] = f_* ([E] \cap [M]_\ABS), \]
the topological homological Chern character $\ch: \K^\topo_* (X) \to H_* (X;\rat)$
corresponds to the homological Chern character 
$\ch: \K^\geo_* (X) \to H_* (X;\rat)$ given by
\[ \ch [M,E,f] = f_* (\ch (E) \cup \td (TM) \cap [M]_H),  \]
see e.g. Jakob \cite[p. 77, 4.2]{Jak:BDH}.
Here one uses the interesting fact that the Todd class $\td (-),$ 
a priori only defined for complex bundles, actually survives to
$\Spin^c$-vector bundles $T$, such as the tangent bundle
$TM$. Indeed, $\Spin^c$ vector bundles have a first Chern class $c_1 (T)$
(the integral class reducing to $w_2 (T)$ which is defined by the 
$\Spin^c$-structure) and one then defines
$\td (T) = e^{c_1 (T)/2} \hat{A} (T)$
in rational cohomology (Baum-Douglas \cite[p. 136]{BauDou:HIT}),
where $\hat{A} (T)$ is the $\hat{A}$-polynomial in the Pontrjagin classes
of $T$.
For $S^4$, $c_1 (TS^4)=0 \in H^2 (S^4)=0$ and
$\hat{A} (TS^4) = 1 -\frac{1}{24} p_1 (TS^4) =1.$
Thus $\td (TS^4)=1$ and the fundamental class 
$[S^4, E=1, f=\id] \in K^\geo_4 (S^4)$ has Chern character
\[ \ch [S^4,1,\id] = \id_* (\ch (1)\cup \td (TS^4) \cap [S^4]_H) 
= [S^4]_H. \]
Since $\ch \circ \varphi = \ch$, we get
\[
	\ch [S^4]_\ABS
	= \ch (\id_* (1\cap [S^4]_\ABS)) 
	= \ch \varphi [S^4, 1, \id] \\
	= \ch [S^4, 1, \id],
\]
and thus
$\ch [S^4]_\ABS = [S^4]_H.$ Using the standard isomorphism
$\mu: \K^\geo_* (X) \to \K^\an_* (X)$
(Baum-Douglas \cite{BauDou:HIT}),
we consider the diagram
\[ \xymatrix{
	K_4^\topo (S^4)\otimes \intg[\smlhf] \ar@{^{(}->}[d]_\loc
	& K^\geo_4 (S^4)\otimes \intg[\smlhf] \ar[l]_\varphi^\simeq 
	\ar@{^{(}->}[d]_\loc \ar[r]^\mu_\simeq 
	& K^\an_4 (S^4) \otimes \intg [\smlhf] \ar@{^{(}->}[d]_\loc \\
	K_4^\topo (S^4) \otimes \rat \ar[dr]_{\ch_\rat}^\simeq
	& K^\geo_4 (S^4) \otimes \rat \ar[l]_{\varphi_\rat}^\simeq 
	\ar[d]^{\ch_\rat}_\simeq \ar[r]^{\mu_\rat}_\simeq 
	& K^\an_4 (S^4) \otimes \rat \ar[dl]^{\ch_\rat}_\simeq \\
	& H_{2*} (S^4;\rat). &
} \]
The vertical arrows are localization maps induced by tensoring with
$\intg [\smlhf] \hookrightarrow \rat$.
These maps are injective, since there is no torsion in the
$\K$-homology of $S^4$. The left-hand part of the diagram commutes
by Jakob \cite[p. 77, 4.2]{Jak:BDH} and the right hand part commutes
according to Baum-Douglas \cite[p. 154]{BauDou:HIT}.
For the Chern characters on the rational groups we have
\begin{align*}
	\ch_\rat \loc (\mu \varphi^{-1}) [S^4]_\ABS
	&= \ch_\rat \loc [S^4]_\ABS 
	= \ch [S^4]_\ABS 
	= [S^4]_H \\
	&= \ch (\smlquart [D^\sign]) 
	= \ch_\rat \loc (\smlquart [D^\sign]).
\end{align*}
Since $\ch_\rat$ is an isomorphism, it follows that
$\loc (\mu \varphi^{-1}) [S^4]_\ABS
= \loc (\smlquart [D^\sign]).$
As $\loc$ is injective,
\[ \mu \varphi^{-1} [S^4]_\ABS
= \smlquart [D^\sign]. \]
By (\ref{equ.deltaandabs}),
\[ \mu \varphi^{-1} c_* (\Psi^2)^{-1} \Delta (S^4) = 
\mu \varphi^{-1} [S^4]_\ABS = \smlquart [D^\sign], \] 
a relation that we will first generalize to arbitrary oriented smooth
manifolds in Theorem \ref{thm.sullivanandksignoponmanifolds}, 
and then further to singular (Witt)
spaces in Theorem \ref{thm.sullivanandksignoponwitt}.

\subsection{The Manifold Case.}

For finite CW pairs $(X,A)$, we identify topological and analytic complex K-homology
via a natural isomorphism
\begin{equation} \label{equ.identktopandkan}
\K^\topo_* (X,A) \cong \K^\an_* (X,A) 
\end{equation}
of homology theories as in \cite{LanNikSch:LC} and \cite{RosWei:SO}.
Let 
\[ \Psi^2: \KO [\smlhf] \longrightarrow \KO [\smlhf] \]
be the stable Adams operation, constructed as a morphism of 
$\mathbb{E}_\infty$-ring spectra. This is an equivalence,
let $(\Psi^2)^{-1}$ denote an inverse.
Let 
\[ c: \KO \longrightarrow \K \]
denote complexification as a morphism of ring spectra.
\begin{thm} \label{thm.sullivanandksignoponmanifolds}
Let $M$ be an $n$-dimensional closed oriented Riemannian manifold and
\[ \Delta_\SO (M) \in \KO^\topo_n (M)[\smlhf] \]
its Sullivan orientation, \cite{Sul:GTP}.
Under the above identification (\ref{equ.identktopandkan}), the element
\[ c(\Psi^2)^{-1} \Delta_\SO (M) \in \K^\topo_n (M)[\smlhf] \] 
corresponds to the signature-operator orientation
\[ \mathrm{sign}_K(M) = 2^{-\lfloor n/2 \rfloor} [D^\sign_M] \in \K^\an_n (M)[\smlhf]. \]
\end{thm}
\begin{proof}
Let $\KO$ denote the $8$-periodic ring spectrum representing real K-theory
and $\K$ the $2$-periodic ring spectrum representing complex K-theory.
The homotopy ring of $\K$ is $\pi_* (\K)=\intg [\beta^{\pm 1}],$
where $\beta$ is the complex Bott element in degree $2$, i.e.
$\beta$ is represented by the reduced canonical complex line bundle
$H-1 \in \widetilde{\K}^0 (S^2)$.
On $\pi_4$, the complexification $c:\KO \to \K$ induces multiplication by $2$,
$c_* =2: \pi_4 (\KO) = \intg \to \intg = \pi_4 (\K)$.
Thus there does not exist an element in $\pi_4 (\KO)$ that maps to
$\beta^2$. But after inverting $2$, such an element exists.
Let $a\in \pi_4 (\KO)[\smlhf]$ be the unique element with
\begin{equation} \label{equ.cplxofaisbottsquare} 
c_* (a) = \beta^2. 
\end{equation}
The localization $\KO [\smlhf]$ is a $4$-periodic ring spectrum
with homotopy ring
$\pi_* (\KO)[\smlhf] =\intg [\smlhf] [a^{\pm 1}].$
Its connective cover $\ko [\smlhf]$ has homotopy ring
$\pi_* (\ko)[\smlhf] =\intg [\smlhf] [a].$
Let $\syml (R)$ denote the (projective) symmetric algebraic L-theory
spectrum of a commutative unital ring $R$ with involution, introduced first by Ranicki.
Let
\[ \kappa: \KO [\smlhf] \stackrel{\simeq}{\longrightarrow} 
   \syml (\real) [\smlhf] \]
be the equivalence of $\mathbb{E}_\infty$-ring spectra
constructed by the first named author in
\cite[Prop. 2.1]{Ban:TSSKSS}. It induces the ring isomorphism
\[ \intg [\smlhf] [a^{\pm 1}] \longrightarrow
   \intg [\smlhf] [x^{\pm 1}],~
   a \mapsto x, \]
on homotopy rings, where $x$ denotes the signature $1$ generator.   
Let $\MSPL$ be the Thom spectrum associated to the bordism theory of oriented PL manifolds.
In \cite[p. 385, Prop. 15.8]{Ran:ATSULNP1}, Ranicki constructed a morphism of ring spectra
\[ \sigma^*: \MSPL \longrightarrow \syml (\intg) \]
such that the resulting $\syml (\intg)$-homology fundamental class
$[M]_\syml := \sigma^* [\id_M] \in \syml (\intg)_n (M)$
of an $n$-dimensional closed oriented PL manifold $M$
hits the Mishchenko-Ranicki symmetric signature
\[ \sigma^* (M) = A[M]_\syml \in L^n (\intg [\pi_1 M]) \]
under the assembly map
$A: \syml (\intg)_n (M) \to L^n (\intg [\pi_1 M]).$
(Ranicki extended $\sigma^*$ to a morphism of ring spectra
$\MSTOP \to \syml (\intg)$ in \cite[p. 290]{Ran:TSO}, 
but we shall not require this extension for the purposes of the present paper.)
Technically, we will work with $\sigma^*$ as constructed by
Laures and McClure in \cite{LauMcC:MPQS} using ad-theories.
Their incarnation of $\sigma^*$ is an $\mathbb{E}_\infty$-ring map
(\cite[1.4]{LauMcC:MPQS}).
The localization morphism $\syml (\intg) \to \syml (\intg)[\smlhf]$
is a morphism of ring spectra. Thus its composition with $\sigma^*$
is a morphism of ring spectra
$\MSPL \to \syml (\intg)[\smlhf]$, which we shall also denote by $\sigma^*$.
By \cite[p. 243]{Ran:ATSIAT}, $\sigma^*$
induces on homotopy groups the map
\[ \sigma^*_{\pt}: \Omega^\SPL_{4k} (\pt) = \MSPL_{4k} (\pt) \longrightarrow
   \syml (\intg) [\smlhf]_{4k} = \intg [\smlhf] \langle x^k \rangle \]
given by
\begin{equation} \label{equ.ranorientoncoeffs}
\sigma^*_{\pt} [M^{4k}] = \sigma (M)\cdot x^k, 
\end{equation}
where $\sigma (M)\in \intg$ denotes the signature of $M$.
Let 
\[ \Delta^\circ_\SPL: \MSPL \longrightarrow
        \KO [\smlhf]    \]
denote the composition
\[ \MSPL \stackrel{\sigma^*}{\longrightarrow} 
    \syml [\smlhf] \stackrel{\kappa^{-1}}{\longrightarrow} 
    \KO [\smlhf]. \]
The first named author showed
in \cite[Prop. 3.3]{Ban:TSSKSS} that $\Delta^\circ_\SPL$ 
is homotopic to Sullivan's orientation $\Delta_\SPL$.
As pointed out in \cite[Cor. 3.4]{Ban:TSSKSS},
this implies in particular that the Sullivan orientation is homotopic
to a morphism of homotopy ring spectra.
Under the canonical morphism $\MSO \to \MSPL$,
which is a morphism of homotopy ring spectra,
$\Delta^\circ_\SPL$ restricts to a morphism of homotopy ring spectra
\[ \Delta^\circ_\SO: \MSO \longrightarrow
        \KO [\smlhf]    \]
that is homotopic to Sullivan's $\Delta_\SO$.
As $\MSO$ is connective, $\Delta^\circ_\SO$ lifts, uniquely up to
homotopy, to the connective cover $\ko [\smlhf]$. The lift
$\MSO \rightarrow \ko [\smlhf]$    
will again be denoted by $\Delta^\circ_\SO$ and is
a morphism of homotopy ring spectra 
(\cite[p. 93, Thm. 4.28.(ii)]{Rud:TSOC}).
Let
$\cL_{AS}: \MSO \longrightarrow \ko [\smlhf]$
be the map of $\mathbb{E}_\infty$-rings constructed 
by Land, Nikolaus and Schlichting in
\cite[Theorem 8.5]{LanNikSch:LC}.
Within maps of $\mathbb{E}_\infty$-rings, $\cL_{AS}$
is characterized uniquely by the homomorphism it induces
on homotopy rings, which is given by
\[ [M^{4k}] \mapsto 2^{-2k} \sigma (M) a^k.  \]
We claim that $\Psi^2 \circ \cL_{AS}$ is homotopic to
$\Delta^\circ_\SO$.
In order to see this, we verify first that both induce the same homomorphism
on homotopy groups:
It is well-known that the Sullivan orientation on homotopy
groups is given by the signature. More precisely,
$\Delta_{\SO *} [M^{4k}] = \sigma (M) \cdot a^k.$
The stable Adams operation acts on powers of $a$ by
\begin{equation} \label{equ.stablesecondadamsona}
\Psi^2 (a^k) = 4^k \cdot a^k.  
\end{equation}
Therefore,
\begin{align*}
(\Psi^2 \cL_{AS})_* [M^{4k}]
 &= \Psi^2_* (2^{-2k} \sigma (M) a^k) 
 = 2^{-2k} \sigma (M) \Psi^2_* (a^k) \\
&= 2^{-2k} \sigma (M) 4^k a^k = \sigma (M) a^k 
 =  \Delta_{\SO *} [M^{4k}] 
 = \Delta^\circ_{\SO *} [M^{4k}].
\end{align*}
According to \cite[Theorem 8.5]{LanNikSch:LC},
the map
\[ \pi_0 \operatorname{Map}^{\operatorname{HoRing}}_{\operatorname{Sp}}
  (\MSO, \ko[\smlhf])
  \stackrel{\pi_*}{\longrightarrow} 
  \Hom_{\operatorname{Ring}} (\MSO_*, \ko [\smlhf]_*), \]
where $\pi_0 \operatorname{Map}^{\operatorname{HoRing}}_{\operatorname{Sp}}$
denotes the connected components of the space of maps of spectra that
are homotopy ring maps, 
is injective. Both $\Psi^2 \circ \cL_{AS}$ and
$\Delta^\circ_\SO$ are homotopy ring maps and we have seen that
$\pi_* (\Psi^2 \circ \cL_{AS}) = \pi_* (\Delta^\circ_\SO)$.
Thus by injectivity of $\pi_*$,
$\Psi^2 \circ \cL_{AS}$ is homotopic to
$\Delta^\circ_\SO$, establishing the claim.

Thm. 8.5 of \emph{loc cit} also states that the composition
\[ \MSO_* (X) \stackrel{\cL_{AS}}{\longrightarrow}
   \ko [\smlhf]_* (X) \stackrel{c}{\longrightarrow} \ku [\smlhf]_* (X) \]
sends $[f:M^n \to X]$ to 
$2^{-\lfloor n/2 \rfloor} f_* [D^\sign_M]$. Since
$\cL_{AS} \simeq (\Psi^2)^{-1} \circ \Delta^\circ_\SO \simeq
 (\Psi^2)^{-1} \circ \Delta_\SO,$
and thus
\[ c \circ \cL_{AS} \simeq c \circ (\Psi^2)^{-1} \circ \Delta_\SO, \]
 it follows that
\[
2^{-\lfloor n/2 \rfloor} [D^\sign_M]
= (c \circ \cL_{AS}) [\id_M]
= (c \circ (\Psi^2)^{-1} \circ \Delta_\SO)[\id_M]
=  (c \circ (\Psi^2)^{-1}) (\Delta_\SO (M)).
\]
\end{proof}

\bigskip

\subsection{The Case of Singular Witt Spaces.}

Our goal now is to extend the above arguments to the case of Witt pseudomanifolds,
which includes all pure-dimensional complex algebraic varieties. \\

Let $Z=\ism$ denote the ring of integers localized
at odd primes.
The Laurent polynomial ring $R:=Z[t,t^{-1}]$
is a $\intg$-graded ring with $\deg (t)=4$.
There is a canonical subring inclusion
$Z[t] \subset R,~ t \mapsto t.$
Let $\cplx P^2$ be complex projective space of complex dimension $2$.
Denote Witt bordism away from $2$ by
\[ W_* (X) := \Omega^\Witt_* (X)\otimes_\intg Z. \]
(Recall that by Proposition \ref{prop.bordofsmoothlystratandPLwitt}, 
we may assume that representatives
of $\Omega^\Witt_* (X)$ are smoothly stratified and thus possess a well-defined
signature operator.)
This is a $\intg$-graded homology theory with coefficients
\[ W_* (\pt) = Z[c], \quad  c:= [\cplx P^2 \to \pt]\otimes 1 
       \in W_4 (\pt). \]
(After inverting $2$, only the signature survives as an
invariant; the $2$- and $4$-torsion is removed.)
The $\intg$-graded abelian group $W_* (X)$ is a right module over 
the ring $W_* (\pt)$ as usual. The ring isomorphism
$Z[t] \to W_* (\pt)$ induced by
$t \mapsto c\otimes 1 \in W_4 (\pt)$
makes $W_* (X)$ into a right $Z[t]$-module. Thus $t$ acts on
$[f:V \to X] \otimes r \in W_n (X)$ by
\[ ([f]\otimes r)\cdot t = [V\times \cplx P^2 
            \stackrel{\pro_1}{\longrightarrow} V
            \stackrel{f}{\longrightarrow} X] \otimes r.  \]
We may then form the tensor product
\[ \overline{W}_* (X) := W_* (X) \otimes_{Z[t]} R, \]
which is $\intg$-graded by
$\deg (x \otimes_{Z[t]} rt^k) = n+4k,$
$x\in W_n (X),$ $r\in Z.$
It is shown in \cite{Ban:TSSKSS} that $\overline{W}_* (-)$ is a
homology theory.
It is naturally a right $R$-module and right multiplication
with $t$ is an isomorphism with inverse given by right multiplication
with $t^{-1}$. This shows that $\overline{W}_* (-)$ is $4$-periodic
and we refer to it as \emph{periodic Witt-bordism at odd primes}. 
The inclusion $Z[t] \subset R$ induces a natural map
\[ i_*: W_* (X) = W_* (X) \otimes_{Z[t]} Z[t] \longrightarrow
   \overline{W}_* (X). \]
As noted above, the homotopy ring of the 
ring spectrum $\K$ is $\pi_* (\K)=\intg [\beta^{\pm 1}],$
where $\beta$ is the complex Bott element in degree $2$.
Away from $2$, $\pi_* (\K [\smlhf])=Z [\beta^{\pm 1}]$.
As $\K [\smlhf]$ is a ring spectrum, the groups 
$(\K [\smlhf])_* (X) = \K^\topo_* (X)[\smlhf]$ 
come with a canonical right $(K[\smlhf])_* (\pt)$-module structure. 
Via the ring homomorphism
\[ R=Z[t,t^{-1}] \longrightarrow Z[\beta, \beta^{-1}],~
     t \mapsto \mbox{$\frac{1}{4}$} \beta^2,  \]
we make $(\K [\smlhf])_* (X)$ into a right $R$-module
$((\K [\smlhf])_* (X), +, \bullet)$.
Thus $t\in R$ acts on $y\in (\K [\smlhf])_* (X)$ by
\[ y \bullet t := y \cdot \mbox{$\frac{1}{4}$} \beta^2.  \]

Let $\MWITT$ denote the (Quinn-type) ring spectrum associated to the
multiplicative ad-theory of Witt spaces, representing Witt bordism
(see Banagl-Laures-McClure \cite{BanLauMcC:LFCISNC}),
and let
\[ \Delta: \MWITT \longrightarrow \KO [\smlhf] \]
denote the second author's ring-spectrum level Siegel-Sullivan orientation,
\cite{Ban:TSSKSS}. It restricts to the Sullivan
orientation $\Delta: \MSPL \to \KO [\smlhf]$ under the canonical
map $\MSPL \to \MWITT$.
By \cite[Prop. 5.7]{Ban:TSSKSS}, the induced natural transformation
$\Delta_*: \Omega^\Witt_* (-) \to \KO_* (-)[\smlhf]$
of homology theories agrees with Siegel's transformation $\mu^\Witt$
as described in \cite{Sie:WSGCTKOP}.
On homotopy groups, $\Delta$ induces the homomorphism
\[ \Delta_*: \Omega^\Witt_{4k} = \MWITT_{4k} \longrightarrow
   \KO [\smlhf]_{4k} = Z \langle a^k \rangle \]
given by
\begin{equation} \label{equ.siegelsullorientoncoeffs} 
\Delta_* [V^{4k} \to \pt] = \sigma (V)\cdot a^k,
\end{equation}
where $\sigma (V)\in \intg$ is the signature of the intersection form
on the intersection homology groups $IH_{2k} (V;\rat)$ of $V$.
The Siegel-Sullivan orientation class of a compact $n$-dimensional Witt space
$(V,\partial V)$ is given by the image
\[ \Delta (V) := \Delta_* [\id_V] \in \KO_n (V,\partial V)[\smlhf]  \]
of the Witt bordism class of the identity on $V$.
Let 
\[ \gamma: W_* (X) \longrightarrow (\K [\smlhf])_* (X) 
                                = \K^\topo_* (X)[\smlhf]  \]
be the natural transformation of homology theories given by
\[ \gamma ([f:V \to X] \otimes r) :=  
          r\cdot c(\Psi^2)^{-1} f_* \Delta (V)
          = r\cdot c(\Psi^2)^{-1} \Delta_* [f],~
             r\in Z.  \]

\begin{lemma} \label{lem.gammaztlinear}
The homomorphism $\gamma: W_* (X) \longrightarrow \K^\topo_* (X)[\smlhf]$
is $Z[t]$-linear.
\end{lemma}
\begin{proof}
Let $[f: V^n \to X]$ be any element of $W_n (X)$.
As $\Delta: \MWITT \to \KO [\smlhf]$ is a morphism of ring spectra,
the diagram
\[ \xymatrix{
\Omega^\Witt_* (X) \otimes \Omega^\Witt_* (\pt)
  \ar[d]_{\Delta_* \otimes \Delta_*} 
    \ar[r]^-{\operatorname{mult}} & \Omega^\Witt_* (X) \ar[d]^{\Delta_*} \\
(\KO [\smlhf])_* (X) \otimes (\KO [\smlhf])_* (\pt)
  \ar[r]_-{\operatorname{mult}} & (\KO [\smlhf])_* (X)
} \]
commutes and hence
\begin{align*}
\Delta_* ([f]\cdot t) 
 &= \Delta_* ([f] \cdot [\cplx P^2 \to \pt])
  = \Delta_* [f] \cdot \Delta_* [\cplx P^2 \to \pt] \\
 &= \Delta_* [f] \cdot \sigma (\cplx P^2) a
  = \Delta_* [f] \cdot a.
\end{align*}
In view of (\ref{equ.cplxofaisbottsquare}) and (\ref{equ.stablesecondadamsona}), 
we find that
\[ c(\Psi^2)^{-1} (a) = c(\mbox{$\frac{1}{4}$} a) = 
                  \mbox{$\frac{1}{4}$} \beta^2. \]
Since $(\Psi^2)^{-1}$ and $c$ are ring maps as well, we conclude
\begin{align*}
\gamma ([f]\cdot t)
 &= c(\Psi^2)^{-1} \Delta_* ([f] \cdot t)
 = c(\Psi^2)^{-1} ((\Delta_* [f]) \cdot a) \\
 &= (c(\Psi^2)^{-1} \Delta_* [f]) \cdot c(\Psi^2)^{-1} (a)
 = \gamma [f] \cdot \mbox{$\frac{1}{4}$} \beta^2 
 = \gamma [f] \bullet t.
\end{align*}
\end{proof}

Recall that in (\ref{equ.thetafromwittbordhlftokanhlf}) 
of Section \ref{subsect:witt-invariance}, we had defined a natural transformation 
\[ \theta: W_n (X) = \Omega^\Witt_n (X)[\smlhf] \cong
\Omega^{\Witt,\infty}_n (X)[\smlhf] \to
	\K^\an_n (X)[\smlhf] \]
of homology theories given by
\[ \theta ([f:W^n \to X]\otimes_\intg r) := 
   r 2^{-\lfloor n/2 \rfloor} f_* [D^\sign_W] = rf_*(\mathrm{sign}_K(W)).  \]    
\begin{lemma} \label{lem.thetaztlinear}
The homomorphism $\theta: W_* (X) \longrightarrow \K^\an_* (X)[\smlhf]$
is $Z[t]$-linear. 
\end{lemma}
\begin{proof}
Let $\pr^2 = \cplx P^2$ denote complex projective space
and let $\const: \pr^2 \to \pt$ be the constant map.
The image of the signature operator class 
$[D^\sign_{\pr^2}] \in \K_4 (\pr^2)[\smlhf]$ under
$\const_*: \K_4 (\pr^2)[\smlhf] \to \K_4 (\pt)[\smlhf]$
is the index (multiplied by the corresponding Bott generator)
\[ \const_* [D^\sign_{\pr^2}] = \sigma (\pr^2)\cdot \beta^2 =
   \beta^2 \in \K_4 (\pt)[\smlhf]. \]
Let $[f: V^n \to X]$ be any element of $W_n (X)$.
By Corollary \ref{cor.signoponproductofwitt},
$[D^\sign_{V\times \pr^2}] = [D^\sign_V] \boxtimes [D^\sign_{\pr^2}].$
(There is no factor of $2$ here, since the dimension of $\pr^2$ is even.)
Consequently,   
\begin{align*}
\theta ([f] \cdot t) 
&= \theta ([f\times \const: V \times \pr^2 \to X\times \pt =X]) \\
&= 2^{-\lfloor \frac{n+4}{2} \rfloor} (f\times \const)_* 
        [D^\sign_{V\times \pr^2}] 
 = 2^{-\lfloor \frac{n}{2} \rfloor} \cdot \mbox{$\frac{1}{4}$} 
   (f\times \const)_* ([D^\sign_V] \boxtimes [D^\sign_{\pr^2}]) \\
&= 2^{-\lfloor \frac{n}{2} \rfloor} f_* [D^\sign_V] \cdot
          \mbox{$\frac{1}{4}$} \const_* [D^\sign_{\pr^2}] 
= 2^{-\lfloor \frac{n}{2} \rfloor} f_* [D^\sign_V] \cdot
          \mbox{$\frac{1}{4}$} \beta^2 \\
&= (\theta [f]) \bullet t.                   
\end{align*}
\end{proof}

The following Proposition generalizes 
\cite[Prop. 5.6]{Ban:TSSKSS} from maps into
$(\KO \smlhf)_* (X)$ to maps into any $Z[t,t^{-1}]$-module $M$;
see also \cite[Prop. 2, p. 597]{BanCapSha:CTSLSS}
of Cappell, Shaneson and the second named author.
The proof follows \cite{Ban:TSSKSS}, but we provide details
for the reader's convenience.
\begin{prop} \label{prop.conntoperiodicfactors}
Let $X$ be a compact PL space and $M$ a right $Z[t,t^{-1}]$-module.\\
1. Given a $Z[t]$-linear map
$\alpha: W_* (X) \to M$
there exists a unique extension of $\alpha$ to a homomorphism 
\[ \overline{\alpha}: \overline{W}_* (X) \longrightarrow
    M  \]
of $Z[t,t^{-1}]$-modules.\\
2. Let $\alpha, \beta: W_* (X) \to M$
be $Z[t]$-linear maps. If 
$\alpha ([g:N \to X]\otimes 1) = \beta ([g]\otimes 1)$ for
every $g$ on smooth manifolds $N$, then 
$\alpha = \beta$ on $W_* (X)$, and
$\overline{\alpha}_* = \overline{\beta}_*$
on $\overline{W}_* (X)$.
\end{prop}
\begin{proof}
We prove statement 1:
We denote the scalar multiplication in $M$ by $\bullet$.
Let 
$A: W_* (X) \times Z[t,t^{-1}] \longrightarrow M$
be the map given by
$A(w,p) := \alpha (w) \bullet p.$
We regard $Z[t,t^{-1}]$ as a left module over $Z[t]$.
For $q\in Z[t],$ the $Z[t]$-linearity of $\alpha$ implies that
\begin{align*}
A(wq,p)
&= \alpha (wq) \bullet p = (\alpha (w) \bullet q) \bullet p \\
&= \alpha (w) \bullet (qp) = A(w,qp).
\end{align*}
Therefore, $A$ is $Z[t]$-bilinear and hence, by the universal
property of the tensor product $\otimes_{Z[t]}$,
induces a well-defined homomorphism 
\[ \overline{\alpha}: W_* (X) \otimes_{Z[t]} Z[t,t^{-1}] 
   \longrightarrow M \]
of abelian groups such that
$\overline{\alpha} (w \otimes_{Z[t]} p) = A(w,p) = \alpha (w) \bullet p$.
Then $\overline{\alpha}_*$ is $Z[t,t^{-1}]$-linear, as for
$p,p' \in Z[t,t^{-1}],$
\begin{align*}
\overline{\alpha} ((w \otimes_{Z[t]} p)\cdot p')
&= \overline{\alpha} ((w \otimes_{Z[t]} (pp'))
  = \alpha (w) \bullet (pp') \\
&= (\alpha (w)\bullet p)\bullet p'
 = \overline{\alpha} (w \otimes_{Z[t]} p)\bullet p',  
\end{align*}
and the diagram 
\begin{equation} \label{equ.overalphaextendsalpha}
\xymatrix@C=50pt@R=40pt{
W_* (X) \ar[rd]^{\alpha} \ar[d]_{i_*} 
   &  \\
\overline{W}_* (X) \ar[r]_{\overline{\alpha}} & M
} 
\end{equation}
commutes. \\
We turn to the proof of uniqueness.
Suppose that $\alpha': \overline{W}_* (X) \to M$
is any $Z[t,t^{-1}]$-linear extension of $\alpha$,
i.e. $\alpha' \circ i_* = \alpha$. Then
\begin{align*}
\alpha' (w \otimes_{Z[t]} p)
&= \alpha' (w \otimes_{Z[t]} (1 \cdot p)) 
 = \alpha' ((w \otimes_{Z[t]} 1) \cdot p) \\
&= \alpha' (w \otimes_{Z[t]} 1) \bullet p 
  = (\alpha' i_* (w)) \bullet p \\
&= \alpha (w) \bullet p 
  = \overline{\alpha} (w \otimes_{Z[t]} p).
\end{align*}
Hence $\overline{\alpha}$ is unique.\\
We prove statement 2:
Since $\alpha$ and $\beta$ are $Z[t]$-linear, they induce
uniquely $Z[t,t^{-1}]$-linear transformations
\[ \overline{\alpha}, \overline{\beta}:
   (W_* (X) \otimes_{Z[t]} Z[t,t^{-1}])_j \longrightarrow
      M,~ j\in \intg, \]
as explained in statement 1.       
For an integer $j$, let $\bar{j}$ denote its residue class in $\intg/_4$.
On the groups 
$C_{\bar{j}} (X) := \bigoplus_{k\in \intg} \Omega^\SO_{j+4k} (X)[\tfrac12]$, 
define an equivalence relation by
\[
 [P^{j+4k} \times N^{4i} \stackrel{\operatorname{proj}}{\longrightarrow} P
      \stackrel{f}{\longrightarrow} X]
 \sim
  \sigma (N) \cdot [P^{j+4k} \stackrel{f}{\longrightarrow} X].
\]
(See also \cite[p. 193]{KreLuc:NC}.)
Let $Q_{\bar{j}} (X,Y) := C_{\bar{j}} (X,Y)/\sim$ denote the corresponding
quotient. One can show that $Q_* (-)$
is a ($\intg/_4$-graded) homology theory on compact PL pairs, see
\cite{Ban:TSSKSS}.
For any $j\in \intg$, a well-defined map
\[  \omega: Q_{\bar{j}} (X) \otimes Z
   \longrightarrow (W_* (X) \otimes_{Z[t]} Z[t, t^{-1}])_j    \]
is given by setting
\[   \omega ([g: N^{j-4k} \to X] \otimes_{\intg} r)
   = [g] \otimes_{Z[t]} rt^k  
     \in W_{j-4k} (X) \otimes_{Z[t]} Z \langle t^k \rangle,~ 
     k \in \intg,~ r \in Z,   \]
where one views the closed oriented smooth manifold $N$
as a Witt space via its canonical PL structure.
This map is an isomorphism on compact PL spaces $X$
as was shown in \cite{Ban:TSSKSS}.
Therefore, Witt bordism classes are representable by smooth manifolds
away from $2$.\\
Given an element
\[ [f: V^{j-4k} \to X] \otimes_{Z[t]} rt^k 
   \in (W_* (X) \otimes_{Z[t]} Z[t,t^{-1}])_j
     = \overline{W}_j (X) \]
$k\in \intg,$ $r\in Z$,     
there exists a (unique) element $q\in Q_{\bar{j}} (X)\otimes Z$
with $\omega (q)= [f] \otimes_{Z[t]} rt^k$,
as $\omega$ is an isomorphism.
Such an element is represented in the quotient $Q_{\bar{j}} (X)\otimes Z$ 
by an element of the form
\[
q = \sum_{i=1}^m [g_i: M_i^{j-4k_i} \to X]\otimes r_i,~
  [g_i] \in \Omega^\SO_{j-4k_i} (X),~ r_i \in Z,~ k_i \in \intg. \]
By the definition of $\omega,$
$\omega ([g_i] \otimes_{\intg} r_i) = [g_i] \otimes_{Z[t]} r_i t^{k_i},$
so that
$[f] \otimes_{Z[t]} rt^k = \sum_{i=1}^m 
   [g_i] \otimes_{Z[t]} r_i t^{k_i}$
and consequently,   
\begin{align*} 
\overline{\alpha}_* ([f] \otimes_{Z[t]} rt^k) 
&= \sum_{i=1}^m 
   \overline{\alpha}_* ([g_i] \otimes_{Z[t]} r_i t^{k_i}) 
   = \sum_{i=1}^m 
     (\alpha_* [g_i]) \cdot r_i a^{k_i} \\
&= \sum_{i=1}^m 
   (\beta_* [g_i]) \cdot r_i a^{k_i} 
   = \sum_{i=1}^m 
   \overline{\beta}_* ([g_i] \otimes_{Z[t]} r_i t^{k_i})        
 = \overline{\beta}_* ([f] \otimes_{Z[t]} rt^k).    
\end{align*}   
This proves that the periodic versions agree on $\overline{W}_* (X)$, 
$\overline{\alpha}_* = \overline{\beta}_*$.
Using the commutativity of (\ref{equ.overalphaextendsalpha})
we deduce
$\alpha_* = \overline{\alpha}_* \circ i_*
    = \overline{\beta}_* \circ i_* = \beta_*.$
\end{proof}

\begin{thm} \label{thm.sullivanandksignoponwitt}
Let $X$ be a closed Witt space and
\[ \Delta (X) \in \KO^\topo_n (X)[\smlhf] \]
its Siegel-Sullivan orientation, \cite{Sie:WSGCTKOP}, \cite{Ban:TSSKSS}.
Under the identification (\ref{equ.identktopandkan}), the element
\[ c(\Psi^2)^{-1} \Delta (X) \in \K^\topo_n (X)[\smlhf] \]
of an $n$-dimensional closed smoothly stratified Witt space $X$ 
corresponds to the signature-operator orientation
\[ \mathrm{sign}_K(X) = 2^{-\lfloor n/2 \rfloor} [D^\sign_X] \in \K^\an_n (X)[\smlhf]. \]
\end{thm}
\begin{proof}
The identification (\ref{equ.identktopandkan})
is $Z[t]$-linear.
Thus, under this identification, we have homomorphisms
\[ \gamma, \theta: W_* (X) \longrightarrow K^\topo_* (X)[\smlhf], \]
which are both $Z[t]$-linear by Lemmas
\ref{lem.gammaztlinear} and \ref{lem.thetaztlinear}.
The two homomorphisms agree on smooth manifolds
according to Theorem \ref{thm.sullivanandksignoponmanifolds}.
Proposition \ref{prop.conntoperiodicfactors} then implies 
that $\gamma = \theta$ on $W_* (X)$ 
(where we have used the standard identification of analytic and topological K-homology, \eqref{equ.identktopandkan}).
\end{proof}


\begin{thebibliography}{ALMP18}



\bibitem[AGR16]{AlbGel:IDOIES}
Pierre Albin and Jesse Gell-Redman.
\newblock The index of {D}irac operators on incomplete edge spaces.
\newblock {\em SIGMA Symmetry Integrability Geom. Methods Appl.}, 12:Paper No.
  089, 45, 2016.

\bibitem[AGR23]{AlbGel:IFFDTOP}
Pierre Albin and Jesse Gell-Redman.
\newblock The index formula for families of {D}irac type operators on
  pseudomanifolds.
\newblock {\em J. Differential Geom.}, 125(2):207--343, 2023.

\bibitem[ALMP12]{AlbLeiMazPia:SPWS}
Pierre Albin, \'{E}ric Leichtnam, Rafe Mazzeo, and Paolo Piazza.
\newblock The signature package on {W}itt spaces.
\newblock {\em Ann. Sci. \'{E}c. Norm. Sup\'{e}r. (4)}, 45(2):241--310, 2012.

\bibitem[ALMP17]{AlbLeiMazPia:NCCS}
Pierre Albin, Eric Leichtnam, Rafe Mazzeo, and Paolo Piazza.
\newblock The {N}ovikov conjecture on {C}heeger spaces.
\newblock {\em J. Noncommut. Geom.}, 11(2):451--506, 2017.

\bibitem[ALMP18]{AlbLeiMazPia:HTCS}
Pierre Albin, Eric Leichtnam, Rafe Mazzeo, and Paolo Piazza.
\newblock Hodge theory on {C}heeger spaces.
\newblock {\em J. Reine Angew. Math.}, 744:29--102, 2018.

\bibitem[AM11]{AlbMel:RSGA}
Pierre Albin and Richard Melrose.
\newblock Resolution of smooth group actions.
\newblock In {\em Spectral theory and geometric analysis}, volume 535 of {\em
  Contemp. Math.}, pages 1--26. Amer. Math. Soc., Providence, RI, 2011.

\bibitem[ABS64]{AtiBotSha:CM}
M.~F. Atiyah, R.~Bott, and A.~Shapiro.
\newblock Clifford modules.
\newblock {\em Topology}, 3(suppl):3--38, 1964.

\bibitem[APS75]{AtiPatSin:SARG}
M.~F. Atiyah, V.~K. Patodi, and I.~M. Singer.
\newblock Spectral asymmetry and {R}iemannian geometry. {I}.
\newblock {\em Math. Proc. Cambridge Philos. Soc.}, 77:43--69, 1975.

\bibitem[AS68]{asindellop3} 
M.~F. Atiyah and I.~M. Singer.
\newblock The Index of Elliptic Operators: III.
\newblock {\em Annals of Math.}, 87, no.~3, 546 -- 604, 1968.

\bibitem[AFT17]{AyaFraTan:LSSS}
David Ayala, John Francis, and Hiro~Lee Tanaka.
\newblock Local structures on stratified spaces.
\newblock {\em Adv. Math.}, 307:903--1028, 2017.

\bibitem[BJ83]{BaaJul:TBKONBCH}
Saad Baaj and Pierre Julg.
\newblock Th\'{e}orie bivariante de {K}asparov et op\'{e}rateurs non born\'{e}s
  dans les {$C\sp{\ast} $}-modules hilbertiens.
\newblock {\em C. R. Acad. Sci. Paris S\'{e}r. I Math.}, 296(21):875--878,
  1983.

\bibitem[Ban20]{Ban:GRTHCCSS}
Markus Banagl.
\newblock Gysin Restriction of Topological and Hodge-Theoretic 
  Characteristic Classes for Singular Spaces.
\newblock New York J. of Math. \textbf{26} (2020), 1273 -- 1337.

\bibitem[Ban24]{Ban:BTLOCSS}
Markus Banagl.
\newblock Bundle Transfer of L-Homology Orientation Classes
  for Singular Spaces.
\newblock Algebraic \& Geometric Topology \textbf{24} (2024), 2579 -- 2618,
  DOI: 10.2140/agt.2024.24.2579

\bibitem[Ban25]{Ban:TSSKSS}
Markus Banagl.
\newblock Transfer and the Spectrum-Level Siegel-Sullivan KO-Orientation for
  Singular Spaces.
\newblock J. of Topology and Analysis, 
  DOI: \texttt{https://doi.org/10.1142/S1793525324500341}

\bibitem[BW24]{BanWra:UTGCCCSS} 
Markus Banagl and Dominik Wrazidlo. 
\newblock The Uniqueness Theorem for Gysin Coherent Characteristic Classes 
  of Singular Spaces.
\newblock J. of the London Math. Soc. \textbf{109} (2024), issue 1, 1 -- 45,
  DOI: 10.1112/jlms.12823

\bibitem[BSW24]{BanSchWra:TGCACCSS} 
Markus Banagl, J\"org Sch\"urmann, and Dominik Wrazidlo.
\newblock Topological Gysin Coherence for Algebraic Characteristic Classes of
   Singular Spaces.
\newblock Preprint, arXiv:2310.15042.

\bibitem[BCS03]{BanCapSha:CTSLSS}
Markus Banagl, Sylvain~E. Cappell, and Julius~L. Shaneson.
\newblock Computing twisted signatures and {$L$}-classes of stratified spaces.
\newblock {\em Math. Ann.}, 326(3):589--623, 2003.

\bibitem[BLM19]{BanLauMcC:LFCISNC}
Markus Banagl, Gerd Laures, and James~E. McClure.
\newblock The {$L$}-homology fundamental class for {IP}-spaces and the
  stratified {N}ovikov conjecture.
\newblock {\em Selecta Math. (N.S.)}, 25(1):Paper No. 7, 104, 2019.

\bibitem[BD82]{BauDou:HIT}
Paul Baum and Ronald~G. Douglas.
\newblock {$K$}\ homology and index theory.
\newblock In {\em Operator algebras and applications, {P}art 1 ({K}ingston,
  {O}nt., 1980)}, volume~38 of {\em Proc. Sympos. Pure Math.}, pages 117--173.
  Amer. Math. Soc., Providence, RI, 1982.

\bibitem[BDT89]{BauDouTay:CRCAK}
Paul Baum, Ronald~G. Douglas, and Michael~E. Taylor.
\newblock Cycles and relative cycles in analytic {$K$}-homology.
\newblock {\em J. Differential Geom.}, 30(3):761--804, 1989.

\bibitem[BGV04]{BerGetVer:HKDO}
Nicole Berline, Ezra Getzler, and Mich\`ele Vergne.
\newblock {\em Heat kernels and {D}irac operators}.
\newblock Grundlehren Text Editions. Springer-Verlag, Berlin, 2004.
\newblock Corrected reprint of the 1992 original.


\bibitem[Bla86]{Bla:KOA}
Bruce Blackadar.
\newblock {\em {$K$}-theory for operator algebras}, volume~5 of {\em
  Mathematical Sciences Research Institute Publications}.
\newblock Springer-Verlag, New York, 1986.

\bibitem[BPR21]{BotPiaRos:PSCSPFGSI}
Boris Botvinnik, Paolo Piazza, and Jonathan Rosenberg.
\newblock Positive scalar curvature on spin pseudomanifolds: the fundamental
  group and secondary invariants.
\newblock {\em SIGMA Symmetry Integrability Geom. Methods Appl.}, 17:Paper No.
  062, 39, 2021.

\bibitem[BPR23]{BotPiaRos:PSCSCSP}
Boris Botvinnik, Paolo Piazza, and Jonathan Rosenberg.
\newblock Positive scalar curvature on simply connected spin pseudomanifolds.
\newblock {\em J. Topol. Anal.}, 15(2):413--443, 2023.

\bibitem[BR24]{BotRos:GPSCSM}
Boris Botvinnik and Jonathan Rosenberg.
\newblock Generalized positive scalar curvature on {${\rm spin}^c$} manifolds.
\newblock {\em Ann. Global Anal. Geom.}, 66(4):Paper No. 18, 20, 2024.

\bibitem[BMvS16]{BraMesSui:GTSTUKP}
Simon Brain, Bram Mesland, and Walter~D. van Suijlekom.
\newblock Gauge theory for spectral triples and the unbounded {K}asparov
  product.
\newblock {\em J. Noncommut. Geom.}, 10(1):135--206, 2016.

\bibitem[BHS91]{BraHecSar:TRPVS}
J.-P. Brasselet, G.~Hector, and M.~Saralegi.
\newblock Th\'{e}or\`eme de de {R}ham pour les vari\'{e}t\'{e}s
  stratifi\'{e}es.
\newblock {\em Ann. Global Anal. Geom.}, 9(3):211--243, 1991.

\bibitem[BHS92]{BraHecSar:LES}
Jean-Paul Brasselet, Gilbert Hector, and Martin Saralegi.
\newblock {${\scr L}^2$}-cohomologie des espaces stratifi\'es.
\newblock {\em Manuscripta Math.}, 76(1):21--32, 1992.

\bibitem[BMRS09]{BroMatRosSza:NCDDBK}
Jacek Brodzki, Varghese Mathai, Jonathan Rosenberg, and Richard~J. Szabo.
\newblock Non-commutative correspondences, duality and {D}-branes in bivariant
  {$K$}-theory.
\newblock {\em Adv. Theor. Math. Phys.}, 13(2):497--552, 2009.


\bibitem[Che83]{Che:SGSRS}
Jeff Cheeger.
\newblock Spectral geometry of singular {R}iemannian spaces.
\newblock {\em J. Differential Geom.}, 18(4):575--657, 1983.

\bibitem[Cho85]{Cho:DOSWCSPSC}
Arthur~Weichung Chou.
\newblock The {D}irac operator on spaces with conical singularities and
  positive scalar curvatures.
\newblock {\em Trans. Amer. Math. Soc.}, 289(1):1--40, 1985.

\bibitem[Con82]{Con:SFOA}
A.~Connes.
\newblock A survey of foliations and operator algebras.
\newblock In {\em Operator algebras and applications, {P}art 1 ({K}ingston,
  {O}nt., 1980)}, volume~38 of {\em Proc. Sympos. Pure Math.}, pages 521--628.
  Amer. Math. Soc., Providence, RI, 1982.

\bibitem[CS81]{ConSka:TLPF}
A.~Connes and G.~Skandalis.
\newblock Th\'{e}or\`eme de l'indice pour les feuilletages.
\newblock {\em C. R. Acad. Sci. Paris S\'{e}r. I Math.}, 292(18):871--876,
  1981.

\bibitem[CS84]{ConSka:LITF}
A.~Connes and G.~Skandalis.
\newblock The longitudinal index theorem for foliations.
\newblock {\em Publ. Res. Inst. Math. Sci.}, 20(6):1139--1183, 1984.

\bibitem[Ebe25]{Ebe:TCHS}
Johannes Ebert.
\newblock Tautological classes and higher signatures, 2025.
\newblock Preprint, arXiv:2403.02755.

\bibitem[Fri20]{Fri:SIH}
Greg Friedman.
\newblock {\em Singular intersection homology}, volume~33 of {\em New
  Mathematical Monographs}.
\newblock Cambridge University Press, Cambridge, 2020.

\bibitem[Gor76]{Gor:GCHSO}
Robert~Mark Goresky.
\newblock {\em Geometric cohomology and homology of stratified objects}.
\newblock ProQuest LLC, Ann Arbor, MI, 1976.
\newblock Thesis (Ph.D.)--Brown University.

\bibitem[Gor78]{Gor:TSO}
R.~Mark Goresky.
\newblock Triangulation of stratified objects.
\newblock {\em Proc. Amer. Math. Soc.}, 72(1):193--200, 1978.

\bibitem[GM80]{GorMac:IHT}
Mark Goresky and Robert MacPherson.
\newblock Intersection homology theory.
\newblock {\em Topology}, 19(2):135--162, 1980.

\bibitem[Hig89]{Hig:KONM}
Nigel Higson.
\newblock K-homology and operators on non-compact manifolds.
\newblock http://web.me.com/ndh2/math/Unpublished.html, 1989.

\bibitem[Hil89]{Hil:FKBPVL}
Michel Hilsum.
\newblock Fonctorialit\'{e} en {$K$}-th\'{e}orie bivariante pour les
  vari\'{e}t\'{e}s lipschitziennes.
\newblock {\em $K$-Theory}, 3(5):401--440, 1989.

\bibitem[Hil14]{Hil:PDBDKP}
Michel Hilsum.
\newblock Un produit d'intersection non born\'e{} dans la {K}-homologie des
  pseudovari\'et\'es.
\newblock {\em Bull. Soc. Math. France}, 142(2):177--192, 2014.

\bibitem[Hit74]{Hit:HS}
Nigel Hitchin.
\newblock Harmonic spinors.
\newblock {\em Advances in Math.}, 14:1--55, 1974.

\bibitem[HLV18]{HarLesVer:DDLTOSS}
Luiz Hartmann, Matthias Lesch, and Boris Vertman.
\newblock On the domain of {D}irac and {L}aplace type operators on stratified
  spaces.
\newblock {\em J. Spectr. Theory}, 8(4):1295--1348, 2018.

\bibitem[HR00]{HigRoe:AK}
Nigel Higson and John Roe.
\newblock {\em Analytic {$K$}-homology}.
\newblock Oxford Mathematical Monographs. Oxford University Press, Oxford,
  2000.
\newblock Oxford Science Publications.

\bibitem[HS87]{HilSka:MKDFFTKDCDC}
Michel Hilsum and Georges Skandalis.
\newblock Morphismes {$K$}-orient\'{e}s d'espaces de feuilles et
  fonctorialit\'{e} en th\'{e}orie de {K}asparov (d'apr\`es une conjecture
  d'{A}. {C}onnes).
\newblock {\em Ann. Sci. \'{E}cole Norm. Sup. (4)}, 20(3):325--390, 1987.

\bibitem[Hus94]{husemoller} 
Dale Husemoller.
\newblock Fibre Bundles. 
\newblock 3rd ed., {\em Graduate Texts in Math.} 20, Springer-Verlag, 1994.

\bibitem[Jak98]{Jak:BDH}
Martin Jakob.
\newblock A bordism-type description of homology.
\newblock {\em Manuscripta Math.}, 96(1):67--80, 1998.

\bibitem[Kas80]{Kas:OKECA}
G.~G. Kasparov.
\newblock The operator {$K$}-functor and extensions of {$C\sp{\ast}
  $}-algebras.
\newblock {\em Izv. Akad. Nauk SSSR Ser. Mat.}, 44(3):571--636, 719, 1980.

\bibitem[KL05]{KreLuc:NC}
Matthias Kreck and Wolfgang L\"uck.
\newblock {\em The {N}ovikov conjecture}, volume~33 of {\em Oberwolfach
  Seminars}.
\newblock Birkh\"auser Verlag, Basel, 2005.
\newblock Geometry and algebra.

\bibitem[KR23]{KotRoc:PMWFC}
Chris Kottke and Fr\'{e}d\'{e}ric Rochon.
\newblock Products of manifolds with fibered corners.
\newblock {\em Ann. Global Anal. Geom.}, 64(2):Paper No. 9, 61, 2023.

\bibitem[Kuc97]{Kuc:KUM}
Dan Kucerovsky.
\newblock The {$KK$}-product of unbounded modules.
\newblock {\em $K$-Theory}, 11(1):17--34, 1997.

\bibitem[KvS18]{KaaSui:RSFDO}
Jens Kaad and Walter~D. van Suijlekom.
\newblock Riemannian submersions and factorization of {D}irac operators.
\newblock {\em J. Noncommut. Geom.}, 12(3):1133--1159, 2018.

\bibitem[KvS20]{KaaSui:FDOAFSM}
Jens Kaad and Walter~D. van Suijlekom.
\newblock Factorization of {D}irac operators on almost-regular fibrations of {$
  {\rm spin}^c$} manifolds.
\newblock {\em Doc. Math.}, 25:2049--2084, 2020.

\bibitem[Lan95]{Lan:HC}
E.~C. Lance.
\newblock {\em Hilbert {$C^*$}-modules}, volume 210 of {\em London Mathematical
  Society Lecture Note Series}.
\newblock Cambridge University Press, Cambridge, 1995.
\newblock A toolkit for operator algebraists.

\bibitem[LM14]{LauMcC:MPQS}
Gerd Laures and James~E. McClure.
\newblock Multiplicative properties of {Q}uinn spectra.
\newblock {\em Forum Math.}, 26(4):1117--1185, 2014.

\bibitem[LNS23]{LanNikSch:LC}
Markus Land, Thomas Nikolaus, and Marco Schlichting.
\newblock L-theory of {$C^\ast$}-algebras.
\newblock {\em Proc. Lond. Math. Soc. (3)}, 127(5):1451--1506, 2023.

\bibitem[Lur]{Lur:HA}
Jacob Lurie.
\newblock Higher algebra.
\newblock http://www.math.harvard.edu/~lurie/.

\bibitem[Mat12]{Mat:NTS}
John Mather.
\newblock Notes on topological stability.
\newblock {\em Bull. Amer. Math. Soc. (N.S.)}, 49(4):475--506, 2012.

\bibitem[Maz91]{Maz:ETDEO}
Rafe Mazzeo.
\newblock Elliptic theory of differential edge operators. {I}.
\newblock {\em Comm. Partial Differential Equations}, 16(10):1615--1664, 1991.

\bibitem[Mel93]{Mel:AIT}
Richard~B. Melrose.
\newblock {\em The {A}tiyah-{P}atodi-{S}inger index theorem}, volume~4 of {\em
  Research Notes in Mathematics}.
\newblock A K Peters, Ltd., Wellesley, MA, 1993.

\bibitem[MP92]{MelPia:AKMWC}
Richard~B. Melrose and Paolo Piazza.
\newblock Analytic {$K$}-theory on manifolds with corners.
\newblock {\em Adv. Math.}, 92(1):1--26, 1992.

\bibitem[MW97]{moscoviciwu}
Henri Moscovici and Fangbing Wu.
\newblock Straight {C}hern character for {W}itt spaces.
\newblock In {\em Cyclic cohomology and noncommutative geometry ({W}aterloo,
  {ON}, 1995)}, volume~17 of {\em Fields Inst. Commun.}, pages 103--113. Amer.
  Math. Soc., Providence, RI, 1997.

\bibitem[NV23]{NocVol:WSCS}
Guglielmo Nocera and Marco Volpe.
\newblock Whitney stratifications are conically smooth.
\newblock {\em Selecta Math. (N.S.)}, 29(5):Paper No. 68, 20, 2023.

\bibitem[PRW95]{PedRoeWei:HIBCASMOOC}
Erik~K. Pedersen, John Roe, and Shmuel Weinberger.
\newblock On the homotopy invariance of the boundedly controlled analytic
  signature of a manifold over an open cone.
\newblock In {\em Novikov conjectures, index theorems and rigidity, {V}ol.\ 2
  ({O}berwolfach, 1993)}, volume 227 of {\em London Math. Soc. Lecture Note
  Ser.}, pages 285--300. Cambridge Univ. Press, Cambridge, 1995.

\bibitem[Pfl01]{Pfl:AGSSS}
Markus~J. Pflaum.
\newblock {\em Analytic and geometric study of stratified spaces}, volume 1768
  of {\em Lecture Notes in Mathematics}.
\newblock Springer-Verlag, Berlin, 2001.

\bibitem[RW24]{randalfamilysignature} 
Oscar Randal-Williams.
\newblock The Family Signature Theorem.
\newblock {\em Proc. of the Royal Society of Edinburgh (Ranicki memorial issue)} 
 154 (6), 2024 -- 2067, 2024. 

\bibitem[Ran]{Ran:ATSULNP1}
Andrew Ranicki.
\newblock The algebraic theory of surgery.
\newblock (unpublished lecture notes, {P}rinceton, 1978),
  {\url{https://www.maths.ed.ac.uk/~v1ranick/papers/ats.pdf}}.

\bibitem[Ran79]{Ran:TSO}
Andrew Ranicki.
\newblock The total surgery obstruction.
\newblock In {\em Algebraic topology, {A}arhus 1978 ({P}roc. {S}ympos., {U}niv.
  {A}arhus, {A}arhus, 1978)}, volume 763 of {\em Lecture Notes in Math.}, pages
  275--316. Springer, Berlin, 1979.

\bibitem[Ran80]{Ran:ATSIAT}
Andrew Ranicki.
\newblock The algebraic theory of surgery. {II}. {A}pplications to topology.
\newblock {\em Proc. London Math. Soc. (3)}, 40(2):193--283, 1980.

\bibitem[Rud98]{Rud:TSOC}
Yuli~B. Rudyak.
\newblock {\em On {T}hom spectra, orientability, and cobordism}.
\newblock Springer Monographs in Mathematics. Springer-Verlag, Berlin, 1998.
\newblock With a foreword by Haynes Miller.

\bibitem[RW06]{RosWei:SO}
Jonathan Rosenberg and Shmuel Weinberger.
\newblock The signature operator at 2.
\newblock {\em Topology}, 45(1):47--63, 2006.

\bibitem[Sie72]{Sie:DHSS}
L.~C. Siebenmann.
\newblock Deformation of homeomorphisms on stratified sets. {I}, {II}.
\newblock {\em Comment. Math. Helv.}, 47:123--136; ibid. 47 (1972), 137--163,
  1972.

\bibitem[Sie83]{Sie:WSGCTKOP}
P.~H. Siegel.
\newblock Witt spaces: a geometric cycle theory for {$K{\rm O}$}-homology at
  odd primes.
\newblock {\em Amer. J. Math.}, 105(5):1067--1105, 1983.

\bibitem[Sto68]{stong} 
Robert E. Stong.
\newblock Notes on Cobordism Theory.
\newblock Princeton University Press, 1968.

\bibitem[Sul71]{Sul:GTP}
Dennis Sullivan.
\newblock {\em Geometric topology. {P}art {I}}.
\newblock Massachusetts Institute of Technology, Cambridge, MA, 1971.
\newblock Localization, periodicity, and Galois symmetry, Revised version.

\bibitem[Teu81]{Teu:APSB}
Michael Teufel.
\newblock Abstract prestratified sets are {$(b)$}-regular.
\newblock {\em J. Differential Geometry}, 16(3):529--536, 1981.

\bibitem[Tho69]{Tho:EMS}
R.~Thom.
\newblock Ensembles et morphismes stratifi\'{e}s.
\newblock {\em Bull. Amer. Math. Soc.}, 75:240--284, 1969.

\bibitem[vdD22]{Dun:KPSOM}
Koen van~den Dungen.
\newblock The {K}asparov product on submersions of open manifolds.
\newblock {\em J. Topol. Anal.}, 14(1):147--181, 2022.

\bibitem[vdDM20]{DunMes:HEUK}
Koen van~den Dungen and Bram Mesland.
\newblock Homotopy equivalence in unbounded {$KK$}-theory.
\newblock {\em Ann. K-Theory}, 5(3):501--537, 2020.

\bibitem[Ver84]{Ver:SMT}
Andrei Verona.
\newblock {\em Stratified mappings---structure and triangulability}, volume
  1102 of {\em Lecture Notes in Mathematics}.
\newblock Springer-Verlag, Berlin, 1984.

\bibitem[vSV22]{SuiVer:IUKPESIES}
Walter~D. van Suijlekom and Luuk~S. Verhoeven.
\newblock Immersions and the unbounded {K}asparov product: embedding spheres
  into {E}uclidean space.
\newblock {\em J. Noncommut. Geom.}, 16(2):489--511, 2022.

\bibitem[Wah10]{Wah:PFAICHS}
Charlotte Wahl.
\newblock Product formula for {A}tiyah-{P}atodi-{S}inger index classes and
  higher signatures.
\newblock {\em J. K-Theory}, 6(2):285--337, 2010.

\bibitem[Zen19]{Zen:BTPLSWP}
S.~Zentarra.
\newblock Bordism theories of piecewise-linear and smooth Witt pseudomanfiolds.
\newblock master thesis, Universit{\"a}t Heidelberg, 2019.

\end{thebibliography}
\end{document}